\documentclass[12pt, reqno]{amsart}
\usepackage{amssymb}
\usepackage{eucal}
\usepackage{amsmath}
\usepackage{amscd}
\usepackage[dvips]{color}
\usepackage{mathtools}
\usepackage{multicol}
\usepackage[all]{xy}           
\usepackage{graphicx}
\usepackage{color}
\usepackage{colordvi}
\usepackage{xspace}
\usepackage{ulem}
\usepackage{cancel}
\usepackage{tikz}
\usepackage{longtable}
\usepackage{multirow}
\usepackage{ifpdf}
\ifpdf
 \usepackage[colorlinks,final,backref=page,hyperindex]{hyperref}
\else
 \usepackage[colorlinks,final,backref=page,hyperindex,hypertex]{hyperref}
\fi

\makeatletter
\@namedef{subjclassname@2020}{\textup{2020} Mathematics Subject Classification}
\makeatother

\begin{document}
\newcommand {\emptycomment}[1]{} 

\newcommand{\tabincell}[2]{\begin{tabular}{@{}#1@{}}#2\end{tabular}}

\newcommand{\nc}{\newcommand}
\newcommand{\delete}[1]{}

\nc{\mlabel}[1]{\label{#1}}  
\nc{\mcite}[1]{\cite{#1}}  
\nc{\mref}[1]{\ref{#1}}  
\nc{\meqref}[1]{~\eqref{#1}} 
\nc{\mbibitem}[1]{\bibitem{#1}} 

\delete{
\nc{\mlabel}[1]{\label{#1}  
{\hfill \hspace{1cm}{\bf{{\ }\hfill(#1)}}}}
\nc{\mcite}[1]{\cite{#1}{{\bf{{\ }(#1)}}}}  
\nc{\mref}[1]{\ref{#1}{{\bf{{\ }(#1)}}}}  
\nc{\meqref}[1]{~\eqref{#1}{{\bf{{\ }(#1)}}}} 
\nc{\mbibitem}[1]{\bibitem[\bf #1]{#1}} 
}

\newtheorem{thm}{Theorem}[section]
\newtheorem{lem}[thm]{Lemma}
\newtheorem{cor}[thm]{Corollary}
\newtheorem{pro}[thm]{Proposition}
\newtheorem{conj}[thm]{Conjecture}
\theoremstyle{definition}
\newtheorem{defi}[thm]{Definition}
\newtheorem{ex}[thm]{Example}
\newtheorem{rmk}[thm]{Remark}
\newtheorem{pdef}[thm]{Proposition-Definition}
\newtheorem{condition}[thm]{Condition}

\renewcommand{\labelenumi}{{\rm(\alph{enumi})}}
\renewcommand{\theenumi}{\alph{enumi}}
\renewcommand{\labelenumii}{{\rm(\roman{enumii})}}
\renewcommand{\theenumii}{\roman{enumii}}

\nc{\tred}[1]{\textcolor{red}{#1}}
\nc{\tblue}[1]{\textcolor{blue}{#1}}
\nc{\tgreen}[1]{\textcolor{green}{#1}}
\nc{\tpurple}[1]{\textcolor{purple}{#1}}
\nc{\btred}[1]{\textcolor{red}{\bf #1}}
\nc{\btblue}[1]{\textcolor{blue}{\bf #1}}
\nc{\btgreen}[1]{\textcolor{green}{\bf #1}}
\nc{\btpurple}[1]{\textcolor{purple}{\bf #1}}


\newcommand{\End}{\text{End}}

\nc{\calb}{\mathcal{B}}
\nc{\call}{\mathcal{L}}
\nc{\calo}{\mathcal{O}}
\nc{\frakg}{\mathfrak{g}}
\nc{\frakh}{\mathfrak{h}}
\nc{\ad}{\mathrm{ad}}

\nc{\ccred}[1]{\tred{\textcircled{#1}}}

\nc{\daop}{dual a-$\mathcal{O}$-operator\xspace}

\nc{\aop}{a-$\mathcal{O}$-operator\xspace} \nc{\sapp}{special
apre-perm algebra\xspace} \nc{\sapps}{special apre-perm
algebras\xspace} \nc{\sappb}{special apre-perm bialgebra\xspace}

\nc{\sappbs}{special apre-perm bialgebras\xspace}
\nc{\app}{apre-perm algebra\xspace} \nc{\apps}{apre-perm
algebras\xspace} \nc{\da}{APP algebra\xspace} \nc{\das}{APP
algebra\xspace} \nc{\sa}{SAPP algebra\xspace} \nc{\sas}{SAPP
algebras\xspace}
\nc{\sappsubs}{special apre-perm subalgebras\xspace}
\nc{\appsubs}{apre-perm subalgebras\xspace}
\nc{\appb}{apre-perm bialgebra\xspace}
\nc{\appbs}{apre-perm bialgebras\xspace}

\title[\delete{Averaging commutative and cocommutative ASI and SAPP bialgebras}]
{A further study on  averaging commutative and cocommutative infinitesimal bialgebras and special apre-perm bialgebras}

\author{Quan Zhao}
\address{Chern Institute of Mathematics \& LPMC, Nankai University, Tianjin 300071, China}
\email{zhaoquan@mail.nankai.edu.cn}

\author{Guilai Liu}
\address{Chern Institute of Mathematics \& LPMC, Nankai University, Tianjin 300071, China}
\email{liugl@mail.nankai.edu.cn}

\delete{
\author{Chengming Bai}
\address{Chern Institute of Mathematics \& LPMC, Nankai University, Tianjin 300071, China}
\email{baicm@nankai.edu.cn}

\author{Li Guo}
\address{Department of Mathematics and Computer Science, Rutgers University, Newark, NJ 07102, USA}
\email{liguo@rutgers.edu}
}

\date{\today}%

\begin{abstract}
	In order to generalize the fact that an averaging commutative algebra gives rise to a perm algebra to the bialgebra level, the notion of a special apre-perm algebra was introduced as a new splitting of perm algebras, and it has been shown that an averaging commutative and cocommutative infinitesimal bialgebra gives rise to a special apre-perm bialgebra.
	In this paper, we give a further study on averaging commutative and cocommutative infinitesimal bialgebras and special apre-perm bialgebras.
A solution of the averaging associative Yang-Baxter equation whose symmetric part is invariant gives rise to an averaging commutative and cocommutative infinitesimal bialgebra that is called quasi-triangular, and such solutions can be equivalently characterized as  $\mathcal{O}$-operators of admissible averaging commutative algebras  with weights.
Moreover assuming the symmetric parts of such solutions to be zero or nondegenerate, we obtain typical subclasses of quasi-triangular averaging commutative and cocommutative infinitesimal bialgebras, namely the triangular and factorizable ones respectively. Both of them are shown to closely relate to symmetric averaging Rota-Baxter Frobenius commutative algebras.
There is a parallel procedure developed for \sappbs.
In particular, the fact that an averaging commutative and cocommutative infinitesimal bialgebra gives rise to a special apre-perm bialgebra is still available when these bialgebras are limited to the quasi-triangular cases.
\end{abstract}

\subjclass[2020]{
    17A36,  
    17A40,  
    17B10, 
    17B38, 
    17D25,  
    18M70.  
}

\keywords{Averaging commutative and cocommutative infinitesimal bialgebra, \sappb, Yang-Baxter equation, Rota-Baxter operator, $\mathcal{O}$-operator}
\maketitle

\tableofcontents
\allowdisplaybreaks

\section{Introduction}

\subsection{The previous study on averaging commutative and cocommutative infinitesimal bialgebras and special apre-perm bialgebras}\

Let $\mathcal{P}$ be a binary operad and $(A,\star_{A})$ be a $\mathcal{P}$-algebra. If there is a linear map $P:A\rightarrow A$ satisfying  
\begin{equation}\label{eq:Ao}
    P(x)\star_{A}P(y)=P\big(P(x)\star_{A}y\big)=P\big(x\star_{A} P(y)\big),\;\forall x,y\in A,
\end{equation}
then $P$ is called an {\bf averaging operator} of $(A,\star_{A})$ and  $(A,\star_{A},P)$ is called an {\bf averaging $\mathcal{P}$-algebra}. 
The notion of an averaging operator was implicitly studied in the famous paper of O. Reynolds
~\mcite{Re} in connection with the theory of turbulence and explicitly defined by Kolmogoroff and Kamp\'e de
F\'eriet in probability~\mcite{Kam}. It later attracted the attentions of other
well-known mathematicians including G. Birkhoff, Miller and Rota \cite{Bi,Mil,R2} with motivation
from quantum physics and combinatorics. Recently it has found diverse applications in
many areas such as combinatorics, number
theory, operads, cohomology and deformation
theory \mcite{DS,GuK,GZ,PBGN2,PG,WZ,ZG}.

 On the other hand, a {\bf perm algebra} \mcite{Chap2001} is a vector
 space $A$ together with a binary operation  $\circ _{A}:A\otimes
 A\rightarrow A$ satisfying
 \begin{equation}\mlabel{eq:perm}
 	x\circ _{A}(y\circ _{A}z)=(x\circ _{A}y)\circ _{A}z=(y\circ
 	_{A}x)\circ _{A}z,\;\;\forall x,y,z\in A.
 \end{equation}
 Perm algebras play an important role in algebraic operad theory,
 since their operad is the Koszul dual to the operad of pre-Lie
 algebras \mcite{LV}, as well as the duplicator of the operad of
 commutative algebras \mcite{GuK,PBGN}.
 The latter fact also leads to the following construction of perm algebras from averaging commutative algebras.
 \begin{pro}\mlabel{ex:comm aver}\mcite{Agu2000*}
 	Let $P$ be an averaging operator on a commutative algebra
 	$(A,\cdot_{A})$. Then there is a perm algebra $(A,\circ_{A})$
 	given by
 	\begin{equation}\mlabel{eq:perm from aver op}
 		x\circ_{A} y=P(x)\cdot_{A} y, \;\;\forall x,y\in A.
 	\end{equation}
 \end{pro}
  
  A  bialgebra structure is a vector space equipped with both an
  algebra structure and a coalgebra structure satisfying certain
  compatible conditions. Some known examples of such structures
  are Lie bialgebras  \cite{Cha,Dri} 
  that are closely related to Poisson-Lie groups and
  play an important role in the infinitesimalization of quantum
  groups, 
   and antisymmetric
  infinitesimal bialgebras  \cite{Agu2000, Agu2001,
  	Agu2004,Bai2010}
  which
  render symmetric Frobenius algebras and thus find applications
  in 2d topological and string theory.
  Both of them have a common property that they can be 
  equivalently characterized as Manin triples associated to
  nondegenerate bilinear forms on the corresponding algebras
  satisfying certain conditions. Explicitly, a Lie bialgebra is
  equivalent to a Manin triple of Lie algebras (associated to the
  nondegenerate symmetric invariant bilinear form) and an ASI
  bialgebra is equivalent to a double construction of Frobenius
  algebra, or equivalently, a Manin triple of associative algebras associated to
  the nondegenerate symmetric invariant bilinear form.
  Such an approach of studying bialgebra structures from the viewpoint of Manin triples can be successively applied to many other types of
  algebras such as pre-Lie algebras \mcite{Bai2008} and Poisson
  algebras \mcite{NB}.
  Moreover,  
   the above bialgebra theory of algebra structures has recently been successfully extended to the
  context of operated algebras through the Manin triple approach, such as Rota-Baxter associative
  algebras and Lie algebras \mcite{BGLM,BGM}, differential
  associative algebras \mcite{LLB} and endo Lie algebras
  \mcite{BGS}.
   
 In \cite{Bai2024}, the authors considered the natural idea on generalizing the fact in \cite{Agu2000*} that an averaging commutative algebra gives rise to a perm algebra to the level of bialgebras.
An admissible condition on an averaging commutative algebra was introduced from the
representations of averaging commutative  algebras on
the dual spaces, and it was shown that a symmetric Frobenius commutative algebra with an averaging operator gives rise to an admissible averaging commutative algebra.
Then the notion of an averaging commutative and cocommutative infinitesimal bialgebra was introduced, which was shown to be equivalent to a double
construction of averaging Frobenius commutative algebras.

  Moreover, the authors showed that a nondegenerate symmetric invariant bilinear form on an averaging commutative algebra is left-invariant on the induced perm algebra. The notion of \sapps was introduced as a typical subclass of apre-perm algebras that give a new splitting of perm algebras. There is a one-to-one correspondence between quadratic \sapps and perm algebras with nondegenerate symmetric left-invariant bilinear forms. It was also shown that an admissible averaging commutative  
   algebra induces a \sapp.
   
   Furthermore,  the notions of a Manin triple of perm algebras associated to the nondegenerate
   symmetric left-invariant bilinear form, a Manin triple
   of \sapps and a \sappb were introduced, and the equivalences among these notions were
   established. Generalizing the fact that a Frobenius commutative algebra with an averaging operator renders a quadratic \sapp to the Manin triple level, it was shown that a double construction of averaging
   Frobenius commutative algebra induces a Manin triple of \sapps. Equivalently from the bialgebra
   viewpoint, an averaging commutative and cocommutative infinitesimal
   bialgebra renders a \sappb.

  \subsection{Our further study on averaging commutative and cocommutative infinitesimal bialgebras and special apre-perm bialgebras}\
   
   This paper gives a further study on averaging commutative and cocommutative infinitesimal bialgebras and special apre-perm bialgebras, especially on the construction theory of these bialgebras.
   
   As the first part of the paper, we study quasi-triangular averaging commutative and cocommutative infinitesimal bialgebras. We introduce the notion of the averaging associative Yang-Baxter equation (AAYBE), and show that a solution of the AAYBE whose symmetric part is invariant gives rise to an averaging commutative and cocommutative infinitesimal bialgebra which is called quasi-triangular. Moreover, we introduce the notion of $\mathcal{O}$-operators of admissible averaging commutative algebras, which serve as the interpretation of the AAYBE in terms of operator forms. 
   
   Moreover assuming symmetric parts of solutions of the AAYBE to be zero or bijective, we obtain  typical subclasses of quasi-triangular averaging commutative and cocommutative infinitesimal bialgebras, namely triangular ones and factorizable ones. 
   On the one hand, we introduce the notion of admissible averaging Zinbiel algebras, which supply $\mathcal{O}$-operators of the sub-adjacent admissible averaging commutative algebras.
   Therefore we obtain skew-symmetric solutions of the AAYBE in the semi-direct product admissible averaging commutative algebras and hence triangular averaging commutative and cocommutative infinitesimal bialgebras.
   On the other hand, we show that a factorizable averaging commutative and cocommutative infinitesimal bialgebra leads to a factorization of the underlying averaging commutative algebras. 
   We also show that there is a canonical factorizable
   averaging commutative and cocommutative infinitesimal bialgebra structure on the double space of an arbitrary averaging commutative and cocommutative infinitesimal bialgebra.
   
   Furthermore, we introduce the notion of a
     symmetric averaging Rota-Baxter Frobenius commutative algebra, which is a
     symmetric Rota-Baxter Frobenius commutative algebra \cite{SW} simultaneously equipped with an averaging operator that commutes with the Rota-Baxter operator.
      We show that such structures give rise to triangular averaging commutative and cocommutative infinitesimal bialgebras when the weight is $0$, and are in one-to-one correspondence with  factorizable averaging commutative and cocommutative infinitesimal bialgebras when the weight is $-1$.

  As the second part of the paper, we study quasi-triangular \sappbs. The main approach is similar to that of quasi-triangular averaging commutative and cocommutative infinitesimal bialgebras. We summarize the invariant condition of a $2$-tensor on a \sapp from a quadratic \sapp. A solution of the special apre-perm Yang-Baxter equation (SAPP-YBE) whose symmetric part is invariant leads to to a \sappb. The \sappb obtained in this way is called   quasi-triangular. In particular, we show that a solution of the AAYBE (whose symmetric part is invariant) in an averaging commutative algebra is also a solution of the SAPP-YBE (whose symmetric part is invariant) in the induced \sapp. Therefore a quasi-triangular averaging commutative and cocommutative infinitesimal bialgebra gives rise to a quasi-triangular \sappb.
  Moreover, solutions of the SAPP-YBE can also be characterized as $\mathcal{O}$-operators of \sapps. 
  
  There are also typical subclasses of quasi-triangular \sappbs, including triangular \sappbs where the solutions of the SAPP-YBE are skew-symmetric, and factorizable \sappbs where the symmetric parts of such solutions are bijective. 
    On the one hand, we introduce the notion of pre-special apre-perm algebras, which render triangular \sappbs in a  way similar to obtaining triangular averaging commutative and cocommutative infinitesimal bialgebras from admissible averaging Zinbiel algebras as aforementioned.
   In particular, admissible averaging Zinbiel algebras give rise to pre-special apre-perm algebras. Therefore we have the following commutative diagram.
    {\small
    	\begin{equation*}
    		\begin{split}
    			\xymatrix{
    				\txt{admissible averaging Zinbiel\\ algebras}
    				\ar@{=>}[r]^-{
    				}
    				\ar@{=>}[d]^-{
    				}
    				& \txt{triangular averaging commutative  and\\ cocommutative  
    					infinitesimal bialgebras
    				}
    				\ar@{=>}[d]^-{
    				}
    				\\  
    				\txt{pre-special apre-perm algebras}
    				\ar@{=>}[r]^-{
    				}
    				& \txt{triangular special apre-perm bialgebras}
    			}
    		\end{split}
    	\end{equation*}
    }
    
    On the other hand, a factorizable special apre-perm bialgebra also leads to a factorization of the underlying special apre-perm algebras, and arises 
    on the double space of an arbitrary special apre-perm bialgebra.
    	Furthermore,  we introduce the notion of a quadratic Rota-Baxter \sapp. We show that a quadratic Rota-Baxter \sapp of weight $0$ gives rise to a triangular \sappb, and there is a one-to-one correspondence between quadratic Rota-Baxter \sapps of weight $-1$ and factorizable \sappbs. We also show that a symmetric averaging Rota-Baxter Frobenius commutative algebra gives rise to a quadratic Rota-Baxter \sapp with the same weight. 
      Therefore we have the following commutative diagram.
  
  {\footnotesize
 	\begin{equation*}
 		\begin{split}
 			\xymatrix{
 				\txt{triangular averaging\\ commutative  and\\ cocommutative  
 					infinitesimal\\ bialgebras}
 				\ar@{<=}[r]^-{\lambda=0 
 				}
 				\ar@{=>}[d]^-{
 				}
 				& \txt{symmetric averaging\\ Rota-Baxter Frobenius\\ commutative algebras}
 				\ar@<.4ex>@{<=>}[r]^-{\lambda=-1 
 				}
 				\ar@{=>}[d]^-{
 				}
 				& \txt{factorizable averaging\\ commutative  and\\ cocommutative  
 					infinitesimal\\ bialgebras
 				}
 				\ar@{=>}[d]^-{ 
 				}
 				\\  
 				\txt{triangular special \\  apre-perm bialgebras}
 				\ar@{<=}[r]^-{\lambda=0
 				}
 				&\txt{quadratic Rota-Baxter\\ special apre-perm algebras}	\ar@{<=>}[r]^-{\lambda=-1 
 				} & \txt{factorizable  special \\ apre-perm bialgebras}
 			}
 		\end{split}
 	\end{equation*}
 }

  In a summary, this paper strengthens the relationship from averaging commutative and cocommutative bialgebras to \sappbs by studying their construction theory in several aspects, such as quasi-triangular, triangular and factorizable cases, Yang-Baxter equations, $\mathcal{O}$-operators and quadratic Rota-Baxter algebras.

\subsection{Organization and convention of the paper} This paper is organized as follows.

In Section \ref{sec:2}, we study quasi-triangular averaging commutative and cocommutative infinitesimal bialgebras which are constructed from solutions of the AAYBE whose symmetric parts are invariant. Such solutions are translated into operator forms, namely $\mathcal{O}$-operators of admissible averaging commutative  algebras. We also study triangular and factorizable averaging commutative and cocommutative infinitesimal bialgebras as the typical subclasses.

In Section \ref{sec:5}, we study quasi-triangular \sappbs which are constructed from solutions of the SAPP-YBE whose symmetric parts are invariant. Such solutions can also be translated into $\mathcal{O}$-operators of \sapps. In particular, we show that quasi-triangular \sappbs naturally arise from quasi-triangular averaging commutative and cocommutative infinitesimal bialgebras. We also study typical subclasses of quasi-triangular \sappbs, namely triangular and factorizable \sappbs.

Throughout this paper, unless otherwise specified, all the vector spaces and algebras are finite-dimensional over an algebraically
closed field $\mathbb {K}$ of characteristic zero, although many results and notions remain valid in the infinite-dimensional case.
By a commutative algebra, we mean a commutative associative algebra not necessarily having a unit. For vector spaces $A$ and $V$, we fix the following notations.
\begin{enumerate}
	\item Define linear maps $\tau:A\otimes A\rightarrow A\otimes A$ and $\xi:A\otimes A\otimes A\rightarrow A\otimes A\otimes A$ by 
	\begin{eqnarray}
		\tau(x\otimes y)=y\otimes x,\;
		\xi(x\otimes y\otimes z)=y\otimes z\otimes x,\;\forall x,y,z\in A .
	\end{eqnarray}
	\item Let $\circ_{ A }: A \otimes A \rightarrow A $ be a multiplication on $A$.
	Define linear maps $\mathcal{L}_{\circ_{ A }},\mathcal{R}_{\circ_{ A }}: A \rightarrow \mathrm{End}_{\mathbb K}(A)$ by 
	\begin{eqnarray}
		\mathcal{L}_{\circ_{ A }}(x)y=x\circ_{ A }y=\mathcal{R}_{\circ_{ A }}(y)x,\;\forall x,y\in A .
	\end{eqnarray}
	\item Let  $P: A \rightarrow A $ be a linear map. Define a linear map $P^{*}: A ^{*}\rightarrow A ^{*}$ by
	\begin{eqnarray}
		\langle P^{*}(a^{*}), x\rangle=\langle a^{*},P(x)\rangle,\;\forall x\in A , a^{*}\in A ^{*},
	\end{eqnarray}
	where $\langle-,-\rangle$ is the ordinary pair between $ A $ and the dual space $ A ^{*}$.
	\item Let $\mu: A \rightarrow\mathrm{End}_{\mathbb K}(V)$ be a linear map. Define a linear map $\mu^{*}: A \rightarrow\mathrm{End}_{\mathbb K}(V^{*})$ by
	\begin{eqnarray}
		\langle \mu^{*}(x)u^{*},v\rangle=
		\langle \big(\mu(x)\big)^{*}u^{*},v\rangle=
		\langle u^{*},\mu(x)v\rangle,\;\forall x\in A ,u^{*}\in V^{*},v\in V.
	\end{eqnarray}
\item  Let  $P :A\rightarrow A$ and $\alpha:V \rightarrow V $ be linear maps.
Define a linear map $P +\alpha: A\oplus V \rightarrow A\oplus V$ by 
\begin{eqnarray}
(P+\alpha)(x+u)=P(x)+\alpha(u),\;\forall x\in A, u\in V.
\end{eqnarray}
	\item An element $r\in V \otimes A$ is identified as a linear map
	$r^{\sharp}: V ^{*}\rightarrow A$ by 
	\begin{eqnarray}
		\langle r^{\sharp}(u^{*}),a^{*}\rangle=\langle r,u^{*}\otimes a^{*}\rangle,\;\forall u^{*}\in V ^{*}, a^{*}\in A^{*}.
	\end{eqnarray}
Conversely, a linear map $T:V\rightarrow A$  is identified as an element $T_{\sharp}\in V^{*}\otimes A$ by
\begin{equation}
	\langle T_{\sharp}, v\otimes a^{*}\rangle=\langle T(v),a^{*}\rangle,\;\forall v\in V,a^{*}\in A^{*}.
\end{equation}
	\item A bilinear form $\mathcal{B}$ on $ A $ is identified as a linear map $\mathcal{B}^{\natural}: A \rightarrow A ^{*}$ by 
\begin{eqnarray}
	\langle \mathcal{B}^{\natural}(x),y\rangle=\mathcal{B}(x,y),\;\forall x,y\in A .
\end{eqnarray}
Moreover, $\mathcal B$ is
nondegenerate if and only if $\mathcal B^\natural$ is invertible. In this case, we define a 2-tensor $\phi_{\mathcal{B}}\in A\otimes A$ by
\begin{equation}\label{eq:2-tensor}
	\langle \phi_{\mathcal{B}},a^{*}\otimes b^{*}\rangle=\langle {\mathcal{B}
		^{\natural}}^{-1}(a^{*}),b^{*}\rangle,\;\forall a^{*},b^{*}\in A^{*}.
\end{equation}
\item Let $\mathcal{B}$ be a nondegenerate symmetric bilinear form on $A$ and $P:A\rightarrow A$ be a linear map. The adjoint map $\hat{P}:A\rightarrow A$ of $P$ with respect to $\mathcal{B}$ is given by
\begin{eqnarray}
	\mathcal{B}\big( \hat{P}(x),y \big)=\mathcal{B}\big( x,P(y) \big),\;\forall x,y\in A.
\end{eqnarray}
\end{enumerate}

\section{Averaging commutative and cocommutative infinitesimal bialgebras}\label{sec:2}\

We study quasi-triangular averaging commutative and cocommutative infinitesimal bialgebras which arise from solutions of the AAYBE whose symmetric parts are invariant. We study $\mathcal{O}$-operators of admissible averaging commutative  algebras which characterize the AAYBE in terms of operator forms. Then we study triangular and factorizable averaging commutative and cocommutative infinitesimal bialgebras. They serve as subclasses of quasi-triangular averaging commutative and cocommutative infinitesimal bialgebras and are closely related to symmetric Rota-Baxter Frobenius commutative algebras with averaging operators satisfying the commutativity conditions.

\subsection{Quasi-triangular averaging commutative and cocommutative infinitesimal bialgebras, the averaging associative Yang-Baxter equation and $\mathcal{O}$-operators}\

We recall some basic conclusions on commutative and cocommutative infinitesimal bialgebras.

\begin{defi}\cite{Bai2010}
	A {\bf cocommutative (coassociative) coalgebra} is a pair $(A,\Delta)$, where $A$ is
	a vector space and $\Delta:A\rightarrow A\otimes A$ is a
	co-multiplication such that the following equations hold:
	\begin{equation}
		\Delta=\tau\Delta,\;(\Delta\otimes\mathrm{id})\Delta=(\mathrm{id}\otimes\Delta)\Delta.
	\end{equation}
	A {\bf commutative and cocommutative infinitesimal
		bialgebra}
	is a triple $(A,\cdot_{A},\Delta)$ such that $(A,\cdot_{A})$ is a
	commutative algebra, $(A,\Delta)$ is a cocommutative coalgebra and the
	following equation holds:
	\begin{equation}\label{eq:bib}
		\Delta(x\cdot_{A}y)=\big(\mathcal{L}_{\cdot_{A}}(x)\otimes\mathrm{id}\big)\Delta(y)+\big(\mathrm{id}\otimes\mathcal{L}_{\cdot_{A}}(y)\big)\Delta(x),\;\forall x,y\in A.
	\end{equation}
\end{defi}

Let $(A,\cdot_{A})$ be a commutative  algebra and $r=\sum\limits_{i}u_{i}\otimes v_{i}\in A\otimes A$.
Let $\Delta_{r}:A\rightarrow A\otimes A$ be a co-multiplication given by
\begin{equation}\label{eq:comul1}
    \Delta_{r}(x)=\big(\mathrm{id}\otimes\mathcal{L}_{\cdot_{A}}(x)-\mathcal{L}_{\cdot_{A}}(x)\otimes\mathrm{id}\big)r ,\;\forall x\in A.
\end{equation}
By \cite{Bai2010}, $(A,\cdot_{A},\Delta_{r})$ is a commutative and cocommutative infinitesimal bialgebra if and only if the following equations hold:
\begin{eqnarray}
&&\big(\mathcal{L}_{\cdot_{A}}(x)\otimes\mathrm{id}-\mathrm{id}\otimes\mathcal{L}_{\cdot_{A}}(x)\big)\big( \mathrm{id}\otimes\mathcal{L}_{\cdot_{A}}(y)-\mathcal{L}_{\cdot_{A}}(y)\otimes\mathrm{id} \big)\big(r+\tau(r)\big)=0,\label{501}\\
&&\big( \mathrm{id}\otimes\mathrm{id}\otimes\mathcal{L}_{\cdot_{A}}(x)-\mathcal{L}_{\cdot_{A}}(x)\otimes\mathrm{id}\otimes\mathrm{id}  \big){\bf A}(r)=0.\label{502}
\end{eqnarray}
Here ${\bf A}(r)\in A\otimes A\otimes A$ is given by
\begin{equation}
{\bf A}(r)=\sum_{i,j} u_{i}\cdot_{A}u_{j}\otimes v_{i}\otimes v_{j}- u_{i}\otimes u_{j}\cdot_{A}v_{i}\otimes v_{j}+u_{i}\otimes u_{j}\otimes v_{i}\cdot_{A}v_{j}
\end{equation}
and the equation ${\bf A}(r)=0$ is called the {\bf associative Yang-Baxter equation} (or {\bf AYBE} in short) in $(A,\cdot_{A})$.

Let $(A,\cdot_{A})$ be a commutative  algebra and $s\in A\otimes A$. If the following equation holds:
\begin{equation*}
    \big(\mathrm{id}\otimes\mathcal{L}_{\cdot_{A}}(x)-\mathcal{L}_{\cdot_{A}}(x)\otimes\mathrm{id}\big)s=0,\;\forall x\in A,
\end{equation*}
then we say $s$ is {\bf invariant} on $(A,\cdot_{A})$. Therefore by the above discussion,
if $r$ is a solution of the AYBE in $(A,\cdot_{A})$ and the symmetric part of $r$ is invariant, that is,
\begin{equation*}
    \big(\mathrm{id}\otimes\mathcal{L}_{\cdot_{A}}(x)-\mathcal{L}_{\cdot_{A}}(x)\otimes\mathrm{id}\big)\big(r+\tau(r)\big)=0,\;{\bf A}(r)=0,
\end{equation*}
then $(A,\cdot_{A},\Delta_{r})$ is a commutative and cocommutative infinitesimal bialgebra which is called {\bf quasi-triangular} \cite{SW}.

The notions of admissible averaging commutative  algebras and averaging commutative and cocommutative infinitesimal bialgebras are introduced as follows.

\begin{defi}\cite{Bai2024}
An {\bf admissible averaging commutative  algebra} is a quadruple $(A,\cdot_{A},P,Q)$, where $(A,\cdot_{A},P)$ is an averaging commutative  algebra and $Q:A\rightarrow A$ is a linear map such that
\begin{eqnarray}
P(x)\cdot_{A} Q(y)=Q\big(P(x)\cdot_{A} y\big)=Q\big(x\cdot_{A}Q(y)\big),\;\forall x,y\in A.\label{eq:ao pair}
\end{eqnarray}
\end{defi}

\begin{defi}\label{defi:ave com ASI bialgebra}\cite{Bai2024}
Let $A$ be a vector space together with linear maps
\begin{equation*}
    \cdot_{A}:A\otimes A\rightarrow A,\; \Delta:A\rightarrow A\otimes A, \; P,Q:A\rightarrow A.
\end{equation*}
Suppose that the following conditions hold:
\begin{enumerate}
    \item the triple $(A,\cdot_{A},\Delta)$ is a commutative and cocommutative infinitesimal bialgebra.
    \item $(A,\cdot_{A},P,Q)$ is an admissible averaging commutative  algebra.
    \item the following equations hold:
    \begin{eqnarray}
    &&(Q\otimes Q)\Delta(x)=(Q\otimes\mathrm{id})\Delta\big(Q(x)\big),\label{eq:aoco1}\\
    &&(Q\otimes P)\Delta(x)=(Q\otimes\mathrm{id})\Delta\big(P(x)\big)=(\mathrm{id}\otimes P)\Delta\big(P(x)\big),\;\forall x\in A. \label{eq:aoco2}
\end{eqnarray}
\end{enumerate}
Then we say $(A,\cdot_{A},\Delta,P,Q)$ is an {\bf averaging commutative and cocommutative infinitesimal bialgebra}.
\end{defi}


Now we study the construction theory of averaging commutative and cocommutative infinitesimal bialgebras.

\begin{pro}\label{pro:coao}
    Let $(A,\cdot_{A},P,Q)$ be an admissible averaging commutative  algebra and $r=\sum\limits_{i} u_{i}\otimes v_{i} \in A\otimes A$.
    Let $\Delta_{r}:A\rightarrow A\otimes A$ be a co-multiplication given by \eqref{eq:comul1}.
    \begin{enumerate}
        \item\label{pro:coao1} $\Delta_{r}$ satisfies \eqref{eq:aoco1} if and only if the following equation holds:
            {\small
        \begin{equation}\label{eq:coao1}
            \big(\mathrm{id}\otimes Q\mathcal{L}_{\cdot_{A}}(x)-\mathrm{id}\otimes\mathcal{L}_{\cdot_{A}} (Qx) \big)(Q\otimes\mathrm{id}-\mathrm{id}\otimes P)r+\big(Q\mathcal{L}_{\cdot_{A}}(x)\otimes\mathrm{id}\big)(P\otimes\mathrm{id}-\mathrm{id}\otimes Q)r=0,\;\forall x\in A.
        \end{equation}}
        \item\label{pro:coao2}  $\Delta_{r}$ satisfies \eqref{eq:aoco2} if and only if the following equations hold:
        \begin{align}
    &\big(\mathrm{id}\otimes P\mathcal{L}_{\cdot_{A}}(x)+Q\mathcal{L}_{\cdot_{A}}(x)\otimes\mathrm{id}-\mathrm{id}\otimes\mathcal{L}_{\cdot_{A}}(Px)\big)(Q\otimes\mathrm{id}-\mathrm{id}\otimes P)r=0,\label{eq:coao2}\\
    &\big(\mathrm{id}\otimes P\mathcal{L}_{\cdot_{A}}(x)+Q\mathcal{L}_{\cdot_{A}}(x)\otimes\mathrm{id}- \mathcal{L}_{\cdot_{A}}(Px)\otimes\mathrm{id}\big)(Q\otimes\mathrm{id}-\mathrm{id}\otimes P)r=0,\;\forall x\in A.\label{eq:coao3}
        \end{align}
    \end{enumerate}
\end{pro}
\begin{proof}
For all $x\in A$, we have
\begin{align*}
&(Q\otimes Q)\Delta_{r}(x)-(Q\otimes\mathrm{id})\Delta_{r}\big(Q(x)\big)\\
&=\sum_{i} Q(u_{i})\otimes Q(x\cdot_{A}v_{i})-Q(x\cdot_{A}u_{i})\otimes Q(v_{i})-Q(u_{i})\otimes Q(x)\cdot_{A}v_{i}+Q\big(Q(x)\cdot_{A}u_{i}\big)\otimes v_{i}\\
&\overset{\eqref{eq:ao pair}}{=}\sum_{i} Q(u_{i})\otimes Q(x\cdot_{A}v_{i})-u_{i}\otimes Q\big(x\cdot_{A}P(v_{i})\big)+u_{i}\otimes P(v_{i})\cdot_{A}Q(x)\\
&\ \ -Q(x\cdot_{A}u_{i})\otimes Q(v_{i})-Q(u_{i})\otimes Q(x)\cdot_{A}v_{i}+Q\big(x\cdot_{A}P(u_{i})\big)\otimes v_{i}\\
&=\big(\mathrm{id}\otimes Q\mathcal{L}_{\cdot_{A}}(x)\big)(Q\otimes\mathrm{id}-\mathrm{id}\otimes P)r+
\big(  Q\mathcal{L}_{\cdot_{A}}(x)\otimes\mathrm{id}\big)(P\otimes\mathrm{id}-\mathrm{id}\otimes Q)r\\
&\ \ +\big(\mathrm{id}\otimes\mathcal{L}_{\cdot_{A}}(Qx)\big)(\mathrm{id}\otimes P-Q\otimes\mathrm{id})r.
\end{align*}
Hence (\ref{pro:coao1}) holds. Similarly we get (\ref{pro:coao2}).
\end{proof}

\begin{defi}
Let $(A,\cdot_{A},P,Q)$ be an admissible averaging commutative  algebra.
If $r=\sum\limits_{i} u_{i}\otimes v_{i}\in A\otimes A$ satisfies ${\bf A}(r)=0$ and the following equations:
\begin{eqnarray}
    && (P\otimes\mathrm{id}-\mathrm{id}\otimes Q)r=0,\;\label{eq:AAYBE1}\\
    && (Q\otimes\mathrm{id}-\mathrm{id}\otimes P)r=0,\;\label{eq:AAYBE2}
\end{eqnarray}
then we say $r$ is a solution of  the {\bf averaging associative Yang-Baxter equation} (or {\bf AAYBE} in short) in $(A,\cdot_{A},P,Q)$.
\end{defi}

\begin{thm}\label{thm:coao}
    Let $(A,\cdot_{A},P,Q)$ be an admissible averaging commutative  algebra and $r\in A\otimes A$.
    Let $\Delta_{r}:A\rightarrow A\otimes A$ be a co-multiplication given by \eqref{eq:comul1}.
    Then $(A,\cdot_{A},\Delta_{r},P,Q)$ is an averaging commutative and cocommutative infinitesimal bialgebra if and only if \eqref{501}, \eqref{502}, \eqref{eq:coao1}-\eqref{eq:coao3} hold. In particular, 
   if $r+\tau(r)$ is invariant on $(A,\cdot_{A})$ and $r$ is a solution of the AAYBE in $(A,\cdot_{A},P,Q)$, then $(A,\cdot_{A},\Delta_{r},P,Q)$ is an averaging commutative and cocommutative infinitesimal bialgebra.
\end{thm}
\begin{proof}
    It follows from Proposition \ref{pro:coao}.
\end{proof}

\begin{defi}
 Let $(A,\cdot_{A},P,Q)$ be an admissible averaging commutative  algebra. Suppose that  $r\in A\otimes A$ is a solution of the AAYBE in $(A,\cdot_{A},P,Q)$ and $r+\tau(r)$ is invariant on $(A,\cdot_{A})$. Then we say the resulting averaging commutative and cocommutative infinitesimal bialgebra $(A,\cdot_{A},\Delta_{r},P,Q)$ obtained from Theorem \ref{thm:coao} is {\bf quasi-triangular}.
\end{defi}

\delete{
\begin{pro}
    Let $(A,\cdot_{A},P,Q)$ be an admissible averaging commutative  algebra.
    Then we have
    \begin{eqnarray*}
        (P\otimes \mathrm{id}-\mathrm{id}\otimes Q)\big(-\tau(r)\big)&=&\tau(Q\otimes\mathrm{id}-\mathrm{id}\otimes P)r,\\
        (Q\otimes \mathrm{id}-\mathrm{id}\otimes P)\big(-\tau(r)\big)&=&\tau(P\otimes\mathrm{id}-\mathrm{id}\otimes Q)r,\\
        {\bf A}\big(-\tau(r)\big)&=&\tau_{13}{\bf A}(r),
    \end{eqnarray*}
    where $\tau_{13}(x\otimes y\otimes z)=z\otimes y\otimes x$ for all $x,y,z\in A$.
    Consequently, $r$ is a solution of the AAYBE in $(A,\cdot_{A},P,Q)$ if and only if $-\tau(r)$ is a solution of the AAYBE in $(A,\cdot_{A},P,Q)$.
    Furthermore, if $(A,\cdot_{A},\Delta_{r},P,Q)$ is a quasi-triangular averaging bicommutative infinitesimal bialgebra, then
    $(A,\cdot_{A},\Delta_{-\tau(r)},P,Q)$ is also a quasi-triangular averaging bicommutative infinitesimal bialgebra.
\end{pro}}


Now we study representations of admissible averaging commutative algebras.

\begin{defi}\cite{Bai2024}
A {\bf representation} of an averaging commutative  algebra $(A,\cdot_{A},$
$P)$ is a triple $(\mu,\alpha,V)$, such that $(\mu,V)$ is a representation of $(A,\cdot_{A})$, that is, $\mu:A\rightarrow\mathrm{End}_{\mathbb K}(V)$ is a linear map satisfying
\begin{equation*}
\mu (x\cdot_{A}y)v=\mu (x)\mu (y)v,\;\forall x,y\in A,v\in V,
\end{equation*}
and $\alpha:V\rightarrow V$ is a linear map satisfying
\begin{equation}\label{eq:rep ao}
    \mu(Px)\alpha(v)=\alpha\big(\mu(Px)v\big)=\alpha\big(\mu(x)\alpha(v)\big),\;\forall x\in A, v\in V.
\end{equation}
\end{defi}


\begin{defi}
Let $(A,\cdot_{A},P,Q)$ be an admissible averaging commutative  algebra.
Suppose that $(\mu,\alpha,V)$ is a representation of the averaging commutative  algebra $(A,\cdot_{A},P)$.
If there is a linear map $\beta:V\rightarrow V$ such that
\begin{eqnarray}
&&\mu(Px)\beta(v)=\beta\big(\mu(Px)v\big)=\beta\big(\mu(x)\beta(v)\big),\label{eq:aver pair rep}\\
&&\mu(Qx)\alpha(v)=\beta\big(\mu(x)\alpha(v)\big)=\beta\big(\mu(Qx)v\big),\;\forall x\in A, v\in V,\label{eq:ex rep}
\end{eqnarray}
then we say the quadruple $(\mu,\alpha,\beta,V)$ is a {\bf representation} of $(A,\cdot_{A},P,Q)$.
\end{defi}

\begin{ex}
Let $(A,\cdot_{A},P,Q)$ be an admissible averaging commutative  algebra.
Then $(\mathcal{L}_{\cdot_{A}},P,Q,A)$ is a representation of $(A,\cdot_{A},P,Q)$ which is called the {\bf adjoint representation} of $(A,\cdot_{A},P,Q)$.
\end{ex}

\delete{
For vector spaces $V_1$ and $V_2$ and linear maps
$\phi_1:V_1\rightarrow V_1$ and $\phi_2:V_2\rightarrow V_2$, let
$\phi_1+\phi_2$ denote the linear map:
\begin{equation}\label{eq:pro:SD RB Lie2}
\phi_{V_1\oplus V_2}:V_1\oplus V_2\rightarrow V_1\oplus
V_2,\;\;\phi_{V_1\oplus
V_2}(v_1+v_2)=\phi_1(v_1)+\phi_2(v_2),\;\;\forall v_1\in
V_1,v_2\in V_2.
\end{equation}}

\begin{pro}
  Let $(A,\cdot_{A},P,Q)$ be an admissible averaging commutative  algebra.
  Let $V$ be a vector space and $\mu:A\rightarrow\mathrm{End}_{\mathbb K}(V),\;\alpha,\beta:V\rightarrow V$ be linear maps.
  Then $(\mu,\alpha,\beta,V)$ is a representation of $(A,\cdot_{A},P,Q)$ if and only if there is a commutative  algebra structure on the direct sum $A\oplus V$ of vector spaces given by
  \begin{equation*}
    (x+u)\cdot_{d}(y+v)=x\cdot_{A}y+\mu(x)v+\mu(y)u,\;\forall  x,y\in A, u,v\in V,
  \end{equation*}
  such that $(A\oplus V,\cdot_{d},P+\alpha,Q+\beta)$ is an admissible averaging commutative  algebra. In this case, we denote the admissible averaging commutative  algebra $(A\oplus V,\cdot_{d},P+\alpha,Q+\beta)$ by $(A\ltimes_{\mu}V,P+\alpha,Q+\beta)$ and call it the {\bf semi-direct product admissible averaging commutative  algebra of $(A,\cdot_{A},P,Q)$ with respect to $(\mu,\alpha,\beta,V)$}.
\end{pro}
\begin{proof}
    It is the special case of Proposition \ref{552} when the commutative multiplication on $V$ is taken to be zero.
\end{proof}

\delete{
Let $\beta:W\rightarrow V$ be a linear map. The linear map $\beta^{*}:V^{*}\rightarrow W^{*}$ is defined by
\begin{equation*}
\langle\beta^{*}(u^{*}),w\rangle=\langle u^{*},\beta(w)\rangle,\;\forall u^{*}\in V^{*},w\in W.
\end{equation*}
Moreover, for a linear map $\mu:A\rightarrow\mathrm{End}_{\mathbb
K}(V)$, we denote the linear map
$\mu^{*}:A\rightarrow\mathrm{End}_{\mathbb K}(V^{*})$ by
\begin{equation*}
\langle \mu^{*}(x)u^{*},v\rangle=\langle u^{*},\mu(x)v\rangle,\;\forall x\in
A, u^{*}\in V^{*}, v\in V.
\end{equation*}}

\begin{pro}\label{pro:3.7}
    Let $(A,\cdot_{A},P,Q)$ be an admissible averaging commutative  algebra.
    If $(\mu,\alpha,\beta,V)$ is a representation of $(A,\cdot_{A},P,Q)$, then $(\mu^{*},\beta^{*},\alpha^{*},V^{*})$ is also a representation of $(A,\cdot_{A},P,Q)$.
    In particular, $(\mathcal{L}^{*}_{\cdot_{A}},Q^{*},P^{*},A^{*} )$ is a representation of $(A,\cdot_{A},P,Q)$.
\end{pro}
\begin{proof}
 By \cite{Bai2024}, $(\mu^{*},\beta^{*},V^{*})$ is a representation of $(A,\cdot_{A},P)$.
    For all $x\in A, u^{*}\in V^{*},v\in V$, we have
    \begin{eqnarray*}
        \langle \mu^{*}(Px)\alpha^{*}(u^{*}),v\rangle&=&\langle u^{*},\alpha\big(\mu(Px)v\big)\rangle,\\
        \langle \alpha^{*}\big(\mu^{*}(Px)u^{*}\big),v\rangle&=&\langle u^{*},\mu(Px)\alpha(v)\rangle,\\
        \langle \alpha^{*}\big(\mu^{*}(x)\alpha^{*}(u^{*})\big),v\rangle&=&\langle u^{*},\alpha\big(\mu(x)\alpha(v)\big)\rangle.
    \end{eqnarray*}
    Hence by \eqref{eq:rep ao}, we have 
    \begin{equation*}
        \mu^{*}(Px)\alpha^{*}(u^{*})=\alpha^{*}\big(\mu^{*}(Px)u^{*}\big)=\alpha^{*}\big(\mu^{*}(x)\alpha^{*}(u^{*})\big).
    \end{equation*}
    Similarly by \eqref{eq:ex rep}, we have
    \begin{equation*}
        \mu^{*}(Qx)\beta^{*}(u^{*})=\alpha^{*}\big( \mu^{*}(x)\beta^{*}(u^{*})\big)=\alpha^{*}\big( \mu^{*}(Qx)u^{*}\big).
    \end{equation*}
    Hence $(\mu^{*},\beta^{*},\alpha^{*},V^{*})$ is a representation of $(A,\cdot_{A},P,Q)$.
    \end{proof}

\delete{
Now we introduce the notion of an $(A,\cdot_{A},P,Q)$-representation algebra, where $(A,\cdot_{A},P,Q)$ is an admissible averaging commutative  algebra.}

\begin{defi}
 Let $(A,\cdot_{A},P,Q)$ and $(V,\cdot_{V},\alpha,\beta)$ be admissible averaging commutative  algebras.
 If there is a linear map $\mu:A\rightarrow\mathrm{End}_{\mathbb K}(V)$ such that $(\mu,\alpha,\beta,V)$ is a representation of $(A,\cdot_{A},P,Q)$ and the following equation holds:
 \begin{equation}
     \mu(x)(u\cdot_{V}v)=\mu(x)u\cdot_{V}v,\;\forall x\in A, u,v\in V,
 \end{equation}
 then we say $(\mu,\alpha,\beta,V,\cdot_{V})$ is an {\bf $(A,\cdot_{A},P,Q)$-representation algebra}.
\end{defi}

\begin{pro}\label{552}
    Let $(A,\cdot_{A},P,Q)$ be an admissible averaging commutative  algebra.
    Suppose that $V$ is a vector space and
    \begin{equation*}
        \mu:A\rightarrow\mathrm{End}_{\mathbb K}(V),\; \alpha,\beta:V\rightarrow V,\; \cdot_{V}:V\otimes V\rightarrow V
    \end{equation*}
    are linear maps.
    Then $(\mu,\alpha,\beta,V,\cdot_{V})$ is an $(A,\cdot_{A},P,Q)$-representation algebra  if and only if $(A\oplus V,\cdot_{d},P+\alpha,Q+\beta)$ is an admissible averaging commutative  algebra, where the multiplication $\cdot_{d}$ on 
  $A\oplus V$ is given by
    \begin{equation}
        (x+u)\cdot_{d}(y+v)=x\cdot_{A}y+\mu(x)v+\mu(y)u+u\cdot_{V}v,\;\forall x,y\in A, u,v\in V.
    \end{equation}
\end{pro}
\begin{proof}
    It follows from a straightforward computation.
\end{proof}

\delete{
Let $A$ be a vector space. We can identify a $2$-tensor $r\in A\otimes A$ as
a linear map $r^{\sharp}:A^{*}\rightarrow A$ by
\begin{equation*}
    \langle r^{\sharp}(a^{*}),b^{*}\rangle=\langle r, a^{*}\otimes
b^{*}\rangle,\;\forall a^{*},b^{*}\in A^{*}.
\end{equation*}}

\begin{pro}
    Let $(A,\cdot_{A},P,Q)$ be an admissible averaging commutative  algebra.
    Let $s\in A\otimes A$ be symmetric and invariant on $(A,\cdot_{A})$.
    If the following equation holds:
    \begin{equation}\label{eq:op:cond}
        (P\otimes\mathrm{id}-\mathrm{id}\otimes Q)s=0,
    \end{equation}
    then there is an $(A,\cdot_{A},P,Q)$-representation algebra $(\mathcal{L}^{*}_{\cdot_{A}},Q^{*},P^{*},A^{*},\lozenge_{s})$ with the commutative  multiplication $\lozenge_{s}$ on $A^{*}$ given by
    \begin{equation}
        a^{*}\lozenge_{s}b^{*}=\mathcal{L}^{*}_{\cdot_{A}}\big(s^{\sharp}(a^{*})\big)b^{*},\;\forall a^{*},b^{*}\in A^{*}.
    \end{equation}
\end{pro}
\begin{proof}
By \cite{BGGZ}, $(A^{*},\lozenge_{s})$ is a commutative  algebra and the following equation holds:
\begin{equation*}
\mathcal{L}^{*}_{\cdot_{A}}(x)(a^{*}\lozenge_{s} b^{*})=\mathcal{L}^{*}_{\cdot_{A}}(x)a^{*}\lozenge_{s} b^{*},\;\forall x\in A, a^{*},b^{*}\in A^{*}.
\end{equation*}
 By Proposition \ref{pro:3.7}, $(\mathcal{L}^{*}_{\cdot_{A}},Q^{*},P^{*},A^{*})$ is a representation of $(A,\cdot_{A}, P,Q)$.
For all $x\in A$,
$a^{*},b^{*}\in A^{*}$, we have
{\small
\begin{eqnarray*}
 \langle Q^{*}(a^{*})\lozenge_{s} Q^{*}(b^{*}),x\rangle&=&\langle \mathcal{L}^{*}_{\cdot_{A}}\Big(s^{\sharp}\big(Q^{*}(a^{*})\big)\Big) Q^{*}(b^{*}),x\rangle=\langle Q^{*}(b^{*}), s^{\sharp}\big(Q^{*}(a^{*})\big)\cdot_{A}x\rangle\\
&=&\langle\mathcal{L}^{*}_{\cdot_{A}}(x)Q^{*}(b^{*}),s^{\sharp}\big(Q^{*}(a^{*})\big)\rangle=\langle s, Q^{*}(a^{*})\otimes\mathcal{L}^{*}_{\cdot_{A}}(x)Q^{*}(b^{*})\rangle\\
&=&\langle\big(Q\otimes Q\mathcal{L}_{\cdot_{A}}(x)\big)s, a^{*}\otimes b^{*}\rangle 
 \overset{\eqref{eq:op:cond}}{=}\langle \big(\mathrm{id}\otimes Q\mathcal{L}_{\cdot_{A}}(x)P\big)s, a^{*}\otimes b^{*}\rangle\\
 &
\overset{\eqref{eq:ao pair}}{=}&\langle \Big(\mathrm{id}\otimes\mathcal{L}_{\cdot_{A}}\big(Q(x)\big)P\Big)s, a^{*}\otimes b^{*}\rangle 
 \overset{\eqref{eq:op:cond}}{=}\langle \Big(Q\otimes\mathcal{L}_{\cdot_{A}}\big(Q(x)\big)\Big)s, a^{*}\otimes b^{*}\rangle\\
 &=&\langle s, Q^{*}(a^{*})\otimes \mathcal{L}^{*}_{\cdot_{A}}\big(Q(x)\big)b^{*}\rangle 
=\langle s^{\sharp}\big(Q^{*}(a^{*})\big),\mathcal{L}^{*}_{\cdot_{A}}\big(Q(x)\big)b^{*}\rangle\\
&=&
 \langle b^{*}, s^{\sharp}\big(Q^{*}(a^{*})\big)\cdot_{A}Q(x)\rangle 
=\langle\mathcal{L}^{*}_{\cdot_{A}}\Big( s^{\sharp}\big(Q^{*}(a^{*})\big)\Big)b^{*},Q(x)\rangle\\
&=&\langle Q^{*}(a^{*})\lozenge_{s}b^{*}, Q(x)\rangle
=\langle Q^{*}\big(Q^{*}(a^{*})\lozenge_{s}b^{*}\big),x\rangle.
\end{eqnarray*}}Hence  $(A^{*},\lozenge_{s},Q^{*})$ is an averaging commutative algebra.
 Similarly, \eqref{eq:ao pair} holds for $(A^{*},\lozenge_{s}$,
 $Q^{*}$,
 $P^{*})$. Thus $(A^{*},\lozenge_{s},$
$Q^{*},P^{*})$ is an admissible averaging commutative  algebra.
Therefore $(\mathcal{L}^{*}_{\cdot_{A}},Q^{*},P^{*},$
$A^{*},\lozenge_{s})$ is an $(A,\cdot_{A},P,Q)$-representation algebra.
\end{proof}

We introduce the notion of $\mathcal{O}$-operators with weights of admissible averaging commutative  algebras.

\begin{defi}
Let $(A,\cdot_{A},P,Q)$ be an admissible averaging commutative  algebra.
Suppose that $(\mu,\alpha,\beta,V,\cdot_{V})$ is an $(A,\cdot_{A},P,Q)$-representation algebra.
A linear map $T:V\rightarrow A$ is called an {\bf $\mathcal{O}$-operator of weight $\lambda\in\mathbb{K}$ of $(A,\cdot_{A},P,Q)$ associated to} $(\mu,\alpha,\beta,V,\cdot_{V})$ if the following equations hold:
\begin{eqnarray}
    Tu\cdot_{A}Tv&=&T\big(\mu(Tu)v+\mu(Tv)u+\lambda u\cdot_{V}v\big),\;\forall u,v\in V,\label{eq:oop1}\\
    PT&=&T\alpha,\label{eq:oop2}\\
    QT&=&T\beta.\label{eq:oop3}
\end{eqnarray}
In particular, if $V$ is equipped with the zero multiplication
in the sense that $u\cdot_{V}v=0$ for all $u,v\in V$, 
then we simply say $T:V\rightarrow A$ is an {\bf $\mathcal{O}$-operator of $(A,\cdot_{A},P,Q)$ associated to $(\mu,\alpha,\beta,V)$}.
\end{defi}

\delete{u\cdot_{V}v=0}

\begin{thm}\label{thm:rt oop}
Let $(A,\cdot_{A},P,Q)$ be an admissible averaging commutative  algebra.
Suppose that $r\in A\otimes A$ and $r+\tau(r)$ is invariant on $(A,\cdot_{A})$.
Then the following conditions are equivalent:
\begin{enumerate}
    \item $r$ is a solution of the AAYBE in $(A,\cdot_{A},P,Q)$ such that $(A,\cdot_{A},\Delta_{r},P,Q)$ is a quasi-triangular averaging commutative and cocommutative infinitesimal bialgebra.
    \item $r^{\sharp}$ is an $\mathcal{O}$-operator of weight $-1$ of $(A,\cdot_{A},P,Q)$ associated to the $(A,\cdot_{A},P,Q)$-representation algebra $(\mathcal{L}^{*}_{\cdot_{A}},Q^{*},P^{*},A^{*},\lozenge_{r+\tau(r)})$, where  the multiplication $\lozenge_{r+\tau(r)}$ is given by
    \begin{equation}
        a^{*}\lozenge_{r+\tau(r)} b^{*}=\mathcal{L}^{*}_{\cdot_{A}}\Big( \big(r+\tau(r)\big)^{\sharp}(a^{*}) \Big)b^{*},\;\forall a^{*},b^{*}\in A^{*}.
    \end{equation}
    That is, the following equations hold:
    \begin{eqnarray}
        r^{\sharp}(a^{*})\cdot_{A} r^{\sharp}(b^{*})&=& r^{\sharp}\Big( \mathcal{L}^{*}_{\cdot_{A}}\big( r^{\sharp}(a^{*})  \big)b^{*}+
        \mathcal{L}^{*}_{\cdot_{A}}\big( r^{\sharp}(b^{*})  \big)a^{*}-a^{*}\lozenge_{r+\tau(r)}b^{*}\Big),\label{eq:rt oop1}\\
        Pr^{\sharp}&=&r^{\sharp}Q^{*},\label{eq:rt oop2}\\
        Qr^{\sharp}&=&r^{\sharp}P^{*}.\label{eq:rt oop3}
    \end{eqnarray}
\end{enumerate}
\end{thm}
\begin{proof}
 By the assumption, it follows from \cite{BGGZ} that $r\in A\otimes A$ is a solution of the AYBE in $(A,\cdot_{A})$ if and only if \eqref{eq:rt oop1} holds.
 Now suppose that $(A,\cdot_{A},P,Q)$ is an admissible averaging commutative  algebra.
 For all $a^{*},b^{*}\in A^{*}$, we have
 {\small
 \begin{equation*}
  \langle (P\otimes\mathrm{id}-\mathrm{id}\otimes Q)r, a^{*}\otimes b^{*}\rangle=\langle r, P^{*}(a^{*})\otimes b^{*}-a^{*}\otimes Q^{*}(b^{*})\rangle=\langle r^{\sharp}P^{*}(a^{*})-Q\big(r^{\sharp}(a^{*})\big),b^{*}\rangle.
 \end{equation*}}Hence \eqref{eq:AAYBE1} holds if and only if \eqref{eq:rt oop3} holds.
 Similarly,  \eqref{eq:AAYBE2} holds if and only if \eqref{eq:rt oop2} holds.
 Therefore the conclusion follows.
\end{proof}

\delete{
Let $\mathcal{B}$ be a nondegenerate bilinear form on a vector space $A$.
Then there is a bijection   $\mathcal{B}^{\natural}:A\rightarrow A^{*}$ given by
\begin{equation}\label{eq:B}
    \mathcal{B}(x,y)=\langle\mathcal{B}^{\natural}(x),y\rangle,\;\forall x,y\in A.
\end{equation} 
Let $V$ be a vector space. Then the isomorphism ${\rm
Hom}_{\mathbb K}(V\otimes V,\mathbb K)\cong {\rm Hom}_{\mathbb
K}(V, V^*)$ identifies a bilinear form  $\mathcal{B}:V\otimes
V\rightarrow \mathbb K$ on V as a linear map $\mathcal
B^\natural:V\rightarrow V^*$ by
\begin{equation}\label{eq:B}
\mathcal B(u,v)=\langle \mathcal B^\natural (u),
v\rangle,\;\;\forall u,v\in V.
\end{equation}
Moreover, $\mathcal B$ is
nondegenerate if and only if $\mathcal B^\natural$ is invertible.}
 Now we apply Theorem \ref{thm:rt oop} to the case of symmetric Frobenius commutative algebras.

\begin{pro}\label{pro:3.14}
Let $(A,\cdot_{A},P,Q)$ be an admissible averaging commutative  algebra.
Suppose that there is a nondegenerate symmetric invariant bilinear form $\mathcal{B}$ on $(A,\cdot_{A})$, 
that is, $(A,\cdot_{A},\mathcal{B})$ is a symmetric Frobenius commutative algebra.
Assume that $r\in A\otimes A$ and $r+\tau(r)$ is invariant on $(A,\cdot_{A})$.
Define a linear map $R:A\rightarrow A$ by
	\begin{equation}\label{Br,Brt}
		R(x)=r^{\sharp}\mathcal{B}^{\natural}(x),\;\forall x\in A.
	\end{equation}
Then
$r$ is a solution of the AAYBE in $(A,\cdot_{A},P,Q)$ if and only if the following equations hold:
\begin{eqnarray}
    R(x)\cdot_{A}R(y)&=& R\big(R(x)\cdot_{A}y+x\cdot_{A}R(y)-x\cdot_{A}\big(r+\tau(r)\big)^{\sharp}\mathcal{B}^{\natural}(y)\big),\;\forall x,y\in A,\label{eq:rb1}\\
    PR&=&R \hat{Q},\label{eq:rb2}\\
    QR&=&R\hat{P}.\label{eq:rb3}
\end{eqnarray}
\delete{
where $\hat{P},\hat{Q}:A\rightarrow A$ are the adjoint maps of $P$ and $Q$ with respect to $\mathcal{B}$ respectively, that is,
\begin{equation*} \mathcal{B}\big(P(x),y\big)=\mathcal{B}\big(x,\hat{P}(y)\big),\;\mathcal{B}\big(Q(x),y\big)=\mathcal{B}\big(x,\hat{Q}(y)\big),\;\forall x,y\in A.
\end{equation*}}
\end{pro}
\begin{proof}
    Let $x,y,z\in A$ and $a^{*}=\mathcal{B}^{\natural}(x), b^{*}=\mathcal{B}^{\natural}(y)$.
    We first observe that
    \begin{equation*}
\langle \mathcal{B}^{\natural}\big(x\cdot_{A}\mathcal{B}^{\natural^{-1}}(b^{*})\big),z\rangle=\mathcal{B}\big(x\cdot_{A}\mathcal{B}^{\natural^{-1}}(b^{*}),z\big)
=\mathcal{B}\big(\mathcal{B}^{\natural^{-1}}(b^{*}),x\cdot_{A}z\big)
=\langle\mathcal{L}^{*}_{\cdot_{A}}(x)b^{*},z\rangle,
    \end{equation*}
    that is,
    \begin{equation}\label{eq:b}
        \mathcal{B}^{\natural}\big( x\cdot_{A}  \mathcal{B}^{\natural^{-1}}(b^{*})\big)=\mathcal{L}^{*}_{\cdot_{A}}(x)b^{*}.
    \end{equation}
    Then we have
    \begin{eqnarray*}
        R(x)\cdot_{A}R(y)&=&r^{\sharp}(a^{*})\cdot_{A} r^{\sharp}(b^{*}),\\
        R\big(R(x)\cdot_{A}y\big)&=&r^{\sharp}\mathcal{B}^{\natural}\big( r^{\sharp}(a^{*})\cdot_{A}\mathcal{B}^{\natural^{-1}}(b^{*})\big)\overset{\eqref{eq:b}}{=}r^{\sharp}\Big( \mathcal{L}^{*}_{\cdot_{A}}\big(r^{\sharp}(a^{*})\big)b^{*}  \Big),\\
        R\big(x\cdot_{A}R(y)\big)&=&r^{\sharp}\mathcal{B}^{\natural}\big( r^{\sharp}(b^{*})\cdot_{A}\mathcal{B}^{\natural^{-1}}(a^{*})\big)\overset{\eqref{eq:b}}{=}r^{\sharp}\Big( \mathcal{L}^{*}_{\cdot_{A}}\big(r^{\sharp}(b^{*})\big)a^{*}  \Big),\\
        -R\Big(x\cdot_{A}\big(r+\tau(r)\big)^{\sharp}\mathcal{B}^{\natural}(y)\Big)&=&-r^{\sharp}\mathcal{B}^{\natural}\Big( \mathcal{B}^{\natural^{-1}}(a^{*})\cdot_{A}\big(r+\tau(r)\big)^{\sharp}(b^{*}) \Big)\\
        &\overset{\eqref{eq:b}}{=}&-r^{\sharp}\bigg( \mathcal{L}^{*}_{\cdot_{A}}\Big(  \big(r+\tau(r)\big)^{\sharp}(b^{*})\Big)a^{*}\bigg)\\
        &=&-r^{\sharp}(a^{*}\lozenge_{r+\tau(r)}b^{*}).
    \end{eqnarray*}
    Hence \eqref{eq:rb1} holds if and only if \eqref{eq:rt oop1} holds.
    Moreover, noticing that $\mathcal{B}^{\natural}\hat{Q}=Q^{*}\mathcal{B}^{\natural}$, we have
    \begin{equation*}
        (PR-R\hat{Q})x=Pr^{\sharp}\mathcal{B}^{\natural}(x)-r^{\sharp}\mathcal{B}^{\natural}\hat{Q}(x)=Pr^{\sharp}\mathcal{B}^{\natural}(x)-r^{\sharp}Q^{*}\mathcal{B}^{\natural}(x)=(Pr^{\sharp}-r^{\sharp}Q^{*})a^{*}.
    \end{equation*}
    Hence  \eqref{eq:rb2} holds if and only if \eqref{eq:rt oop2} holds, and similarly  \eqref{eq:rb3} holds if and only if \eqref{eq:rt oop3} holds.
    Therefore the conclusion follows from Theorem \ref{thm:rt oop}.
\end{proof}

\subsection{Triangular averaging commutative and cocommutative infinitesimal bialgebras}\

Let $(A,\cdot_{A},P,Q)$ be an admissible averaging commutative
algebra and 
  $r$ be a skew-symmetric
solution of the AAYBE in $(A,\cdot_{A}$,
$P,Q)$. Then by
Theorem \ref{thm:coao}, $(A,\cdot_{A},\Delta_{r}$,
$P$,
$Q)$ is an averaging commutative and cocommutative infinitesimal bialgebra, where $\Delta_{r}$ is given by \eqref{eq:comul1}.
In this case, we say $(A,\cdot_{A},\Delta_{r},P,Q)$ is {\bf triangular}.

\begin{pro}
    Let $(A,\cdot_{A},P,Q)$ be an admissible averaging commutative  algebra and $r\in A\otimes A$ be skew-symmetric.
    Then \eqref{eq:rt oop2} holds if and only if \eqref{eq:rt oop3} holds.
    Moreover,
     $r$ is a solution of the AAYBE in $(A,\cdot_{A},P,Q)$ if and only if $r^{\sharp}$ is an $\mathcal{O}$-operator of $(A,\cdot_{A},P,Q)$
    associated to $(\mathcal{L}^{*}_{\cdot_{A}},Q^{*},P^{*},A^{*})$.
\end{pro}
\begin{proof}
For all $a^{*}, b^{*}\in A^{*}$, we have
\begin{eqnarray*}
    \langle (Pr^{\sharp}-r^{\sharp}Q^{*})a^{*},b^{*}\rangle&=&\langle r^{\sharp}(a^{*}),P^{*}(b^{*})\rangle-\langle r^{\sharp}Q^{*}(a^{*}),b^{*}\rangle\\
    &=&\langle r, a^{*}\otimes P^{*}(b^{*})-Q^{*}(a^{*})\otimes b^{*}\rangle\\
    &=&\langle r, b^{*}\otimes Q^{*}(a^{*})-P^{*}(b^{*})\otimes a^{*}\rangle\\
    &=&\langle (Qr^{\sharp}-r^{\sharp}P^{*})b^{*},a^{*}\rangle.
\end{eqnarray*}
Hence the first half part holds.
    The second half part  follows from Theorem \ref{thm:rt oop} by observing that $r+\tau(r)=0$.
\end{proof}

Let $A$ be a vector space with a multiplication $\star_{A}:A\otimes A\rightarrow A$.
If  $R:A\rightarrow A$ is a linear map satisfying
\begin{equation}
R(x)\star_{A} R(y)=R\big(R(x)\star_{A} y+x\star_{A}R(y)+\lambda x\star_{A}y\big),\;\forall x,y\in A,
\end{equation}
then we say $R$ is a {\bf Rota-Baxter operator on $(A,\star_{A})$ of weight $\lambda$}.

\begin{defi} \cite{SW}
Let $(A,\cdot_{A},\mathcal{B})$ be a symmetric  Frobenius commutative algebra.
If there is a Rota-Baxter operator $R$ on $(A,\cdot_{A})$ of weight $\lambda$ such that
\begin{equation}\label{strong}
\mathcal{B}\big(R(x),y\big)+\mathcal{B}\big(x,R(y)\big)+\lambda\mathcal{B}(x,y)=0,\;\forall x,y\in A,
\end{equation}
then we say $(A,\cdot_{A},R,\mathcal{B})$ is a {\bf symmetric Rota-Baxter Frobenius commutative algebra of weight $\lambda$}.
\end{defi}

\begin{pro}\label{pro:3.16}
Let $R$ be a Rota-Baxter operator of weight $\lambda$ on a commutative  algebra $(A,\cdot_{A})$. Then
\begin{equation*}
    \big(A\ltimes_{\mathcal{L}_{\cdot_{A}}^{*}}A^{*},\mathcal{B}_{d}, R-(R+\lambda\mathrm{id}_{A})^{*}\big)
\end{equation*}
is a symmetric Rota-Baxter Frobenius commutative algebra of weight $\lambda$, where   $\mathcal{B}_{d}$ is the natural nondegenerate symmetric bilinear form on $A\oplus A^{*}$  given by
\begin{equation}\label{eq:bfds}
	\mathcal{B}_{d}(x+a^{*},y+b^{*})=\langle x,b^{*}\rangle+\langle a^{*},y\rangle,\;\forall x,y\in A, a^{*},b^{*}\in A^{*}.
\end{equation}
\end{pro}
\begin{proof}
    It follows from a straightforward computation.
\end{proof}

\begin{pro}\label{pro:317}
Let $(A,\cdot_{A},R,\mathcal{B})$ be a symmetric Rota-Baxter Frobenius commutative algebra of weight $0$.
Suppose that $(A,\cdot_{A},P,Q)$ is an admissible averaging commutative  algebra satisfying \eqref{eq:rb2} and \eqref{eq:rb3}.
Then there is a triangular averaging commutative and cocommutative infinitesimal bialgebra $(A,\cdot_{A},\Delta_{r},P,Q)$,
where $r\in A\otimes A$ is given through the operator form $r^{\sharp}:A^{*}\rightarrow A$ by \eqref{Br,Brt}, that is,
\begin{equation}\label{eq:thm:quadratic to fact}
r^{\sharp}(a^{*})=R\mathcal{B}^{{\natural}^{-1}}(a^{*}),\;\;\forall a^{*}\in A^{*}.
\end{equation}
\end{pro}
\begin{proof}
Let $x,y\in A$ and $a^{*}=\mathcal{B}^{\natural}(x), b^{*}=\mathcal{B}^{\natural}(y)$.
Then we have
\begin{eqnarray*}
  \langle r+\tau(r), a^{*}\otimes b^{*}\rangle&=&\langle r^{\sharp}(a^{*}),b^{*}\rangle+\langle r^{\sharp}(b^{*}), a^{*}\rangle\\
    &=&\langle R(x),\mathcal{B}^{\natural}(y)\rangle+\langle R(y),\mathcal{B}^{\natural}(x)\rangle\\
    &=&\mathcal{B}\big(R(x),y\big)+\mathcal{B}\big(x,R(y)\big)\\
    &=&0.
\end{eqnarray*}
Hence $r$ is skew-symmetric and the conclusion follows from Proposition \ref{pro:3.14}.
\end{proof}

\begin{defi}
Let	$(A,\cdot_{A},R,\mathcal{B})$ be a  symmetric Rota-Baxter Frobenius commutative algebra of weight $\lambda$. 
Suppose that $P$ is an averaging operator of $(A,\cdot_{A})$ which satisfies the commutativity condition
\begin{equation}\label{eq:commute}
	PR=RP.
\end{equation}
Then we say $(A,\cdot_{A},P,R,\mathcal{B})$ is a {\bf symmetric averaging Rota-Baxter Frobenius commutative algebra of weight $\lambda$}.
\end{defi}

Next we explore the relationship between symmetric averaging Rota-Baxter Frobenius commutative algebras of weight $0$ and triangular averaging commutative and cocommutative infinitesimal bialgebras.

\begin{cor}\label{cor:2.24}
Let $(A,\cdot_{A},P,R,\mathcal{B})$ be a symmetric averaging  Rota-Baxter Frobenius commutative algebra of weight $0$.
Then there is a triangular averaging commutative and cocommutative infinitesimal bialgebra $(A,\cdot_{A}$,
$\Delta_{r},$
$P,\hat{P})$ where $r\in A\otimes A$ is given through the operator form $r^{\sharp}:A^{*}\rightarrow A$ by  \eqref{eq:thm:quadratic to fact}.
\end{cor}
\begin{proof}
By \cite{Bai2024}, $(A,\cdot_{A},P,\hat{P})$ is an admissible averaging commutative  algebra. For all $x,y\in A$, we have
    \begin{eqnarray*}
        &&\mathcal{B}\big(x,R\hat{P}(y)\big)\overset{\eqref{strong}}{=}-\mathcal{B}\big(R(x),\hat{P}(y)\big) =-\mathcal{B}\big(PR(x),y\big) \\
        &&=-\mathcal{B}\big(RP(x),y\big) \overset{\eqref{strong}}{=}\mathcal{B}\big(P(x),R(y)\big)=\mathcal{B}\big(x,\hat{P}R(y)\big).
    \end{eqnarray*}
    That is, $R\hat{P}=\hat{P}R$.
    Hence \eqref{eq:rb2} and \eqref{eq:rb3} holds for $Q=\hat{P}$.
    Therefore the conclusion follows from Proposition \ref{pro:317} by taking $Q=\hat{P}$.
\end{proof}

\begin{thm}\label{thm:3.16}
    Let $(A,\cdot_{A},P,Q)$ be an admissible averaging commutative  algebra and $(\mu,\alpha,\beta,V)$ be a representation of $(A,\cdot_{A},P,Q)$.
    Suppose that $T:V\rightarrow A$ is a linear map which is identified as
    \begin{equation*}
        T_{\sharp}\in V^{*}\otimes A\subset (A\ltimes_{\mu^{*}}V^{*})\otimes (A\ltimes_{\mu^{*}}V^{*}).
    \end{equation*}
    Then $r=T_{\sharp}-\tau(T_{\sharp})$ is a skew-symmetric solution of the AAYBE in $( A\ltimes_{\mu^{*}}V^{*},P+\beta^{*},Q+\alpha^{*} )$
    if and only if
    $T$ is an $\mathcal{O}$-operator of $(A,\cdot_{A},P,Q)$ associated to  $(\mu,\alpha,\beta,V)$.
\end{thm}
\begin{proof}
By \cite{Bai2010}, $r$ is a solution of the AYBE in $A\ltimes_{\mu^{*}}V^{*}$ if and only if the following equation holds:
\begin{equation*}
   Tu\cdot_{A}Tv=T\big(\mu(Tu)v+\mu(Tv)u\big),\;\forall u,v\in V.
\end{equation*}
By \cite{BGM}, $r$ satisfies
\begin{equation*}
    \big( (P+\beta^{*})\otimes\mathrm{id}\big)r=\big( \mathrm{id}\otimes(Q+\alpha^{*}) \big)r
\end{equation*}
if and only if \eqref{eq:oop2} and \eqref{eq:oop3} hold.
Hence the conclusion follows.
\end{proof}

Recall a \textbf{Zinbiel algebra} \cite{Lod3}  is a  vector space $A$ together with the  multiplication $\star_{A}:A\otimes A\rightarrow A$ such that the following equation holds:
    \begin{equation}\label{eq:Zinb}
        x\star_{A}(y\star_{A} z)=(x\star_{A} y)\star_{A} z+(y\star_{A} x)\star_{A} z, \;\;\forall x,y,z\in A.
    \end{equation}
    There is consequently a commutative  algebra $(A,\cdot_{A})$ with a multiplication $\cdot_{A}:A\otimes A\rightarrow A$ defined by
    \begin{equation*}
        x\cdot_{A}y=x\star_{A}y+y\star_{A}x
    \end{equation*}
    which is called the {\bf sub-adjacent commutative  algebra} of $(A,\star_{A})$. Moreover, $(\mathcal{L}_{\star_{A}},A)$ is a representation of $(A,\cdot_{A})$.

\begin{defi}
   Let $P:A\rightarrow A$ be an averaging operator of a Zinbiel algebra $(A,\star_{A})$, that is, the following equation holds:
\begin{equation*}
    P(x)\star_{A}P(y)=P\big(P(x)\star_{A} y\big)=P\big(x\star_{A}P(y)\big),\;\forall x,y\in A.
\end{equation*}
Suppose that $Q:A\rightarrow A$ is a linear map satisfying the following equations:
\begin{eqnarray}
&&Q\big(P(x)\star_{A} y\big)=P(x)\star_{A}Q(y)=Q\big(x\star_{A}Q(y)\big),\\
&&Q\big(Q(x)\star_{A} y\big)=Q(x)\star_{A}P(y)=Q\big(x\star_{A}P(y)\big).
\end{eqnarray}
Then we say $(A,\star_{A},P,Q)$ is an {\bf admissible averaging Zinbiel algebra}.
\end{defi}

\begin{lem}\label{lem:3.18}
    Let $(A,\star_{A},P,Q)$ be an admissible averaging Zinbiel algebra and $(A,\cdot_{A})$ be the sub-adjacent commutative  algebra of $(A,\star_{A})$. Then $(A,\cdot_{A},P,Q)$ is an admissible averaging commutative  algebra.
    Moreover, $(\mathcal{L}_{\star_{A}},P,Q,A)$ is a representation of $(A,\cdot_{A},P,Q)$ and the identity map $\mathrm{id}:A\rightarrow A$ is an $\mathcal{O}$-operator of $(A,\cdot_{A},P,Q)$ associated to  $(\mathcal{L}_{\star_{A}},P,Q,A)$.
\end{lem}
\begin{proof}
    It is straightforward.
\end{proof}

\begin{pro}\label{pro:2.28}
Suppose that $(A,\star_{A},P,Q)$ is an admissible averaging Zinbiel algebra and $(A,\cdot_{A})$ is the sub-adjacent commutative  algebra of $(A,\star_{A})$.
Let $\{e_{1},\cdots, e_{n}\}$ be a basis of $A$ and $\{e^{*}_{1},\cdots, e^{*}_{n}\}$ be the dual basis.
Then
\begin{equation}\label{1109}
    r=\sum_{i=1}^{n}e^{*}_{i}\otimes e_{i}-e_{i}\otimes e^{*}_{i}
\end{equation}
is a skew-symmetric solution of the AAYBE in $(A\ltimes_{\mathcal{L}^{*}_{\star_{A}}}A^{*}, P+Q^{*},Q+P^{*})$.
Therefore there is a triangular averaging commutative and cocommutative infinitesimal bialgebra
$$ (A\ltimes_{\mathcal{L}^{*}_{\star_{A}}}A^{*} ,\Delta_{r}, P+Q^{*},Q+P^{*}),  $$
where the linear map $\Delta_{r}$ is given by \eqref{eq:comul1}.
\end{pro}
\begin{proof}
    By Lemma \ref{lem:3.18}, $(A,\cdot_{A},P,Q)$ is an admissible averaging commutative  algebra and $(\mathcal{L}_{\star_{A}},P,Q,A)$ is a representation of $(A,\cdot_{A},P,Q)$.
    Moreover, the identity map $\mathrm{id}:A\rightarrow A$ is an $\mathcal{O}$-operator of $(A,\cdot_{A},P,Q)$ associated to $(\mathcal{L}_{\star_{A}},P,Q,A)$.
    Hence by Theorem \ref{thm:3.16},
    \begin{eqnarray}\label{1121}
    r=\mathrm{id}_{\sharp}-\tau(\mathrm{id}_{\sharp})=\sum\limits_{i=1}^{n}e^{*}_{i}\otimes e_{i}-e_{i}\otimes e^{*}_{i}
    \end{eqnarray}
    is a skew-symmetric solution of the AAYBE in $(A\ltimes_{\mathcal{L}^{*}_{\star_{A}}}A^{*},P+Q^{*},Q+P^{*})$.
\end{proof}

\subsection{Factorizable averaging commutative and cocommutative infinitesimal bialgebras}\

Let $(A,\cdot_{A},\Delta_{r})$ be a quasi-triangular commutative and cocommutative infinitesimal bialgebra.
If $\big(r+\tau(r)\big)^{\sharp}:A^{*}\rightarrow A$ is a bijection, then we say  $(A,\cdot_{A},\Delta_{r})$ is {\bf factorizable} \cite{Bai2011,SW}.
Now we generalize the above notion to averaging commutative and cocommutative infinitesimal bialgebras.

\begin{defi}
A {\bf factorizable averaging commutative and cocommutative infinitesimal bialgebra} is a quasi-triangular averaging commutative and cocommutative infinitesimal bialgebra $(A,\cdot_{A},\Delta_{r},P,Q)$ such that $\big( r+\tau(r)\big)^{\sharp}$ is bijective.
\end{defi}


\begin{defi}
Let $(A,\cdot_{A},P)$ and $(A',\cdot_{A'},P')$ be two averaging commutative algebras. A linear map $\psi:A\rightarrow A'$ is called {\bf an isomorphism of averaging commutative algebras}, if $\psi$ is a linear 
isomorphism of vector spaces such that 
\begin{equation*}
\psi(x\cdot_{A}y)=\psi(x)\cdot_{A'}\psi(y),\;
\psi\big(P(x)\big)=P'\big(\psi(x)\big),\;\forall x,y\in A.
\end{equation*}
Moreover, let $(A,\cdot_{A},P,Q)$ and $(A',\cdot_{A'},P',Q')$ be two admissible averaging commutative algebras. A linear map $\psi:A\rightarrow A'$ is called {\bf an isomorphism of admissible averaging commutative algebras}, if
$\psi$ is an isomorphism of averaging commutative algebras such that
\begin{equation*}
\psi\big(Q(x)\big)=Q'\big(\psi(x)\big),\;\forall x\in A.
\end{equation*}
\end{defi}

\begin{defi}\label{defi:aFca}\cite{Bai2024}
Let $\big( (A\oplus
A^{*},\cdot_{d},\mathcal{B}_{d}),(A,\cdot_{A}),(A^{*},\cdot_{A^{*}})
\big)$ be a double construction of Frobenius commutative algebra \cite{Bai2010}, that is, there exists a commutative algebra structure
$(A\oplus A^{*},\cdot_{d})$ on $A\oplus A^{*}$ such that it contains $(A,\cdot_{A})$ and $(A^{*},\cdot_{A^{*}})$ as commutative subalgebras and the bilinear form $\mathcal{B}_{d}$ given by \eqref{eq:bfds} is invariant on $(A\oplus A^{*},\cdot_{d})$.
Let $P:A\rightarrow A$ be an averaging operator on
$(A,\cdot_{A})$ and $Q^{*}:A^{*}\rightarrow A^{*}$ be an averaging
operator on $(A^{*},\cdot_{A^{*}})$. If $P+Q^{*}$ is an averaging
operator on $(A\oplus A^{*},\cdot_{d})$, then we call $\big(
(A\oplus A^{*},\cdot_{d},P+Q^{*},\mathcal{B}_{d}), (A,\cdot_{A},P),(A^{*},\cdot_{A^{*}},Q^{*})
\big)$ a {\bf double construction of an averaging Frobenius
commutative algebra}. 
\end{defi}

\begin{lem}\label{lem:2.11}\cite{Bai2024}
	Let $(A,\cdot_{A},P)$ and $(A^{*},\cdot_{A^{*}},Q^{*})$ be averaging commutative algebras and $\Delta:A\rightarrow A\otimes A$ be the linear dual of $\cdot_{A^{*}}$, that is,\begin{eqnarray}
		\langle \Delta(x),a^{*}\otimes b^{*}\rangle=\langle x, a^{*}\cdot_{A^{*}}b^{*}\rangle,\;\forall x\in A, a^{*},b^{*}\in A^{*}.
	\end{eqnarray}
Then there is a double construction of an averaging Frobenius commutative algebra
$\big( (A\oplus A^{*},\cdot_{d},P+Q^{*},\mathcal{B}_{d}),(A,\cdot_{A},P),(A^{*},\cdot_{A^{*}},Q^{*}) \big)$  if and only if $(A,\cdot_{A},\Delta,P,Q)$ is an averaging commutative and cocommutative infinitesimal bialgebra. In this case, 
$(A,\cdot_{A},P,Q)$ and $(A^{*},\cdot_{A^{*}},Q^{*},P^{*})$ are both admissible averaging commutative algebras, and 
we have
\begin{equation}\label{eq:commassomul}
(x+a^{*})\cdot_{d}(y+b^{*})=x\cdot_{A}y+\mathcal{L}^{*}_{\cdot_{A^{*}}}(b^{*})x
+\mathcal{L}^{*}_{\cdot_{A^{*}}}(a^{*})y+
a^{*}\cdot_{A^{*}}b^{*}+\mathcal{L}^{*}_{\cdot_{A}}(y)a^{*}
+\mathcal{L}^{*}_{\cdot_{A}}(x)b^{*}
\end{equation}
for all $x,y\in A, a^{*},b^{*}\in A^{*}$.
\end{lem}

We have the following proposition which justifies the terminology of factorizable averaging commutative and cocommutative infinitesimal bialgebras.

\begin{pro}
Let $(A,\cdot_{A},\Delta_{r},P,Q)$ be a factorizable averaging commutative and cocommutative infinitesimal bialgebra, and $\big( (D=A\oplus A^{*},\cdot_{d},P+Q^{*},\mathcal{B}_{d}),(A,\cdot_{A},P),
(A^{*},\cdot_{r}$,
$Q^{*}) \big)$ be a double construction of averaging Frobenius commutative algebra which is equivalent to $(A,\cdot_{A},\Delta_{r},P,Q)$. Define a linear map $\psi:D=A\oplus A^{*}\rightarrow A\oplus A$  by
\begin{equation}\label{eq:factorizable}
	\psi(x)=(x,x), \;\psi(a^{*})=\Big(r^{\sharp}(a^{*}),\big(-\tau(r)\big)^{\sharp}(a^{*})\Big),\;\forall x\in A, a^{*}\in A^{*}.
\end{equation}
Then $\psi$ gives the admissible averaging commutative algebra isomorphism between $(D,\cdot_{d},P+Q^{*},Q+P^{*})$ and the direct sum $A\oplus A$ of admissible averaging commutative algebras. In particular, $\psi|_{A^{*}}$ gives the admissible averaging commutative algebra isomorphism between $(A^{*},\cdot_{r},Q^{*})$ and $\mathrm{Im}(r^{\sharp}\oplus \big(-\tau(r)\big)^{\sharp})$ as an admissible averaging commutative subalgebra of $A\oplus A$. Moreover, for any $x\in A$, there is a unique decomposition
	$x=x_{1}-x_{2},$
	where $(x_{1},x_{2})\in\mathrm{Im}(r^{\sharp}\oplus \big(-\tau(r)\big)^{\sharp}).$
\end{pro}

\begin{proof}
By \cite{SW}, $\psi|_{A^{*}}$ gives the commutative algebra isomorphism between $(A^{*},\cdot_{r})$ and $\mathrm{Im}(r^{\sharp}\oplus \big(-\tau(r)\big)^{\sharp})$ as a commutative subalgebra of $A\oplus A$, that is, we have
\begin{eqnarray*}
\psi(a^{*}\cdot_{r}b^{*})=\psi(a^{*})\cdot\psi(b^{*}),\;\forall a^{*},b^{*}\in A^{*},
\end{eqnarray*}
where $\cdot$ denotes the commutative multiplication on $A\oplus A$.
For all $x\in A,a^{*},b^{*}\in A^{*}$,
\begin{eqnarray*}
\langle r^{\sharp}\big(\mathcal{L}^{*}_{\cdot_{A}}(x)a^{*}\big)
+\mathcal{L}^{*}_{\cdot_{r}}(a^{*})x,b^{*}\rangle
&=&\langle r, \mathcal{L}^{*}_{\cdot_{A}}(x)a^{*}\otimes b^{*}\rangle+\langle x, a^{*}\cdot_{r}b^{*}\rangle\\
&=&\langle \big(\mathcal{L}_{\cdot_{A}}(x)\otimes\mathrm{id}\big)r,a^{*}\otimes b^{*}\rangle+\langle \Delta_{r}(x), a^{*}\otimes b^{*}\rangle\\
&=&\langle 
\big(\mathrm{id}\otimes\mathcal{L}_{\cdot_{A}}(x)\big)r,a^{*}\otimes b^{*}\rangle\\
&=&\langle x\cdot_{A}r^{\sharp}(a^{*}), b^{*}\rangle,
\end{eqnarray*}
that is, 
\begin{equation}\label{eq:tau3} 
	r^{\sharp}\big(\mathcal{L}^{*}_{\cdot_{A}}(x)a^{*}\big)
	+\mathcal{L}^{*}_{\cdot_{r}}(a^{*})x=x\cdot_{A}r^{\sharp}(a^{*}).
\end{equation}
Similarly, we have
\begin{equation}\label{eq:tau4} 
	\big(-\tau(r)\big)^{\sharp}\big(\mathcal{L}^{*}_{\cdot_{A}}(x)a^{*}\big)
	+\mathcal{L}^{*}_{\cdot_{r}}(a^{*})x=x\cdot_{A}\big(-\tau(r)\big)^{\sharp}(a^{*}).
\end{equation}
Thus
\begin{eqnarray*}
	\psi(x\cdot_{D} a^{*})&=&\psi\big(\mathcal{L}^{*}_{\cdot_{A}}(x)a^{*}
	+\mathcal{L}^{*}_{\cdot_{r}}(a^{*})x\big)\\	&=&\Big(r^{\sharp}\big(\mathcal{L}^{*}_{\cdot_{A}}(x)a^{*}\big)
	+\mathcal{L}^{*}_{\cdot_{r}}(a^{*})x,
	\big(-\tau(r)\big)^{\sharp}\big(\mathcal{L}^{*}_{\cdot_{A}}(x)a^{*}\big)
	+\mathcal{L}^{*}_{\cdot_{r}}(a^{*})x\Big)\\ 
	&\overset{\eqref{eq:tau3},\eqref{eq:tau4}}{=}&
	\big(x\cdot_{A}r^{\sharp}(a^{*}), x\cdot_{A}\big(-\tau(r)\big)^{\sharp}(a^{*})\big)\\
	&=&\psi(x)\cdot \psi(a^{*}).
\end{eqnarray*}
Hence $\psi:D\rightarrow A\oplus A$ is an isomorphism of commutative algebras.
Noticing that \eqref{eq:AAYBE1} gives rise to
\begin{equation}\label{eq:taur1}
P\big(-\tau(r)\big)^{\sharp}(a^{*})=\big(-\tau(r)\big)^{\sharp}Q^{*}(a^{*}),
\end{equation}
we have
\begin{eqnarray*}
\psi|_{A^{*}}Q^{*}(a^{*})
&=&\Big(r^{\sharp}Q^{*}(a^{*}),
\big(-\tau(r)\big)^{\sharp}Q^{*}(a^{*})\Big)\\
&\overset{\eqref{eq:rt oop2},\eqref{eq:taur1}}{=}&
\Big(Pr^{\sharp}(a^{*}),
P\big(-\tau(r)\big)^{\sharp}(a^{*})\Big)\\
&=&(P\oplus P)\psi|_{A^{*}}(a^{*}).
\end{eqnarray*}
Similarly, we obtain $\psi|_{A^{*}}P^{*}(a^{*})=(Q\oplus Q)\psi|_{A^{*}}(a^{*})$, and more generally
\begin{equation}
\psi(P+Q^{*})=(P\oplus P)\psi,\;
\psi(Q+P^{*})=(Q\oplus Q)\psi.
\end{equation}
Therefore,  $\psi:D\rightarrow A\oplus A$ is an admissible averaging commutative algebra isomorphism, and $\mathrm{Im}(r^{\sharp}\oplus \big(-\tau(r)\big)^{\sharp})$ is isomorphic to $(A^{*},\cdot_{r},Q^*)$ as admissible averaging commutative subalgebras.
The last part also follows from \cite{SW}. Hence the proof is finished.
\delete{
We have
\begin{eqnarray*}
&&\psi(P+Q^{*})(x+a^{*})=\psi\big(P(x)+Q^{*}(a^{*})\big)
=\Big(P(x)+r^{\sharp}Q^{*}(a^{*}),
P(x)+\big(-\tau(r)\big)^{\sharp}Q^{*}(a^{*})\Big),\\
&&(P\oplus P)\psi(x+a^{*})=(P\oplus P)\Big(x+r^{\sharp}(a^{*}),x+\big(-\tau(r)\big)^{\sharp}(a^{*})\Big)
=\Big(P(x)+Pr^{\sharp}(a^{*}),
P(x)+P\big(-\tau(r)\big)^{\sharp}(a^{*})\Big).
\end{eqnarray*}
Thus by \eqref{eq:rt oop2} and \eqref{eq:taur1}, we obtain
\begin{equation}
\psi(P+Q^{*})=(P\oplus P)\psi.
\end{equation}
Similarly, we also obtain
\begin{equation}
\psi(Q+P^{*})=(Q\oplus Q)\psi.
\end{equation}}
\end{proof}

\delete{
	\delete{
		For all $a^{*},b^{*}\in A^{*}$, we have
		\begin{eqnarray*}
			0&=&\langle (P\otimes\mathrm{id}-\mathrm{id}\otimes Q)r,a^{*}\otimes b^{*}\rangle\\
			&=&\langle \tau(P\otimes\mathrm{id}-\mathrm{id}\otimes Q)r,a^{*}\otimes b^{*}\rangle\\
			&=&\langle (\mathrm{id}\otimes P-Q\otimes\mathrm{id})\tau(r),a^{*}\otimes b^{*}\rangle\\
			&=&\langle \tau(r),a^{*}\otimes P^*(b^{*})-Q^*(a^{*})\otimes b^{*}\rangle\\
			&=&\langle P\big(\tau(r)\big)^{\sharp}(a^{*})-\big(\tau(r)\big)^{\sharp}Q^*(a^{*}) ,b^{*}\rangle,
		\end{eqnarray*}
		that is,}
	Similarly, we also have 
	\begin{equation}\label{eq:taur2}
		Q\big(-\tau(r)\big)^{\sharp}(a^{*})=\big(-\tau(r)\big)^{\sharp}P^{*}(a^{*}).
\end{equation}}

In the following, we show that there is a factorizable averaging commutative and cocommutative infinitesimal bialgebra structure on an arbitrary double construction of averaging Frobenius commutative algebra.

\begin{thm}
Let $\big( (A\oplus A^{*},\cdot_{d},P+Q^{*},\mathcal{B}_{d}),(A,\cdot_{A},P),(A^{*},\cdot_{A^{*}},Q^{*}) \big)$ be a double construction of averaging Frobenius commutative algebra. Suppose that
	$\{e_{1}$, $\cdots$, $e_{n}\}$ is a basis of $A$ and  $\{e^{*}_{1},\cdots,e^{*}_{n}\}$ is the dual basis. Set 
	\begin{eqnarray}\label{1306}
	r=\sum\limits_{i=1}^{n}e^{*}_{i}\otimes e_{i}\in A^{*}\otimes A\subset D\otimes D.
	\end{eqnarray}
Then $(D,\cdot_{d},\Delta_{r},P+Q^*,Q+P^*)$ with $\Delta_{r}$ defined by \eqref{eq:comul1} is a factorizable averaging commutative and cocommutative infinitesimal bialgebra.
\end{thm}

\begin{proof}
By \cite{SW}, $(D,\cdot_{d},\Delta_{r})$ is a factorizable commutative and cocommutative infinitesimal bialgebra. 
Moreover, we have
\begin{equation*}
\big((P+Q^*)\otimes\mathrm{id}-\mathrm{id}\otimes (Q+P^*)\big)r=\sum\limits_{i=1}^{n}Q^*(e^{*}_{i})\otimes e_{i}-e^{*}_{i}\otimes Q(e_{i})=0,
\end{equation*}
and similarly
\begin{equation*}
	\big((Q+P^*)\otimes\mathrm{id}-\mathrm{id}\otimes (P+Q^*)\big)r=0.
\end{equation*}
Hence $(D,\cdot_{d},\Delta_{r},P+Q^*,Q+P^*)$ is a factorizable averaging commutative and cocommutative infinitesimal bialgebra.
\end{proof}

\delete{
	Similarly, we also have
	\delete{Let $k,l\in\{1,\cdots, n\}$, then we have
		\begin{eqnarray*}
			&&\langle \big((P+Q^*)\otimes\mathrm{id}-\mathrm{id}\otimes (Q+P^*)\big)r,
			e_{k}\otimes e^{*}_{l}\rangle\\
			&=&\langle \sum\limits_{i=1}^{n}Q^*(e^{*}_{i})\otimes e_{i}-e^{*}_{i}\otimes Q(e_{i}), e_{k}\otimes e^{*}_{l}\rangle\\
			&=&\langle Q^*(e^{*}_{k})\otimes e_{k}-e^{*}_{k}\otimes Q(e_{k}), e_{k}\otimes e^{*}_{l}\rangle+\langle Q^*(e^{*}_{l})\otimes e_{l}-e^{*}_{l}\otimes Q(e_{l}), e_{k}\otimes e^{*}_{l}\rangle\\
			&=&\langle Q^*(e^{*}_{k}),e_{k}\rangle\langle e_{k},e^{*}_{l}\rangle
			-\langle Q(e_{l}),e^{*}_{l}\rangle\langle e_{k},e^{*}_{l}\rangle\\
			&=&0.
		\end{eqnarray*}
		Thus we have }
	\begin{eqnarray*}
		&&\langle \big((Q+P^*)\otimes\mathrm{id}-\mathrm{id}\otimes (P+Q^*)\big)r,
		e_{k}\otimes e^{*}_{l}\rangle\\
		&=&\langle \sum_{i}P^*(e^{*}_{i})\otimes e_{i}-e^{*}_{i}\otimes P(e_{i}), e_{k}\otimes e^{*}_{l}\rangle\\
		&=&\langle P^*(e^{*}_{l})\otimes e_{l}-e^{*}_{l}\otimes P(e_{l}), e_{k}\otimes e^{*}_{l}\rangle+\langle P^*(e^{*}_{k})\otimes e_{k}-e^{*}_{k}\otimes P(e_{k}), e_{k}\otimes e^{*}_{l}\rangle\\
		&=&\langle P^*(e^{*}_{k}),e_{k}\rangle\langle e_{k},e^{*}_{l}\rangle
		-\langle P(e_{l}),e^{*}_{l}\rangle\langle e_{k},e^{*}_{l}\rangle\\
		&=&0.
\end{eqnarray*}}

\begin{lem}\label{lem:3.23}\cite{SW}
Let $(A,\cdot_{A},R,\mathcal{B})$ be a symmetric Rota-Baxter Frobenius commutative algebra of weight $-1$.
Then there is a factorizable commutative and cocommutative infinitesimal bialgebra $(A,\cdot_{A},\Delta_{r})$ with $r$ given  through the operator form $r^{\sharp}$ by \eqref{eq:thm:quadratic to fact}.
Conversely, let $(A,\cdot_{A},\Delta_{r})$ be a factorizable commutative and cocommutative infinitesimal bialgebra.
Then there is a symmetric Rota-Baxter Frobenius commutative algebra $(A,\cdot_{A},R,\mathcal{B} )$ of weight $-1$ with $R$ given by
\begin{equation}\label{eq:fact to quadratic0}
    R(x)=r^{\sharp}\big(r+\tau(r)\big)^{\sharp^{-1}}(x)
\end{equation}
and $\mathcal{B}$ given by
\begin{equation}\label{eq:fact to quadratic}
\mathcal{B}(x,y)=\langle {\big(r+\tau(r)\big)^{\sharp}}^{-1}(x),y\rangle,\; \forall x,y\in A.
\end{equation}
\end{lem}

\begin{thm}\label{thm:commute}
Let $(A,\cdot_{A},P,R,\mathcal{B})$ be a symmetric averaging Rota-Baxter Frobenius commutative algebra of weight $-1$.
Then there is a factorizable averaging commutative and cocommutative infinitesimal bialgebra $(A,\cdot_{A},\Delta_{r},P,\hat{P})$
with $r$ given  through the operator form $r^{\sharp}$ by \eqref{eq:thm:quadratic to fact}.
Conversely, let $(A,\cdot_{A},\Delta_{r},P,Q)$ be a factorizable averaging commutative and cocommutative infinitesimal bialgebra.
Then there is a symmetric averaging Rota-Baxter Frobenius commutative algebra $(A,\cdot_{A},P,R,\mathcal{B})$  of weight $-1$
given by \eqref{eq:fact to quadratic0} and \eqref{eq:fact to quadratic}
such that  $Q=\hat{P}$.
\end{thm}
\begin{proof}
    Suppose that $(A,\cdot_{A},P,R,\mathcal{B})$ is a symmetric averaging Rota-Baxter Frobenius commutative algebra of weight $-1$.
    Then by \cite{Bai2024}, $(A,\cdot_{A},P,\hat{P})$ is an admissible averaging commutative  algebra.
    For all $x,y\in A$, we have
    \begin{eqnarray*}
        &&\mathcal{B}\big(x,R\hat{P}(y)\big)\overset{\eqref{strong}}{=}-\mathcal{B}\big(R(x),\hat{P}(y)\big)+\mathcal{B}\big(x,\hat{P}(y)\big)=-\mathcal{B}\big(PR(x),y\big)+\mathcal{B}\big(P(x),y\big)\\
        &&=-\mathcal{B}\big(RP(x),y\big)+\mathcal{B}\big(P(x),y\big)\overset{\eqref{strong}}{=}\mathcal{B}\big(P(x),R(y)\big)=\mathcal{B}\big(x,\hat{P}R(y)\big).
    \end{eqnarray*}
    That is, $R\hat{P}=\hat{P}R$.
    Hence \eqref{eq:rb2} and \eqref{eq:rb3} hold  for $Q=\hat{P}$.
    Thus by Proposition \ref{pro:3.14} and Lemma \ref{lem:3.23},
    $(A,\cdot_{A},\Delta_{r},P,\hat{P})$  is a factorizable averaging commutative and cocommutative infinitesimal bialgebra.

    Conversely, let $(A,\cdot_{A},\Delta_{r},P,Q)$ be a factorizable averaging commutative and cocommutative infinitesimal bialgebra, and $(A,\cdot_{A},R,\mathcal{B})$ be the corresponding symmetric Rota-Baxter Frobenius commutative algebra of weight $-1$
given by \eqref{eq:fact to quadratic0} and \eqref{eq:fact to quadratic}.
Let $a^{*},b^{*}\in A^{*}$ and $x=\big(r+\tau(r)\big)^{\sharp}a^{*}, y=\big(r+\tau(r)\big)^{\sharp}b^{*}$.
Then we have
\begin{eqnarray*}
&&\mathcal{B}\big(P(x),y\big)=\langle\mathcal{B}^{\natural}(y),P(x)\rangle=\langle b^{*},P(x)\rangle=\langle P^{*}(b^{*}), x\rangle\\
&&=\langle P^{*}(b^{*}), \big(r+\tau(r)\big)^{\sharp}a^{*}\rangle=\langle r+\tau(r), a^{*}\otimes P^{*}(b^{*})\rangle=\langle (\mathrm{id}\otimes P)\big(r+\tau(r)\big), a^{*}\otimes b^{*}\rangle\\
&&=\langle (Q\otimes\mathrm{id})\big(r+\tau(r)\big), a^{*}\otimes b^{*}\rangle=\langle r+\tau(r), Q^{*}(a^{*})\otimes b^{*}\rangle=\langle \big(r+\tau(r)\big)^{\sharp}b^{*}, Q^{*}(a^{*})\rangle\\
&&=\langle y,Q^{*}(a^{*})\rangle=\langle Q(y),a^{*}\rangle=\langle \mathcal{B}^{\natural}(x),Q(y)\rangle=\mathcal{B}\big(x,Q(y)\big).
\end{eqnarray*}
Hence $Q=\hat{P}$, that is, $\mathcal{B}^{\natural}P=Q^{*}\mathcal{B}^{\natural}$.
Moreover, we have
\begin{eqnarray*}
&&\langle RP(x), a^{*}\rangle=\langle r^{\sharp}\big(r+\tau(r)\big)^{\sharp^{-1}}P(x), a^{*}\rangle=\langle r, \big(r+\tau(r)\big)^{\sharp^{-1}}P(x)\otimes a^{*}\rangle\\
&&=\langle r, \mathcal{B}^{\natural}P(x)\otimes a^{*}\rangle=\langle r, Q^{*}\mathcal{B}^{\natural}(x)\otimes a^{*}\rangle=\langle (Q\otimes \mathrm{id})r, \mathcal{B}^{\natural}(x)\otimes a^{*}\rangle\\
&&=\langle (\mathrm{id}\otimes P)r, \mathcal{B}^{\natural}(x)\otimes a^{*}\rangle=\langle r,\mathcal{B}^{\natural}(x)\otimes P^{*}(a^{*})\rangle=\langle r^{\sharp}\mathcal{B}^{\natural}(x),P^{*}(a^{*})\rangle\\
&&=\langle r^{\sharp}\big(r+\tau(r)\big)^{\sharp^{-1}}(x), P^{*}(a^{*})\rangle
=\langle R(x), P^{*}(a^{*})\rangle=\langle PR(x),a^{*}\rangle.
\end{eqnarray*}
That is, \eqref{eq:commute} holds. Hence the conclusion follows.
\end{proof}

By Theorem \ref{thm:commute}, there is a one-to-one correspondence between  symmetric averaging Rota-Baxter Frobenius commutative algebras of weight $-1$ and factorizable averaging commutative and cocommutative infinitesimal bialgebras.
Now we give an explicit example of a factorizable averaging commutative and cocommutative infinitesimal bialgebra, starting from a symmetric averaging Rota-Baxter Frobenius commutative algebra of weight $-1$.

\begin{ex}\label{ex:3.25}
Let $(A=\mathrm{span}\{e_{1},e_{2}\},\cdot_{A})$ be a $2$-dimensional commutative  algebra defined   by the following nonzero products:
\begin{equation}\label{eq:2-dim}
    e_{1}\cdot_{A}e_{1}=e_{1},\; e_{1}\cdot_{A}e_{2}=e_{2}.
\end{equation}
The identity map $\mathrm{id}_{A}$ is a Rota-Baxter operator on $(A,\cdot_{A})$ of weight $-1$.
By Proposition \ref{pro:3.16},
\begin{equation*}
    \big(A\ltimes_{\mathcal{L}^{*}_{\cdot_{A}}}A^{*},\mathcal{B}_{d}, R=\mathrm{id}_{A}-(\mathrm{id}_{A}-\mathrm{id}_{A}  )^{*}=\mathrm{id}_{A}  \big)
\end{equation*}
is a symmetric Rota-Baxter Frobenius commutative algebra of weight $-1$,
where the nonzero products of $A\ltimes_{\mathcal{L}^{*}_{\cdot_{A}}}A^{*}$ are given by \eqref{eq:2-dim} and
\begin{equation*}
    e_{1}\cdot_{d}e^{*}_{1}=e^{*}_{1},\;  e_{1}\cdot_{d}e^{*}_{2}=e^{*}_{2},\; e_{2}\cdot_{d}e^{*}_{2}=e^{*}_{1}.
\end{equation*}
Moreover, $P=\mathrm{id}_{A}$ is also an averaging operator of $A\ltimes_{\mathcal{L}^{*}_{\cdot_{A}}}A^{*}$ and clearly commutes with $R$, and
$\hat{P}=\mathrm{id}_{A^{*}}$.
By Theorem \ref{thm:commute}, there is a factorizable averaging commutative and cocommutative infinitesimal bialgebra
\begin{equation*}
    ( A\ltimes_{\mathcal{L}^{*}_{\cdot_{A}}}A^{*}, \Delta_{r}, P, \hat{P}  ),
\end{equation*}
where $r$ is given through the operator form $r^{\sharp}$ by \eqref{eq:thm:quadratic to fact}.
Explicitly, we have
\begin{equation*}
    r^{\sharp}(x+a^{*})=R\mathcal{B}_{d}^{\natural^{-1}}(x+a^{*})=R(x+a^{*})=x,
\end{equation*}
and hence
\begin{equation*}
    r=e^{*}_{1}\otimes e_{1}+e^{*}_{2}\otimes e_{2}.
\end{equation*}
The non-zero co-multiplications are given by
\begin{equation*}
    \Delta_{r}(e^{*}_{1})=e^{*}_{1}\otimes e^{*}_{1},\; \Delta_{r}(e^{*}_{2})=e^{*}_{1}\otimes e^{*}_{2}+e^{*}_{2}\otimes e^{*}_{1}.
\end{equation*}
\end{ex}

\delete{
\section{Apre-perm bialgebras}\label{sec:4}\

\begin{defi}
Another Definition:
Let $(A,\triangleright_{A},\triangleleft_{A})$ be a \app and $(A,\circ_{A})$ be
the sub-adjacent perm algebra. Let $V$ be a vector space and
$l_{\triangleright_{A}},r_{\triangleright_{A}},
l_{\triangleleft_{A}},r_{\triangleleft_{A}}:A\rightarrow\mathrm{End}_{\mathbb
K}(V)$ be linear maps. Set
\begin{eqnarray}
l_{\circ_{A}}=l_{\triangleright_{A}}+l_{\triangleleft_{A}},\;\;
r_{\circ_{A}}=r_{\triangleright_{A}}+r_{\triangleleft_{A}},\;\;
l_{\bullet_{A}}=l_{\triangleright_{A}}+r_{\triangleleft_{A}},\;\;
r_{\bullet_{A}}=r_{\triangleright_{A}}+l_{\triangleleft_{A}}.\label{eq:sum linear}
\end{eqnarray}
If $(l_{\circ_{A}},r_{\circ_{A}},V)$ is a representation of $(A,\circ_{A})$ and the following equations hold:
\begin{eqnarray}
&&l_{\bullet_{A}}(x\circ_{A}y)v=l_{\bullet_{A}}(x)l_{\bullet_{A}}(y)v
=l_{\bullet_{A}}(y)l_{\bullet_{A}}(x)v,\label{eq:dpprep1}\\
&&r_{\bullet_{A}}(x)l_{\circ_{A}}(y)v=l_{\bullet_{A}}(y)r_{\bullet_{A}}
(x)v=r_{\bullet_{A}}(y\bullet_{A}x)v,\label{eq:dpprep2}\\
&&r_{\bullet_{A}}(x)r_{\circ_{A}}(y)v=r_{\bullet_{A}}(y\bullet_{A}x)v
=l_{\bullet_{A}}(y)r_{\bullet_{A}}(x)v,\label{eq:dpprep3}\\
&&l_{\triangleleft_{A}}(x)l_{\circ_{A}}(y)v=-r_{\triangleleft_{A}}(y)
l_{\triangleleft_{A}}(x)v=l_{\bullet_{A}}(y)l_{\triangleleft_{A}}(x)v
=l_{\triangleleft_{A}}(y\bullet_{A}x)v,\label{eq:dpprep4}\\
&&l_{\triangleleft_{A}}(x)r_{\circ_{A}}(y)v=-l_{\triangleleft_{A}}
(x\triangleleft_{A}y)v=r_{\bullet_{A}}(x\triangleleft_{A}y)v
=r_{\triangleleft_{A}}(y)r_{\bullet_{A}}(x)v,\label{eq:dpprep5}\\
&&r_{\triangleleft_{A}}(x\circ_{A}y)v=-r_{\triangleleft_{A}}
(x)r_{\triangleleft_{A}}(y)v=l_{\bullet_{A}}(x)r_{\triangleleft_{A}}(y)v
=r_{\triangleleft_{A}}(y)l_{\bullet_{A}}(x)v,\;\forall x,y\in A, v\in V,\ \ \ \ \label{eq:dpprep6}
\end{eqnarray}
then we say $(l_{\triangleright_{A}},r_{\triangleright_{A}},
l_{\triangleleft_{A}},r_{\triangleleft_{A}},V)$ is a {\bf representation} of $(A,\triangleright_{A},\triangleleft_{A})$.
\end{defi}

\begin{pro}
Another Pro:
Let $(A,\triangleright_{A},\triangleleft_{A})$ be a \app and $(A,\circ_{A})$ be
the sub-adjacent perm algebra. If
$(l_{\triangleright_{A}},r_{\triangleright_{A}},
l_{\triangleleft_{A}},r_{\triangleleft_{A}},V)$ is a
representation of $(A,\triangleright_{A},\triangleleft_{A})$, then
$(l^{*}_{\circ_{A}}+r^{*}_{\circ_{A}},r^{*}_{\triangleright_{A}},$
$-l^{*}_{\triangleleft_{A}}-r^{*}_{\triangleright_{A}},
-r^{*}_{\circ_{A}},V^{*})$ is also a representation of $(A,\triangleright_{A},\triangleleft_{A})$.
In particular, $(\mathcal{L}^{*}_{\circ_{A}}+\mathcal{R}^{*}_{\circ_{A}},\mathcal{R}^{*}_{\triangleright_{A}},$
$-\mathcal{L}^{*}_{\triangleleft_{A}}-\mathcal{R}^{*}_{\triangleright_{A}},
-\mathcal{R}^{*}_{\circ_{A}},A^{*})$ is a representation of $(A,\triangleright_{A},\triangleleft_{A})$, which is called the {\bf coadjoint representation} of $(A,\triangleright_{A},\triangleleft_{A})$.
\end{pro}

\begin{proof}
It is clear that $l^{*}_{\bullet_{A}}$ satisfies \eqref{eq:rep1}.
For all $x,y\in A,u^{*}\in V^{*},v\in V$, we have
{\small
\begin{eqnarray*}
&&\langle -r^{*}_{\triangleleft_{A}}(x\circ_{A}y)u^{*},v\rangle =\langle u^{*
},-r _{\triangleleft_{A}}(x\circ _{A}y)v\rangle ,\;\langle r_{\triangleleft_{A}}^{* }(y)r_{\triangleleft_{A}}^{* }(x)u^{*
},v\rangle =\langle u^{* },r_{\triangleleft_{A}}(x)r_{\triangleleft_{A}}(y)v\rangle , \\
&&\langle -r_{\triangleleft_{A}}^{* }(y)l_{\bullet_{A}}^{* }(x)u^{* },v\rangle =\langle u^{*
},-l _{\bullet_{A}}(x)r_{\triangleleft_{A}}(y)v\rangle ,\;\langle -l_{\bullet_{A}}^{* }(x)r_{\triangleleft_{A}}^{* }(y)u^{* },v\rangle
=\langle u^{* },-r_{\triangleleft_{A}}(y)l_{\bullet_{A}}(x)v\rangle .
\end{eqnarray*}}Hence by \eqref{eq:dpprep6}, we have
\begin{equation*}
-r_{\triangleleft_{A}}^{* }(x\circ _{A}y)u^{* }=r_{\triangleleft_{A}}^{* }(y)r_{\triangleleft_{A}}^{* }(x)u^{* }=-r_{\triangleleft_{A}}^{*
}(y)l_{\bullet_{A}}^{* }(x)u^{* }=-l_{\bullet_{A}}^{* }(x)r_{\triangleleft_{A}}^{* }(y)u^{* }.
\end{equation*}
Thus $$ (l^{*}_{\bullet_{A}},-r^{*}_{\triangleleft_{A}},V^{*})=(l^{*}_{\circ_{A}}+r^{*}_{\circ_{A}}
-l^{*}_{\triangleleft_{A}}-r^{*}_{\triangleright_{A}},r^{*}_{\triangleright_{A}}-r^{*}_{\circ_{A}},V^{*})$$
is a representation of $(A,\circ_{A})$.
Similarly,\eqref{eq:dpprep1}-\eqref{eq:dpprep6} hold for
$(l^{*}_{\circ_{A}}+r^{*}_{\circ_{A}},r^{*}_{\triangleright_{A}},
-l^{*}_{\triangleleft_{A}}-r^{*}_{\triangleright_{A}},-r^{*}_{\circ_{A}},$
$V^{*})$.
\delete{
\begin{eqnarray*}
&&-r_{\circ_{A}}^{* }(x\circ _{A}y)u^{* }=-l_{\circ_{A}}^{* }(x)r_{\circ_{A}}^{* }(y)u^{* }=-r_{\circ_{A}}^{*
}(x)r_{\circ_{A}}^{* }(y)u^{* }=-r_{\circ_{A}}^{* }(y)l_{\circ_{A}}^{* }(x)u^{* }, \\
&&-r_{\circ_{A}}^{* }(x\triangleleft_{A}y)u^{* }=-r_{\circ_{A}}^{* }(y)l_{\triangleleft_{A}}^{* }(x)u^{* }=l_{\triangleleft_{A}}^{*
}(x\triangleleft_{A}y)u^{* }.
\end{eqnarray*}}
Thus $(l^{*}_{\circ_{A}}+r^{*}_{\circ_{A}},r^{*}_{\triangleright_{A}},
-l^{*}_{\triangleleft_{A}}-r^{*}_{\triangleright_{A}},-r^{*}_{\circ_{A}},V^{*})$ is a representation of $(A,\triangleright_{A},\triangleleft_{A})$.
\end{proof}

\begin{defi}
Let $(A,\triangleright_{A},\triangleleft_{A})$ be a \app and $(A,\circ_{A})$ be
the sub-adjacent perm algebra. Let $V$ be a vector space and
$l_{\triangleright_{A}},r_{\triangleright_{A}},
l_{\triangleleft_{A}},r_{\triangleleft_{A}}:A\rightarrow\mathrm{End}_{\mathbb
K}(V)$ be linear maps. Set
\begin{eqnarray}
l_{\circ_{A}}=l_{\triangleright_{A}}+l_{\triangleleft_{A}},\;\;
r_{\circ_{A}}=r_{\triangleright_{A}}+r_{\triangleleft_{A}},\;\;
l_{\bullet_{A}}=l_{\triangleright_{A}}+r_{\triangleleft_{A}},\;\;
r_{\bullet_{A}}=r_{\triangleright_{A}}+l_{\triangleleft_{A}}.\label{eq:sum linear}
\end{eqnarray}
If $(l_{\circ_{A}},r_{\circ_{A}},V)$ is a representation of $(A,\circ_{A})$ and the following equations hold:
\begin{eqnarray}
&&l_{\triangleright_{A}}(x)l_{\triangleleft_{A}}(y)v
+2r_{\triangleleft_{A}}(x)l_{\triangleleft_{A}}(y)v=0,\label{eq:dpprep1}\\
&&l_{\triangleright_{A}}(x)r_{\triangleleft_{A}}(y)v
+2r_{\triangleleft_{A}}(x)r_{\triangleleft_{A}}(y)v=0,\label{eq:dpprep2}\\
&&r_{\triangleright_{A}}(x\triangleleft_{A}y)v
+2l_{\triangleleft_{A}}(x\triangleleft_{A}y)v=0,\label{eq:dpprep3}\\
&&l_{\triangleleft_{A}}(x)l_{\circ_{A}}(y)v=-r_{\triangleleft_{A}}(y)
l_{\triangleleft_{A}}(x)v=l_{\triangleleft_{A}}(y\bullet_{A}x)v,\label{eq:dpprep4}\\
&&l_{\triangleleft_{A}}(x)r_{\circ_{A}}(y)v
=-l_{\triangleleft_{A}}(x\triangleleft_{A}y)v
=r_{\triangleleft_{A}}(y)r_{\bullet_{A}}(x)v,\label{eq:dpprep5}\\
&&r_{\triangleleft_{A}}(x\circ_{A}y)v
=-r_{\triangleleft_{A}}(x)r_{\triangleleft_{A}}(y)v
=r_{\triangleleft_{A}}(y)l_{\bullet_{A}}(x)v,\label{eq:dpprep6}\\
&&l_{\triangleright_{A}}(x\circ_{A}y)v
=l_{\triangleright_{A}}(y)l_{\bullet_{A}}(x)v
=l_{\bullet_{A}}(x)l_{\triangleright_{A}}(y)v,\label{eq:dpprep7}\\
&&r_{\triangleright_{A}}(x)l_{\circ_{A}}(y)v
=r_{\triangleright_{A}}(y\bullet_{A}x)v
=l_{\bullet_{A}}(y)r_{\triangleright_{A}}(x)v,\label{eq:dpprep8}\\
&&r_{\triangleright_{A}}(x)r_{\circ_{A}}(y)v
=l_{\triangleright_{A}}(y)r_{\bullet_{A}}(x)v
=r_{\bullet_{A}}(y\triangleright_{A}x)v
,\;\forall x,y\in A, v\in V,\label{eq:dpprep9}
\end{eqnarray}
then we say $(l_{\triangleright_{A}},r_{\triangleright_{A}},
l_{\triangleleft_{A}},r_{\triangleleft_{A}},V)$ is a {\bf representation} of $(A,\triangleright_{A},\triangleleft_{A})$.
Two representations $(l _{\triangleright_{A}},r _{\triangleright_{A}},l _{\triangleleft_{A}},r_{\triangleleft_{A}},V)$ and $(l' _{\triangleright_{A}},r' _{\triangleright_{A}},l' _{\triangleleft_{A}},r'_{\triangleleft_{A}},V')$ of $(A,\triangleright_{A},\triangleleft_{A})$ are called
\textbf{equivalent} if there exists a linear isomorphism $\phi:V \rightarrow V'$ such that the following equations hold:
\begin{eqnarray*}
\phi l_{\triangleright_{A}}(x)=l'_{\triangleright_{A}}(x)\phi,\;\phi r_{\triangleright_{A}}(x)=r'_{\triangleright_{A}}(x)\phi,\;\phi l_{\triangleleft_{A}} (x)=l'_{\triangleleft_{A}}(x)\phi, \;r_{\triangleleft_{A}}(x)=r'_{\triangleleft_{A}}(x)\phi,\;\forall x\in A.
\end{eqnarray*}
\end{defi}

In fact, for a vector space $V$ and linear maps
$l_{\triangleright_{A}},r_{\triangleright_{A}},
l_{\triangleleft_{A}},r_{\triangleleft_{A}}:A\rightarrow \mathrm{
End}(V)$, the quadruple
$(l_{\triangleright_{A}},r_{\triangleright_{A}},
l_{\triangleleft_{A}},r_{\triangleleft_{A}},V)$ is a
representation of the \app $(A,\triangleright_{A},\triangleleft_{A})$ if and only
if there is a \app structure on $A\oplus V$ given by
\begin{eqnarray}
&&(x+u)\triangleright_{d} (y+v)=x\triangleright_{A}y+l_{\triangleright_{A}}(x)v+r_{\triangleright_{A}}(y)u,\label{eq:sd SDPP1}\\
&&(x+u)\triangleleft_{d} (y+v)=x\triangleleft_{A}y+l_{\triangleleft_{A}}(x)v+r_{\triangleleft_{A}}(y)u,\;\forall x,y\in A,u,v\in V.
\label{eq:sd SDPP2}
\end{eqnarray}
We denote the \app structure on $A\oplus V$ by $A\ltimes
_{l_{\triangleright_{A}},r_{\triangleright_{A}},l_{\triangleleft_{A}},r_{\triangleleft_{A}}}V$.

\begin{ex}
Let $(A,\triangleright_{A},\triangleleft_{A})$ be a \app. Then
$(\mathcal{L}_{\triangleright_{A}},\mathcal{R}_{\triangleright_{A}},
\mathcal{L}_{\triangleleft_{A}},\mathcal{R}_{\triangleleft_{A}},A)$
is a representation of $(A,\triangleright_{A},\triangleleft_{A} )$, which is called
the \textbf{adjoint representation}.
\end{ex}

\begin{pro}
Let $(A,\triangleright_{A},\triangleleft_{A})$ be a \app and $(A,\circ_{A})$ be
the sub-adjacent perm algebra. If
$(l_{\triangleright_{A}},r_{\triangleright_{A}},
l_{\triangleleft_{A}},r_{\triangleleft_{A}},V)$ is a
representation of $(A,\triangleright_{A},\triangleleft_{A})$, then
$(l^{*}_{\circ_{A}}+r^{*}_{\circ_{A}},r^{*}_{\triangleright_{A}},$
$-l^{*}_{\triangleleft_{A}}-r^{*}_{\triangleright_{A}},
-r^{*}_{\circ_{A}},V^{*})$ is also a representation of $(A,\triangleright_{A},\triangleleft_{A})$.
In particular, $(\mathcal{L}^{*}_{\circ_{A}}+\mathcal{R}^{*}_{\circ_{A}},\mathcal{R}^{*}_{\triangleright_{A}},$
$-\mathcal{L}^{*}_{\triangleleft_{A}}-\mathcal{R}^{*}_{\triangleright_{A}},
-\mathcal{R}^{*}_{\circ_{A}},A^{*})$ is a representation of $(A,\triangleright_{A},\triangleleft_{A})$, which is called the {\bf coadjoint representation} of $(A,\triangleright_{A},\triangleleft_{A})$.
\end{pro}

\begin{proof}
For all $x,y\in A,u^{*}\in V^{*},v\in V$, we have
\begin{eqnarray*}
\langle l^{*}_{\bullet_{A}}(x\circ_{A}y)u^{*},v\rangle =\langle u^{*},l _{\bullet_{A}}(x\circ_{A}y)v\rangle ,\;
\langle l^{*}_{\bullet_{A}}(x)l^{*}_{\bullet_{A}}(y)u^{*},v\rangle =\langle u^{*},l_{\bullet_{A}}(y)l_{\bullet_{A}}(x)v\rangle.
\end{eqnarray*}
Hence by \eqref{eq:dpprep6} and \eqref{eq:dpprep7}, we have
\begin{eqnarray*}
l_{\bullet_{A}}(x\circ_{A}y)v&=&l_{\bullet_{A}}(y)l_{\bullet_{A}}(x)v
=l_{\bullet_{A}}(x)l_{\triangleright_{A}}(y)v
-r_{\triangleleft_{A}}(x)r_{\triangleleft_{A}}(y)v\\
&\overset{\eqref{eq:dpprep2}}{=}&l_{\bullet_{A}}(x)l_{\triangleright_{A}}(y)v
+l_{\bullet_{A}}(x)r_{\triangleleft_{A}}(y)v=l_{\bullet_{A}}(x)l_{\bullet_{A}}(y)v.
\end{eqnarray*}
Thus $l^{*}_{\bullet_{A}}$ satisfies \eqref{eq:rep1}.
We also have
{\small
\begin{eqnarray*}
&&\langle -r^{*}_{\triangleleft_{A}}(x\circ_{A}y)u^{*},v\rangle =\langle u^{*
},-r _{\triangleleft_{A}}(x\circ_{A}y)v\rangle ,\;\langle r_{\triangleleft_{A}}^{* }(y)r_{\triangleleft_{A}}^{* }(x)u^{*
},v\rangle =\langle u^{* },r_{\triangleleft_{A}}(x)r_{\triangleleft_{A}}(y)v\rangle , \\
&&\langle -r_{\triangleleft_{A}}^{* }(y)l_{\bullet_{A}}^{* }(x)u^{* },v\rangle =\langle u^{*
},-l _{\bullet_{A}}(x)r_{\triangleleft_{A}}(y)v\rangle ,\;\langle -l_{\bullet_{A}}^{* }(x)r_{\triangleleft_{A}}^{* }(y)u^{* },v\rangle
=\langle u^{* },-r_{\triangleleft_{A}}(y)l_{\bullet_{A}}(x)v\rangle .
\end{eqnarray*}}Hence by \eqref{eq:dpprep2} and \eqref{eq:dpprep6}, we have
\begin{equation*}
-r_{\triangleleft_{A}}^{* }(x\circ _{A}y)u^{* }=r_{\triangleleft_{A}}^{* }(y)r_{\triangleleft_{A}}^{* }(x)u^{* }=-r_{\triangleleft_{A}}^{*
}(y)l_{\bullet_{A}}^{* }(x)u^{* }=-l_{\bullet_{A}}^{* }(x)r_{\triangleleft_{A}}^{* }(y)u^{* }.
\end{equation*}
Thus $$ (l^{*}_{\bullet_{A}},-r^{*}_{\triangleleft_{A}},V^{*})=(l^{*}_{\circ_{A}}+r^{*}_{\circ_{A}}
-l^{*}_{\triangleleft_{A}}-r^{*}_{\triangleright_{A}},r^{*}_{\triangleright_{A}}-r^{*}_{\circ_{A}},V^{*})$$
is a representation of $(A,\circ_{A})$.
Similarly, \eqref{eq:dpprep1}-\eqref{eq:dpprep9} hold.
\delete{
\begin{eqnarray*}
&&-r_{\circ_{A}}^{* }(x\circ _{A}y)u^{* }=-l_{\circ_{A}}^{* }(x)r_{\circ_{A}}^{* }(y)u^{* }=-r_{\circ_{A}}^{*
}(x)r_{\circ_{A}}^{* }(y)u^{* }=-r_{\circ_{A}}^{* }(y)l_{\circ_{A}}^{* }(x)u^{* }, \\
&&-r_{\circ_{A}}^{* }(x\triangleleft_{A}y)u^{* }=-r_{\circ_{A}}^{* }(y)l_{\triangleleft_{A}}^{* }(x)u^{* }=l_{\triangleleft_{A}}^{*
}(x\triangleleft_{A}y)u^{* }.
\end{eqnarray*}}
Thus $(l^{*}_{\circ_{A}}+r^{*}_{\circ_{A}},r^{*}_{\triangleright_{A}},
-l^{*}_{\triangleleft_{A}}-r^{*}_{\triangleright_{A}},-r^{*}_{\circ_{A}},V^{*})$ is a representation of $(A,\triangleright_{A},\triangleleft_{A})$.
\end{proof}

\begin{pro}\label{pro:336}
Let $(A,\triangleright_{A},\triangleleft_{A})$ be a \app and $(A,\circ_{A})$ be
the sub-adjacent perm algebra. Then  $(\mathcal{L}_{\triangleright_{A}},\mathcal{R}_{\triangleright_{A}},
\mathcal{L}_{\triangleleft_{A}},\mathcal{R}_{\triangleleft_{A}},A)$
and
$(\mathcal{L}^{*}_{\circ_{A}}+\mathcal{R}^{*}_{\circ_{A}},\mathcal{R}^{*}_{\triangleright_{A}},
-\mathcal{L}^{*}_{\triangleleft_{A}}-\mathcal{R}^{*}_{\triangleright_{A}},
-\mathcal{R}^{*}_{\circ_{A}},A^{*})$
are equivalent as representations of $(A,\triangleright_{A},\triangleleft_{A})$
if and only if
there is a nondegenerate bilinear form $\mathcal{B}$ on $A$ satisfying \eqref{eq:cor3}, \eqref{eq:cor4} and the following equations:
\begin{eqnarray}
\mathcal{B}(x\bullet_{A}y,z)&=&-\mathcal{B}(x,y\triangleleft_{A}z)
-\mathcal{B}(x,z\triangleright_{A}y),\label{eq:dppbf3}\\
\mathcal{B}(x\triangleright_{A}y,z)&=&
\mathcal{B}(x,z\triangleright_{A}y),\;\forall x,y,z\in A.\label{eq:cor3.37}
\end{eqnarray}
\end{pro}

\begin{proof}
Let $\phi:A\rightarrow A^{*}$ be a linear map and $\mathcal{B}$ be a bilinear form on $A$ given by $\mathcal{B}^{\natural}=\phi$.
Then $\phi$ is bijective if and only if $\mathcal{B}$ is nondegenerate.
Moreover, the following equations hold
\begin{eqnarray*}
&&\phi\big(\mathcal{L}_{\triangleright_{A}}(x)y\big)=(\mathcal{L}^{*}_{\circ_{A}}+\mathcal{R}^{*}_{\circ_{A}})(x)\phi(y),\;
\phi\big(\mathcal{R}_{\triangleright_{A}}(y)x\big)=\mathcal{R}^{*}_{\triangleright_{A}}(y)\phi(x),\\
&&\phi\big(\mathcal{L}_{\triangleleft_{A}}(y)x\big)=
(-\mathcal{L}^{*}_{\triangleleft_{A}}-\mathcal{R}^{*}_{\triangleright_{A}})(y)\phi(x),\;
\phi\big(\mathcal{R}_{\triangleleft_{A}}(y)x\big)=-\mathcal{R}^{*}_{\circ_{A}} (y)\phi(x),\;\forall x,y\in A
\end{eqnarray*}
if and only if  \eqref{eq:cor3}, \eqref{eq:cor4}, \eqref{eq:dppbf3} and \eqref{eq:cor3.37} are satisfied.
Hence the conclusion follows.
\end{proof}

\begin{defi}
A (symmetric) bilinear form $\mathcal{B}$ on a perm algebra $(A,\circ_{A})$ is called
{\bf quasi-left-invariant} if \eqref{eq:left inv1} holds.
\end{defi}

\begin{defi}
A (symmetric) bilinear form $\mathcal{B}$ on a \app  $(A,\triangleright_{A},\triangleleft_{A})$ is called
{\bf invariant} if \eqref{eq:cor3} and \eqref{eq:cor4} hold. A {\bf quadratic \app}
$(A,\triangleright_{A},\triangleleft_{A},\mathcal{B})$ is a \app $(A,\triangleright_{A},\triangleleft_{A})$ together with a nondegenerate symmetric invariant
bilinear form $\mathcal{B}$.
\end{defi}

\begin{pro}\label{pro:330}
Let $\mathcal{B}$ be a nondegenerate symmetric quasi-left-invariant bilinear form on a perm algebra $(A,\circ_{A})$. Define
multiplications $\triangleleft_{A},\triangleright_{A}:A\otimes A\rightarrow A$
respectively by \eqref{eq:cor4} and the following equation
\begin{equation}\label{eq:succ}
x\triangleright_{A}y=x\circ_{A}y-x\triangleleft_{A}y,\;\forall x,y\in A.
\end{equation}
Then $(A,\triangleright_{A},\triangleleft_{A},\mathcal{B})$ is a quadratic
\app. Conversely, let $(A,\triangleright_{A},\triangleleft_{A},\mathcal{B})$ be
a quadratic \app. Then $\mathcal B$ is quasi-left-invariant on the
sub-adjacent perm algebra $(A,\circ_A)$. Hence there is a
one-to-one correspondence between perm algebras with nondegenerate
symmetric quasi-left-invariant bilinear forms and quadratic \apps.
\end{pro}

\begin{proof}
The first half part follows directly from
Theorem \ref{thm:420}.
The second half part is obtained directly.
\end{proof}

\begin{defi}
    \begin{enumerate}
        \item Let $(A,\circ_{A})$ and
        $(A^{*},\circ_{A^{*}})$ be perm algebras. If there is a perm algebra
        structure $(A\oplus
        A^{*},\circ_{d})$ on $A\oplus A^{*}$
        containing $(A,\circ_{A})$ and $(
        A^{*},\circ_{A^{*}})$ as perm subalgebras, and the natural nondegenerate symmetric bilinear form $\mathcal{B}_{d}$
        is quasi-left-invariant on $(A\oplus
        A^{*},\circ_{d})$, then we say
        $\big(  (  A\oplus
        A^{*},\circ_{d},\mathcal{B}_{d}),(A,\circ_{A}),(A^{*},\circ_{A^{*}})\big)
        $ is a {\bf Manin triple of perm algebras associated to the nondegenerate
            symmetric quasi-left-invariant bilinear form}.
        \item Let $(A,\triangleright_{A},\triangleleft_{A})$ and
        $(A^{*},\triangleright_{A^{*}},\triangleleft_{A^{*}})$ be \apps. If there is a
        quadratic \app structure $(A\oplus
        A^{*},\triangleright_{d},\triangleleft_{d},\mathcal{B}_{d})$ on $A\oplus A^{*}$
        which contains $(A,\triangleright_{A},\triangleleft_{A})$ and $(
        A^{*},\triangleright_{A^{*}},\triangleleft_{A^{*}})$ as \appsubs, then we say
        $\big(  (  A\oplus
        A^{*},\triangleright_{d},\triangleleft_{d},\mathcal{B}_{d}),
        (A,\triangleright_{A},\triangleleft_{A}),$
        $(A^{*},\triangleright_{A^{*}},\triangleleft_{A^{*}})\big)
        $ is a \textbf{Manin triple of \apps}.
    \end{enumerate}
\end{defi}

\begin{pro}\label{pro:equ}
There is a one-by-one correspondence between
Manin triples of perm algebras associated to nondegenerate
symmetric quasi-left-invariant bilinear forms and Manin triples of \apps.
\end{pro}
\begin{proof}
    Let $\big(  (  A\oplus
    A^{*},\circ_{d},\mathcal{B}_{d}),(A,\circ_{A}),(A^{*},\circ_{A^{*}})\big)
    $ be a Manin triple of perm algebras associated to the nondegenerate
    symmetric quasi-left-invariant bilinear form. Then by Proposition \ref{pro:330}, there is a quadratic \app  $(A\oplus A^{*},\triangleright_{d},\triangleleft_{d},\mathcal{B})$ with the multiplications $\triangleright_{d},\triangleleft_{d}$ defined by
    {\small
    \begin{eqnarray*}
    &&\mathcal{B}_{d}\big( (x+a^{*})\triangleleft_{d}(y+b^{*}) ,z+c^{*}\big)=-\mathcal{B}_{d}\big( x+a^{*}, (z+c^{*})\circ_{d}(y+b^{*}) \big),\\
    &&(x+a^{*})\triangleright_{d}(y+b^{*})=(x+a^{*})\circ_{d}(y+b^{*})-(x+a^{*})\triangleleft_{d}(y+b^{*}),\;\forall x,y,z\in A, a^{*},b^{*},c^{*}\in A^{*}.
    \end{eqnarray*}}In particular, we have
\begin{eqnarray*}
    \mathcal{B}_{d}(x\triangleleft_{d}y,z)=-\mathcal{B}_{d}(x,z\circ_{d}y)=0,\;\;\forall
    x,y,z\in A.
\end{eqnarray*}
Hence we have $x\triangleleft_{d}y\in A$, and $x\triangleright_{d}y=x\circ_{d}y-x\triangleleft_{d}y=x\circ_{A}y-x\triangleleft_{d}y\in A$.
Thus $(A,\triangleright_{A},\triangleleft_{A}):=(A,\triangleright_{d}|_{A},\triangleleft_{d}|_{A})$ is a subalgebra of $(A\oplus A^{*},\triangleright_{d},\triangleleft_{d})$.
Similarly, $(A^{*},\triangleright_{A^{*}},\triangleleft_{A^{*}})$ is also a subalgebra of $(A\oplus A^{*},\triangleright_{d},\triangleleft_{d})$.
Hence $\big(  (  A\oplus
A^{*},\triangleright_{d},\triangleleft_{d},\mathcal{B}_{d}),(A,\triangleright_{A},\triangleleft_{A}),(A^{*},\triangleright_{A^{*}},\triangleleft_{A^{*}})\big)
$ is a  Manin triple of \apps.

Conversely, suppose that $\big(  (  A\oplus
A^{*},\triangleright_{d},\triangleleft_{d},\mathcal{B}_{d}),(A,\triangleright_{A},\triangleleft_{A}),(A^{*},\triangleright_{A^{*}},\triangleleft_{A^{*}})\big)
$ is a  Manin triple of \apps. Then $\big(  (  A\oplus
A^{*},\circ_{d},\mathcal{B}_{d}),(A,\circ_{A}),(A^{*},\circ_{A^{*}})\big)
$ is straightforwardly a Manin triple of perm algebras associated to the nondegenerate
symmetric quasi-left-invariant bilinear form, where $(A\oplus A^{*},\circ_{d}),(A,\circ_{A})$ and $(A^{*},\circ_{A^{*}})$ are the sub-adjacent perm algebras of $(A\oplus A^{*},\triangleright_{d},\triangleleft_{d})$,
$(A,\triangleright_{A},\triangleleft_{A})$ and $(A^{*},\triangleright_{A^{*}},\triangleleft_{A^{*}})$ respectively.
\end{proof}

\begin{thm}\label{thm:Manin triple}
Let $(A,\triangleright_{A},\triangleleft_{A})$ and
$(A^{*},\triangleright_{A^{*}},\triangleleft_{A^{*}})$ be \apps and their sub-adjacent perm algebras be $(A,\circ_{A})$ and
$(A^{*},\circ_{A^{*}})$ respectively. Then the
following conditions are equivalent:
\begin{enumerate}
\item\label{E4} There is a Manin triple $\big(  (  A\oplus
A^{*},\circ_{d},\mathcal{B}_{d}),(A,\circ_{A}),(A^{*},\circ_{A^{*}})\big)
$ of perm algebras associated to the nondegenerate symmetric
quasi-left-invariant bilinear form such that the compatible \app
$(A\oplus A^{*},\triangleright_{d},\triangleleft_{d})$ induced
from $\mathcal{B}_{d}$ contains
$(A,\triangleright_{A},\triangleleft_{A})$ and
$(A^{*},\triangleright_{A^{*}},\triangleleft_{A^{*}})$ as
subalgebras.
\item\label{E1} There is a Manin triple
$\big((A\oplus
A^{*},\triangleright_{d},\triangleleft_{d},\mathcal{B}_{d}),
(A,\triangleright_{A},\triangleleft_{A}),
(A^{*},\triangleright_{A^{*}},
\triangleleft_{A^{*}})\big)$ of \apps.
\item\label{E2} There is
a perm algebra $(A\oplus A^{*},\circ_{d})$ with $ \circ_{d} $
defined by
\begin{equation}\label{eq:A ds}
    (x+a^{*})\circ_{d}(y+b^{*})=x\circ_{A}y+\mathcal{L}^{*}_{\bullet_{A^{*}}}(a^{*})y
    -\mathcal{R}^{*}_{\triangleleft_{A^{*}}}(b^{*})x+a^{*}\circ_{A^{*}}b^{*}+
    \mathcal{L}^{*}_{\bullet_{A}}(x)b^{*}-\mathcal{R}^{*}_{\triangleleft_{A}}(y)a^{*},
\end{equation}
for all $x,y\in A, a^{*},b^{*}\in A^{*}$. \item\label{E3} There is
a \app  $(A\oplus A^{*},\triangleright_{d},\triangleleft_{d})$
with $\triangleright_{d},\triangleleft_{d}$ defined by
\begin{eqnarray}
    (x+a^{*})\triangleright_{d}(y+b^{*})&=&\;x\triangleright_{A}y+(\mathcal{L}^{*}_{\circ_{A^{*}}}+\mathcal{R}^{*}_{\circ_{A^{*}}})(a^{*})y
    +\mathcal{R}^{*}_{\triangleright_{A^{*}}}(b^{*})x\nonumber\\
    &&\;+a^{*}\triangleright_{A^{*}}b^{*}+(\mathcal{L}^{*}_{\circ_{A}}+\mathcal{R}^{*}_{\circ_{A}})(x)b^{*}+\mathcal{R}^{*}_{\triangleright_{A}}(y)a^{*},\label{eq:A ds1}\\
    (x+a^{*})\triangleleft_{d}(y+b^{*})&=&\;x\triangleleft_{A}y-
    (\mathcal{L}^{*}_{\triangleleft_{A^{*}}}
    +\mathcal{R}^{*}_{\triangleright_{A^{*}}})(a^{*})y
    -\mathcal{R}^{*}_{\circ_{A^{*}}}(b^{*})x\nonumber\\
    &&\;+a^{*}\triangleleft_{A^{*}}b^{*}-(\mathcal{L}^{*}_{\triangleleft_{A}}
    +\mathcal{R}^{*}_{\triangleright_{A}})(x)b^{*}
    -\mathcal{R}^{*}_{\circ_{A }}(y)a^{*},\label{eq:A ds2}
\end{eqnarray}
for all $x,y\in A, a^{*},b^{*}\in A^{*}$.
\end{enumerate}
\end{thm}

\begin{proof}
(\ref{E4})$\Longleftrightarrow$(\ref{E1}) It follows from Proposition \ref{pro:equ}.

(\ref{E1})$\Longrightarrow$(\ref{E3}) \delete{ Suppose that there
is a Manin triple of \sapps $\big(( A\oplus
A^{*}, \triangleright_{d}, \triangleleft_{d},\mathcal{B}_{d}),(A, \triangleright_{A}$,
$ \triangleleft_{A}),(A^{*}, \triangleright_{A^{*}}, \triangleleft_{A^{*}})\big) $. Then for
all $x,y\in A, a^{*},b^{*}\in A^{*}$, we have} Let $x,y\in A,
a^{*},b^{*}\in A^{*}$. By the assumption, we have
\begin{eqnarray*}
\mathcal{B}_{d}(x\triangleright_{d} b^{*},y)&\overset{\eqref{eq:cor3}}{=}&\mathcal{B}
_{d}(b^{*},x\circ_{A} y+y\circ_{A} x)=
\mathcal{B}_{d}\big((\mathcal{L}
^{*}_{\circ_{A}}+\mathcal{R}
^{*}_{\circ_{A}})(x)b^{*},y\big),\\
\mathcal{B}_{d}(x\triangleright_{d} b^{*},a^{*})&\overset{\eqref{eq:cor3.37}}{=}&
\mathcal{B}_{d}(x,a^{*}\triangleright_{A^{*}}b^{*})=\langle x,
a^{*}\triangleright_{A^{*}}b^{*}\rangle=\mathcal{B}_{d}\big(\mathcal{R}
^{*}_{\triangleright_{A^{*}}}(b^{*})x,a^{*}\big).
\end{eqnarray*}
Thus we have
\[
x\triangleright_{d} b^{*}=(\mathcal{L}
^{*}_{\circ_{A}}+\mathcal{R}
^{*}_{\circ_{A}})(x)b^{*}+\mathcal{R}
^{*}_{\triangleright_{A^{*}}}(b^{*})x,
\]
and similarly
\[
a^{*}\triangleright_{d} y=(\mathcal{L}
^{*}_{\circ_{A^{*}}}+\mathcal{R}
^{*}_{\circ_{A^{*}}})(a^{*})y+\mathcal{R}
^{*}_{\triangleright_{A}}(y)a^{*}.
\]
Hence \eqref{eq:A ds1} holds. Moreover, we have
{\small
\begin{align*}
\mathcal{B}_{d}(x\triangleleft_{d} b^{*},y)&\overset{\eqref{eq:dppbf3}}{=}-\mathcal{B}
_{d}(b^{*},x\triangleleft_{A}y)-\mathcal{B}
_{d}(b^{*},y\triangleright_{A}x)
=-\mathcal{B}_{d}\big((\mathcal{L}^{*}_{\triangleleft_{A}}+\mathcal{R}^{*}_{\triangleright_{A}})(x)b^{*},y\big), \\
\mathcal{B}_{d}(x\triangleleft_{d} b^{*},a^{*})&\overset{\eqref{eq:cor4}}{=}
-\mathcal{B}_{d}(x,a^{*}\circ_{A^{*}}b^{*})=-\langle \mathcal{R}^{*}_{
    \circ_{A^{*}}} (b^{*})x, a^{*}\rangle=-\mathcal{B}_{d}\big(\mathcal{R}^{*}_{
    \circ_{A^{*}}}(b^{*})x, a^{*} \big).
\end{align*}}Thus we have
\[
x\triangleleft_{d} b^{*}=-(\mathcal{L}^{*}_{\triangleleft_{A}}+\mathcal{R}^{*}_{\triangleright_{A}})(x)b^{*}-\mathcal{R}^{*}_{
\circ_{A^{*}}} (b^{*})x.
\]
Hence \eqref{eq:A ds2} holds.

(\ref{E3})$\Longrightarrow$(\ref{E2}) It is straightforward.

(\ref{E2})$\Longrightarrow$(\ref{E1}) Suppose that there is a perm
algebra $(A\oplus A^{*},\circ_{d})$ with $\circ_{d}$ defined by \eqref{eq:A ds}. Then
it is straightforward to check that $\mathcal{B}_{d}$ is
quasi-left-invariant on $(A\oplus A^{*},\circ_{d})$. By Proposition
\ref{pro:330}, there is a quadratic \app $(A\oplus
A^{*},\triangleright_{d},\triangleleft_{d},\mathcal{B}_{d})$ with $\triangleright_{d},\triangleleft_{d}$ defined by
 \eqref{eq:cor3} and \eqref{eq:cor4}. Moreover, let $x,y,z\in
A$. Then we have
\begin{equation*}
\mathcal{B}_{d}(x\triangleleft_{d}y,z)=-\mathcal{B}_{d}(x,z\circ_{d}y)=-\mathcal{B}_{d}(x,z\circ_{A}y)=0,
\end{equation*}
which indicates $x\triangleleft_{d}y\in A$. Similarly we have
$$x\triangleright_{d}y\in A,\;a^{*}\triangleright_{d}b^{*}\in A^{*},\;a^{*}\triangleleft_{d}b^{*}\in A^{*},\;\forall x,y\in A, a^{*},b^{*}\in A^{*}.$$
Hence $(A,\triangleright_{A}=\triangleright_{d}|_{A})$ and
$(A^{*},\triangleright_{A^{*}}=\triangleright_{d}|_{A^{*}})$ are subalgebras of
$(A\oplus A^{*},\triangleright_{d},\triangleleft_{d})$. Therefore $\big((A\oplus
A^{*},\triangleright_{d},\triangleleft_{d},\mathcal{B}_{d}),(A,\triangleright_{A},\triangleleft_{A}),
(A^{*},\triangleright_{A^{*}},\triangleleft_{A^{*}})\big) $ is a Manin triple of
\apps.
\end{proof}

Recall that a {\bf perm coalgebra} \cite{LZB,Hou} is a vector
space $A$ with a co-multiplication $\eta:A\rightarrow A\otimes A$
such that the following equations hold:
\begin{eqnarray}
(  \eta\otimes\mathrm{id})  \eta(  x)  &=&(
\mathrm{id}\otimes\eta)  \eta(  x),\label{eq:co2}\\
(\mathrm{id}\otimes\eta)  \eta(  x) &=&(\tau\otimes\mathrm{id})(
\eta\otimes\mathrm{id})  \eta(  x),\;\forall x\in A.\label{eq:co3}
\end{eqnarray}

Now we introduce the notion of a apre-perm coalgebra.

\begin{defi}
Another Definition:
Let $A$ be a vector space with co-multiplications $\vartheta,\theta:A\rightarrow A\otimes A$, and $\eta=\vartheta+\theta$.
If $(A,\eta)$ is a perm coalgebra and the following equations hold:
\begin{eqnarray}
(\eta\otimes\mathrm{id})\varepsilon(x)&=&
(\mathrm{id}\otimes\varepsilon)\varepsilon(x),\label{eq:co1}\\
(\mathrm{id}\otimes\varepsilon)\varepsilon(x)&=&
(\tau\otimes\mathrm{id})(\mathrm{id}\otimes\varepsilon)\varepsilon(x),
\label{eq:co4}\\
(\mathrm{id}\otimes\eta)\theta(x)&=&-(\mathrm{id}\otimes\tau)
(\theta\otimes\mathrm{id})\theta(x),\label{eq:co5}\\
-(\mathrm{id}\otimes\tau)
(\theta\otimes\mathrm{id})\theta(x)&=&(\tau\otimes\mathrm{id})
(\mathrm{id}\otimes\theta)\varepsilon(x),\label{eq:co9}\\
(\tau\otimes\mathrm{id})
(\mathrm{id}\otimes\theta)\varepsilon(x)&=&
(\tau\otimes\mathrm{id})(\varepsilon\otimes\mathrm{id})\theta(x)
,\;\forall x\in A,\label{eq:co10}
\end{eqnarray}
where $\varepsilon:A\rightarrow A\otimes A$ is also a co-multiplication defined by
\begin{equation}\label{eq:co11}
\varepsilon(x)=\vartheta(x)+\tau\theta(x),\;\forall x\in A.
\end{equation}
Then we say $(A,\vartheta,\theta)$ is a {\bf apre-perm coalgebra}.
\end{defi}

\begin{pro}\label{lem:co}
Another Pro:
Let $A$ be a vector space and $\vartheta,\theta:A\rightarrow A\otimes A$ be co-multiplications.
Let $\triangleright_{A^{*}},\triangleleft_{A^{*}}:A^{*}\otimes A^{*}\rightarrow A^{*}$ be the linear duals of $\vartheta$ and $\theta$ respectively, that is, the following equations hold:
\begin{equation}
\langle a^{*}\triangleright_{A^{*}}b^{*},x\rangle=\langle a^{*}\otimes b^{*},\vartheta(x)\rangle,\;\langle a^{*}\triangleleft_{A^{*}}b^{*},x\rangle=\langle a^{*}\otimes b^{*},\theta(x)\rangle,\;\forall x\in A, a^{*},b^{*}\in A^{*}.
\end{equation}
Then $(A^{*},\triangleright_{A^{*}},\triangleleft_{A^{*}})$ is a \app if and only
if $(A,\vartheta,\theta)$ is a apre-perm coalgebra.
\end{pro}
\begin{proof}
Let $\eta=\vartheta+\theta$ and the linear dual of $\eta$ be $\circ_{A^{*}}$. Suppose that \eqref{eq:co11} holds and the linear dual of $\varepsilon$ is $\bullet_{A^{*}}$.
From a straightforward computation, we have
\begin{enumerate}
\item $(A,\eta)$ is a perm coalgebra if and only if
$(A^{*},\circ_{A^{*}})$ is a perm algebra.
\item \eqref{eq:co1} and \eqref{eq:co4} hold if and only if \eqref{eq:gppa2,2} holds on $A^{*}$.
\item \eqref{eq:co5}, \eqref{eq:co9} and \eqref{eq:co10} hold if and only if \eqref{eq:gppa3,2} holds on $A^{*}$.
\end{enumerate}
Hence the conclusion follows.
\end{proof}

\begin{defi}
Let $A$ be a vector space with co-multiplications $\vartheta,\theta:A\rightarrow A\otimes A$, and $\eta=\vartheta+\theta$.
If $(A,\eta)$ is a perm coalgebra and the following equations hold:
\begin{eqnarray}
(\mathrm{id}\otimes\theta)\vartheta(x)&=&-2(\tau\otimes\mathrm{id})
(\mathrm{id}\otimes\tau)
(\theta\otimes\mathrm{id})\theta(x),\label{eq:co1}\\
(\mathrm{id}\otimes\eta)\theta(x)&=&-(\mathrm{id}\otimes\tau)
(\theta\otimes\mathrm{id})\theta(x),\label{eq:co4}\\
-(\mathrm{id}\otimes\tau)(\theta\otimes\mathrm{id})\theta(x)
&=&(\tau\otimes\mathrm{id})
(\varepsilon\otimes\mathrm{id})\theta(x),\label{eq:co5}\\
(\eta\otimes\mathrm{id})\vartheta(x)&=&(\tau\otimes\mathrm{id})
(\mathrm{id}\otimes\varepsilon)\vartheta(x),\label{eq:co9}\\
(\tau\otimes\mathrm{id})(\mathrm{id}\otimes\varepsilon)\vartheta(x)
&=&(\mathrm{id}\otimes\vartheta)\varepsilon(x)
,\;\forall x\in A,\label{eq:co10}
\end{eqnarray}
where $\varepsilon:A\rightarrow A\otimes A$ is also a co-multiplication defined by
\begin{equation}\label{eq:co11}
\varepsilon(x)=\vartheta(x)+\tau\theta(x),\;\forall x\in A.
\end{equation}
Then we say $(A,\vartheta,\theta)$ is a {\bf apre-perm coalgebra}.
\end{defi}

\begin{pro}\label{lem:co}
Let $A$ be a vector space and $\vartheta,\theta:A\rightarrow A\otimes A$ be co-multiplications.
Let $\triangleright_{A^{*}},\triangleleft_{A^{*}}:A^{*}\otimes A^{*}\rightarrow A^{*}$ be the linear duals of $\vartheta$ and $\theta$ respectively, that is, the following equations hold:
\begin{equation}
\langle a^{*}\triangleright_{A^{*}}b^{*},x\rangle=\langle a^{*}\otimes b^{*},\vartheta(x)\rangle,\;\langle a^{*}\triangleleft_{A^{*}}b^{*},x\rangle=\langle a^{*}\otimes b^{*},\theta(x)\rangle,\;\forall x\in A, a^{*},b^{*}\in A^{*}.
\end{equation}
Then $(A^{*},\triangleright_{A^{*}},\triangleleft_{A^{*}})$ is a \app if and only
if $(A,\vartheta,\theta)$ is a apre-perm coalgebra.
\end{pro}
\begin{proof}
Let $\eta=\vartheta+\theta$ and the linear dual of $\eta$ be $\circ_{A^{*}}$. Suppose that \eqref{eq:co11} holds and the linear dual of $\varepsilon$ is $\bullet_{A^{*}}$.
From a straightforward computation, we have
\begin{enumerate}
\item \eqref{eq:co1} holds if and only if \eqref{eq:dpp1.1} holds on $A^{*}$.
\item $(A,\eta)$ is a perm coalgebra if and only if
$(A^{*},\circ_{A^{*}})$ is a perm algebra, that is \eqref{eq:dpp1.2} holds on $A^{*}$.
\item \eqref{eq:co4} and \eqref{eq:co5} hold if and only if \eqref{eq:dpp1.3} holds on $A^{*}$.
\item \eqref{eq:co9} and \eqref{eq:co10} hold if and only if
    \eqref{eq:dpp1.4} holds on $A^{*}$.
\end{enumerate}
Hence the conclusion follows from Proposition-Definition \ref{pdef:aa}.
\end{proof}

Now we give the notion of a \appb.

\begin{defi}
Let $(A,\triangleright_{A},\triangleleft_{A})$ be a \app. Suppose that there is a
apre-perm coalgebra $(A,\vartheta,\theta)$ and the following equations hold:
\begin{eqnarray}
\eta(x\bullet_{A}y)&=&\big(\mathcal{L}_{\bullet_{A}}(x)  \otimes\mathrm{id}\big)\eta(y)-\tau
\big(\mathcal{R}_{\bullet_{A}}(y)\otimes\mathrm{id}\big)\theta(x),  \label{eq:bialg1}\\
\eta(x\bullet_{A}y)&=& \big(\mathrm{id}\otimes\mathcal{R}_{\bullet_{A}}(y)\big)\varepsilon(x)
-\big(\mathcal{R}_{\triangleleft_{A}}(x)\otimes\mathrm{id}\big)
\eta(y),\label{eq:bialg2}\\
\eta(x\bullet_{A}y)&=&\big(\mathrm{id}\otimes\mathcal{L}
_{\bullet_{A}}(x)\big)\eta(y)+
\big(\mathcal{L}_{\triangleleft_{A}}(y)\otimes\mathrm{id}\big)
\theta(x),\label{eq:bialg3}\\
\eta(x\triangleleft_{A}y)&=&\big(  \mathrm{id}\otimes\mathcal{R}_{\triangleleft_{A}}(y)\big)
\eta(x)+\tau  \big(\mathrm{id}\otimes\mathcal{L}_{\triangleleft_{A}
}(x)\big)\varepsilon(y),\label{eq:bialg4}\\
\eta(x\triangleleft_{A}y)&=&\tau\eta(  x\triangleleft_{A}y),\label{eq:bialg5}\\
\varepsilon(x\circ_{A}y)&=&
\big(\mathrm{id}\otimes\mathcal{R}_{\circ_{A}}(y)\big)\varepsilon
(x)-\big(\mathcal{R}_{\triangleleft_{A}}(x)
\otimes\mathrm{id}\big)\varepsilon(y),\label{eq:bialg6}\\
\varepsilon(x\circ_{A}y)&=&\big(\mathcal{L}_{\bullet_{A}}(x)  \otimes\mathrm{id}\big)\varepsilon(y)-\tau
\big(\mathcal{R}_{\circ_{A}}(y)\otimes\mathrm{id}\big)\theta
(x),\label{eq:bialg7}\\
\varepsilon(x\circ_{A}y)&=&
\big(\mathrm{id}\otimes\mathcal{L}_{\circ_{A}}(x)\big)\varepsilon
(y)+\tau\big(\mathrm{id}
\otimes\mathcal{R}_{\triangleleft_{A}}(y)\big)\theta(x),\label{eq:bialg8}\\
\theta(x\circ_{A}y)&=&\big(\mathcal{L}_{\circ_{A}}(x)\otimes\mathrm{id}\big)
\theta(y)+  \big(\mathrm{id}\otimes\mathcal{L}_{\bullet_{A}}(y)\big)\theta(x),\label{eq:bialg9}\\
\theta(x\circ_{A}y)&=&\theta(y\circ_{A}x),\label{eq:bialg10}
\end{eqnarray}
for all $x,y\in A$. Then we say $(A,\triangleright_{A},\triangleleft_{A},\vartheta,\theta)$ is a \textbf{\appb}.
\end{defi}

\begin{lem}\cite{LZB,Hou}\label{lem:mp}
Let $(A,\circ_{A})$ and $(B,\circ_{B})$ be two perm algebras, and
$l_{A},r_{A}:A\rightarrow \mathrm{End}_{\mathbb K}(B)$ and
$l_{B},r_{B}:B\rightarrow\mathrm{End}_{\mathbb K}(A)$ be linear
maps. Then there is a perm algebra structure $(A\oplus B,\circ_{d})$ on $A\oplus B$ with $\circ_{d}$ defined
by
\begin{equation}
(x+a)\circ_{d}(y+b)=x\circ_{A}y+l_{B}(a)y+r_{B}(b)x+a\circ_{B}b+l_{A}(x)b+r_{A}(y)a,\;\forall x,y\in A, a,b\in B
\end{equation}
if and only if $(l_{A},r_{A},B)$ and $(l_{B},r_{B},A)$ are representations of $(A,\circ_{A})$ and $(B,\circ_{B})$ respectively, and the following equations hold:
\begin{eqnarray}
l_{_{A}}(  x)  (  a\circ_{B}b)  &=&l_{_{A}}(
x)  a\circ_{B}b+l_{_{A}}(  r_{_{B}}(  a)  x)  b,\label{eq:mp1}\\
l_{_{A}}(  x)  (  a\circ_{B}b)&=&r_{_{A}}(  x)  a\circ_{B}b+l_{_{A}}(
l_{_{B}}(  a)  x)  b,\label{eq:mp2}\\
l_{_{A}}(  x)  (  a\circ_{B}b)&=&a\circ_{B}l_{_{A}}(  x)  b+r_{_{A}}(
r_{_{B}}(  b)  x)  a,\label{eq:mp3}\\
r_{_{A}}(  x)  (  a\circ_{B}b)&=&a\circ_{B}r_{_{A}%
}(  x)  b+r_{_{A}}(  l_{_{B}}(  b)  x)  a,\label{eq:mp4}\\
r_{_{A}}(  x)  (  a\circ_{B}b)  &=&r_{_{A}}(
x)  (  b\circ_{B}a),\label{eq:mp5}\\
l_{_{B}}(  a)  (  x\circ_{A}y)  &=&l_{_{B}}(
a)  x\circ_{A}y+l_{_{B}}(  r_{_{A}}(  x)  a)  y,
\label{eq:mp6}\\
l_{_{B}}(  a)  (  x\circ_{A}y)&=&r_{_{B}}(  a)  x\circ_{A}y+l_{_{B}}(
l_{_{A}}(  x)  a)  y,\label{eq:mp7}\\
l_{_{B}}(  a)  (  x\circ_{A}y)&=&x\circ_{A}l_{_{B}}(  a)  y+r_{_{B}}(
r_{_{A}}(  y)  a)  x,\label{eq:mp8}\\
r_{_{B}}(  a)  (  x\circ_{A}y)  &=&x\circ_{A}r_{_{B}%
}(  a)  y+r_{_{B}}(  l_{_{A}}(  y)  a)  x,\label{eq:mp9}\\
r_{_{B}}(  a)  (  x\circ_{A}y)  &=&r_{_{B}}(
a)  (  y\circ_{A}x).\label{eq:mp10}
\end{eqnarray}
\end{lem}

\begin{thm}\label{thm:2-2}
Let $(A,\triangleright_{A},\triangleleft_{A})$ and
$(A^{*},\triangleright_{A^{*}},\triangleleft_{A^{*}})$ be two \apps, and linear
maps $\vartheta,\theta:A\rightarrow A\otimes A$ be the linear
duals of $\triangleright_{A^{*}}$ and $\triangleleft_{A^{*}}$ respectively. Then
$(A,\triangleright_{A},\triangleleft_{A},\vartheta,\theta)$ is a \appb if and only if there is a perm algebra $(A\oplus A^{*},\circ_{d})$ with $\circ_{d}$
defined by \eqref{eq:A ds}.
\end{thm}

\begin{proof}
Since $(A,\triangleright_{A},\triangleleft_{A})$ and $(A^{*},\triangleright_{A^{*}},\triangleleft_{A^{*}})$
are \apps, $(\mathcal{L}^{*}_{\bullet_{A}},-\mathcal{R}
^{*}_{\triangleleft_{A}},A^{*})$ and $(\mathcal{L}^{*}_{\bullet_{A^{*}}},-\mathcal{R}^{*}_{\triangleleft_{A^{*}}},A)$ are representations of the sub-adjacent perm algebras $(A,\circ_{A})$ and $
(A^{*},\circ_{A^{*}})$ respectively, and \eqref{eq:co1}-\eqref{eq:co10} hold
by Proposition \ref{lem:co}. For all $x,y\in A, a^{*},b^{*}\in A^{*}$, we have
\begin{eqnarray*}
&&\langle \mathcal{L}_{\bullet_{A}}^{*}( x) (
a^{*}\circ_{A^{*}}b^{*}),y\rangle=\langle a^{*}\circ_{A^{*}}b^{*},
x\bullet_{A}y\rangle =\langle a^{*}\otimes b^{*},\eta( x\bullet_{A}y) \rangle,
\\
&&\langle \mathcal{L}^{*}_{\bullet_{A}}(x)a^{*}\circ_{A^{*}}b^{*},y\rangle=
\langle\mathcal{L}^{*}_{\bullet_{A}}(x)a^{*}\otimes
b^{*},\eta(y)\rangle=\langle a^{*}\otimes b^{*},\big( \mathcal{L}
_{\bullet_{A}}(x)\otimes\mathrm{id} \big)\eta(y)\rangle, \\
&&\langle\mathcal{L}^{*}_{\bullet_{A}}\big( -\mathcal{R}^{*}_{
   \triangleleft_{A^{*}}}(a^{*})x\big)b^{*},y\rangle=\langle b^{*},-\mathcal{R}
^{*}_{\triangleleft_{A^{*}}}(a^{*})x\bullet_{A}y\rangle=\langle \mathcal{R}
^{*}_{\bullet_{A}}(y)b^{*},-\mathcal{R}^{*}_{\triangleleft_{A^{*}}}(a^{*})x\rangle \\
&&=-\langle a^{*}\triangleleft_{A^{*}}\mathcal{R}^{*}_{\circ_{A}}(y)b^{*},x\rangle=
-\langle \mathcal{R}^{*}_{\bullet_{A}}(y)b^{*}\otimes a^{*},\theta(x)\rangle=
-\langle a^{*}\otimes b^{*},\tau\big(\mathcal{R}
_{\bullet_{A}}(y)\otimes\mathrm{id}\big)\theta(x)\rangle.
\end{eqnarray*}
Thus \eqref{eq:bialg1} holds if and only if \eqref{eq:mp1} holds for $l_{A}=
\mathcal{L}^{*}_{\bullet_{A}},r_{A}=-\mathcal{R}^{*}_{\triangleleft_{A}}, l_{B}=
\mathcal{L}^{*}_{\bullet_{A^{*}}},r_{B}=-\mathcal{R}^{*}_{\triangleleft_{A^{*}}}. $
Similarly, we have
\begin{eqnarray*}
&&\eqref{eq:bialg2}\Leftrightarrow\eqref{eq:mp2},\;
\eqref{eq:bialg3}\Leftrightarrow\eqref{eq:mp3},\;
\eqref{eq:bialg4}\Leftrightarrow\eqref{eq:mp4},\;
\eqref{eq:bialg5}\Leftrightarrow\eqref{eq:mp5},\; \eqref{eq:bialg6}
\Leftrightarrow\eqref{eq:mp6},\\
&&\eqref{eq:bialg7}\Leftrightarrow
\eqref{eq:mp7},\;
\eqref{eq:bialg8}\Leftrightarrow\eqref{eq:mp8},\;
\eqref{eq:bialg9}\Leftrightarrow\eqref{eq:mp9},\;
\eqref{eq:bialg10}\Leftrightarrow\eqref{eq:mp10},
\end{eqnarray*}
where $l_{A}=\mathcal{L}^{*}_{\bullet_{A}},r_{A}=-\mathcal{R}^{*}_{\triangleleft_{A}},
l_{B}=\mathcal{L}^{*}_{\bullet_{A^{*}}},r_{B}=-\mathcal{R}^{*}_{\triangleleft_{A^{*}}}
$. Hence the conclusion follows from Lemma \ref{lem:mp}.
\end{proof}

Combining Theorems \ref{thm:Manin triple} and \ref{thm:2-2}, we have the following result.

\begin{cor}\label{cor:5.8}
Let $(A,\triangleright_{A},\triangleleft_{A})$ and
$(A^{*},\triangleright_{A^{*}},\triangleleft_{A^{*}})$ be two \apps.
Then the following conditions are equivalent:
\begin{enumerate}
    \item There is a Manin triple $\big(  (  A\oplus
    A^{*},\circ_{d},\mathcal{B}_{d}),(A,\circ_{A}),(A^{*},\circ_{A^{*}})\big)
    $ of perm algebras associated to the nondegenerate
    symmetric quasi-left-invariant bilinear form such that the compatible \app $(A\oplus A^{*},\triangleright_{d},\triangleleft_{d})$ induced from $\mathcal{B}_{d}$ contains $(A,\triangleright_{A},\triangleleft_{A})$ and
    $(A^{*},\triangleright_{A^{*}},\triangleleft_{A^{*}})$ as subalgebras.
    \item There is a Manin triple $\big(  (  A\oplus
    A^{*},\triangleright_{d},\triangleleft_{d},\mathcal{B}_{d}),
    (A,\triangleright_{A},\triangleleft_{A}),
    (A^{*},\triangleright_{A^{*}},\triangleleft_{A^{*}})\big)
    $ of \apps.
    \item $(A,\triangleright_{A},\triangleleft_{A},\vartheta,\theta)$ is a \appb, where $\vartheta,\theta:A\rightarrow A\otimes A$ are the linear
    duals of $\triangleright_{A^{*}}$ and $\triangleleft_{A^{*}}$ respectively.
\item
There is a \app $(A\oplus A^{*},\triangleright_{d},\triangleleft_{d})$ with $\triangleright_{d},\triangleleft_{d}$
defined by \eqref{eq:A ds1} and \eqref{eq:A ds2} respectively.
\end{enumerate}
\end{cor}}

\section{Special apre-perm bialgebras}\label{sec:5}\
We introduce the notion of the special apre-perm Yang-Baxter equation (SAPP-YBE), and show that a solution of the SAPP-YBE in a \sapp whose symmetric part is invariant gives rise to a \sappb that we call quasi-triangular. Furthermore, we introduce the notion of $\mathcal{O}$-operators of \sapps in order to characterize the SAPP-YBE in terms of operator forms. 
Moreover,  quasi-triangular \sappbs can be obtained from quasi-triangular averaging commutative and cocommutative infinitesimal bialgebras.
We also study triangular   and factorizable \sappbs as subclasses of quasi-triangular \sappbs.

\subsection{Quasi-triangular special apre-perm bialgebras, the special apre-perm Yang-Baxter equation and $\mathcal{O}$-operators}\label{sec4.1}\

\begin{defi}
A {\bf \sapp} is a triple $(A,\triangleright_{A},\triangleleft_{A})$, such that $A$ is a vector space, $\triangleright_{A},\triangleleft_{A}:A\otimes A\rightarrow A$ are multiplications on $A$ and the following conditions hold:
\begin{enumerate}
	\item the multiplication $\triangleleft_{A}$ is commutative.
	\item $(A,\circ_{A})$ is a perm algebra, where the multiplication $\circ_{A} :A\otimes A\rightarrow A$ is given by
	\begin{eqnarray}\label{1987}
	x\circ_{A}y=x\triangleright_{A}y+x\triangleleft_{A}y,\;\forall x,y\in A.
	\end{eqnarray}  
	\item the following equation holds:
	 \begin{equation}\mlabel{eq:SDPP}
		(x\circ_{A}y)\triangleleft_{A}z=x\circ_{A}(y\triangleleft_{A}z)=-x\triangleleft_{A}(y\triangleleft_{A}z),\;\forall x,y,z\in A.
	\end{equation}
\end{enumerate}
\end{defi}

Let $(A,\triangleright_{A},\triangleleft_{A})$ be a \sapp. Then $(A,\circ_{A})$ given by \eqref{1987} is called the {\bf sub-adjacent perm algebra} of $(A,\triangleright_{A},\triangleleft_{A})$, and $(A,\triangleright_{A},\triangleleft_{A})$ is called a {\bf compatible \sapp} of $(A,\circ_{A})$. Moreover, by \cite{Bai2024},  
$(\mathcal{L}^{*}_{\circ_{A}},-\mathcal{L}^{*}_{\triangleleft_{A}},A^{*})$ and  $(\mathcal{L}_{\circ_{A}},\mathcal{L}_{\circ_{A}}+\mathcal{L}_{\triangleleft_{A}},A)$ are representations of $(A,\circ_{A})$, thus giving a new splitting of perm algebras besides pre-perm algebras \cite{LZB}.

\begin{defi}\cite{Bai2024}
A {\bf special apre-perm coalgebra} is a triple $(A,\vartheta,\theta)$, such that $A$ is a vector space and $\vartheta,\theta:A\rightarrow A\otimes A$ are co-multiplications satisfying the following equations:
\begin{eqnarray}
(  \eta\otimes\mathrm{id})  \eta(  x)  &=&(
	\mathrm{id}\otimes\eta)  \eta(  x),\label{eq:co2}\\
	(\mathrm{id}\otimes\eta)  \eta(  x) &=&(\tau\otimes\mathrm{id})(
	\eta\otimes\mathrm{id})  \eta(  x),\label{eq:co3}\\
\theta(x)&=&\tau\theta(x),\label{eq:co1} \\
	(  \eta\otimes\mathrm{id})  \theta(  x)  &=&(
	\mathrm{id}\otimes\theta)  \eta(  x),\label{eq:co4}\\
	(\mathrm{id}\otimes\theta) (\eta+\theta)(x) &=&0,\;\;\forall x\in A,\label{eq:co5}
\end{eqnarray}
where $\eta=\theta+\vartheta$.
\end{defi}

Let $A$ be a vector space, $\vartheta,\theta:A\rightarrow A\otimes A$ be co-multiplications and $\triangleright_{A^{*}},\triangleleft_{A^{*}}:A^{*}\otimes A^{*}\rightarrow A^{*}$ be the linear duals of $\vartheta$ and $\theta$ respectively.
Then $(A,\vartheta,\theta)$ is a special apre-perm coalgebra if and only if $(A^{*},\triangleright_{A^{*}},\triangleleft_{A^{*}} )$ is a special apre-perm  algebra.

Now let us recall the definition of \sappbs.

\begin{defi}\cite{Bai2024}
	Let $(A,\triangleright_{A},\triangleleft_{A})$ be a \sapp and $(A,\vartheta,\theta)$ be a special apre-perm coalgebra.
	Suppose that the following equations hold:
	\begin{eqnarray}
		\eta(  x\circ_{A}y)  &=&\big(  \mathcal{L}_{\circ_{A}}(
		x)  \otimes\mathrm{id}\big)  \eta(  y)-\big(
		\mathrm{id}\otimes\mathcal{R}_{\circ_{A}}(  y)  \big)
		\theta(  x),  \label{eq:bialg1}\\
		\eta(  x\circ_{A}y)  &=&
		\big(  \mathrm{id}\otimes\mathcal{R}
		_{\circ_{A}}(  y)  \big)  \eta(  x)-\big(
		\mathcal{L}_{\triangleleft_{A}}(  x)  \otimes\mathrm{id}\big)
		\eta(  y) ,\label{eq:bialg2}\\
		\eta(  x\circ_{A}y)  &=&\big(  \mathrm{id}\otimes\mathcal{L}%
		_{\circ_{A}}(  x)  \big)  \eta(  y)  +\big(
		\mathcal{L}_{\triangleleft_{A}}(  y)  \otimes\mathrm{id}\big)
		\theta(  x),\label{eq:bialg3}\\
		\eta(  x\triangleleft_{A}y)&=&\big(  \mathrm{id}\otimes\mathcal{L}_{\triangleleft_{A}}(  x)  \big)
		\eta(  y)  +\tau\big(  \mathrm{id}\otimes\mathcal{L}_{\triangleleft_{A}
		}(  y)  \big)  \eta(  x), \label{eq:bialg4}\\
		\eta(  x\triangleleft_{A}y)  &=&\tau\eta(  x\triangleleft_{A}y),\label{eq:bialg5}\\
		\theta(  x\circ_{A}y)&=&\big(
		\mathrm{id}\otimes\mathcal{L}_{\circ_{A}}(  x)  \big)
		\theta(  y)  +\big(  \mathcal{L}_{\circ_{A}}(  y)
		\otimes\mathrm{id}\big)  \theta(  x),\label{eq:bialg6}\\
		\theta(  x\circ_{A}y)  &=&\theta(  y\circ_{A}x),\label{eq:bialg7}
	\end{eqnarray}
	for all $x,y\in A$. Such a structure is called a \textbf{\sappb} and is denoted by $(A,\triangleright_{A},\triangleleft_{A},\vartheta,\theta)$.
\end{defi}

\begin{defi}
Let $(A,\triangleright_{A},\triangleleft_{A})$ be a \sapp and $(A,\circ_{A})$ be the sub-adjacent perm algebra.
Let $r=\sum\limits_{i}u_{i}\otimes v_{i}\in A\otimes A$ and set
\begin{equation}\label{eq:Ps}
    SA(r)=\sum_{i,j} u_{i}\circ_{A} u_{j}\otimes v_{i}\otimes v_{j}+u_{i}\otimes v_{i}\triangleleft_{A} u_{j}\otimes v_{j}+u_{i}\otimes u_{j}\otimes v_{j}\circ_{A} v_{i}.
\end{equation}
We say $r$ is a solution of the \textbf{special apre-perm Yang-Baxter equation} (or \textbf{SAPP-YBE} in short) if $SA(r)=0$.
\end{defi}

\delete{
Recall \cite{Bai2024} that an element $r\in A\otimes A$ is called {\bf invariant} on a SDPP algebra $(A, \triangleright_{A}, \triangleleft_{A})$, such that $(A,\circ_{A})$ is the sub-adjacent perm algebra of $(A, \triangleright_{A}, \triangleleft_{A})$, $f,g:A\rightarrow \mathrm{End}(A\otimes A)$ are linear maps given by
\begin{eqnarray}
&&f(x)=\mathrm{id}\otimes\mathcal{R}_{\circ_{A}}(x)-\mathcal{L}_{ \triangleleft_{A}}(x)\otimes\mathrm{id},\\
&&g(x)=\mathrm{id}\otimes\mathcal{L}_{\circ_{A}}(x)-\mathcal{L}_{\circ_{A}}(x)\otimes\mathrm{id},\;\forall x\in A,
\end{eqnarray}
and $f(x)r=g(x)r=0$ for all $x\in A$.

Now we investigate the tensor
form of the bilinear form $\mathcal{B}$ in a quadratic \sapp
$(A, \triangleright_{A}, \triangleleft_{A},\mathcal{B})$. Let $\mathcal{B}$ be a
nondegenerate bilinear form on a vector space $A$ which
corresponds to a linear map $\mathcal{B}^{\natural}:A\rightarrow
A^{*}$ by
\begin{equation}\label{eq:B}
\mathcal{B}(x,y)=\langle\mathcal{B}^{\natural}(x),y\rangle,\;\forall x,y\in A.
\end{equation}Define a 2-tensor $\phi_{\mathcal{B}}\in A\otimes A$ to be
the tensor form of ${\mathcal{B}^{\natural}}^{-1}$, that is,
\begin{equation}\label{eq:2-tensor}
\langle \phi_{\mathcal{B}},a^{*}\otimes b^{*}\rangle=\langle {\mathcal{B}
^{\natural}}^{-1}(a^{*}),b^{*}\rangle,\;\forall a^{*},b^{*}\in A^{*}.
\end{equation}}

\begin{defi}
Let $(A,\triangleright_{A},\triangleleft_{A})$ be a \sapp and $(A,\circ_{A})$ be the sub-adjacent perm algebra. Set linear maps
$f,g:A\rightarrow \mathrm{End}_{\mathbb K}(A\otimes A)$ by
\begin{eqnarray}
&&f(x)=\mathrm{id}\otimes\mathcal{R}_{\circ_{A}}(x)+\mathcal{L}_{\triangleleft_{A}}(x)\otimes\mathrm{id},\\
&&g(x)=\mathcal{L}_{\circ_{A}}(x)\otimes\mathrm{id}-\mathrm{id}\otimes\mathcal{L}_{\circ_{A}}(x),\;\forall x\in A.
\end{eqnarray}
An element  $r\in A\otimes A$ is called {\bf invariant} on $(A,\triangleright_{A},\triangleleft_{A})$ if $f(x)r=g(x)r=0$ for all $x\in A$.
\end{defi}

\begin{pro}\label{pro:co}
Let $(A,\triangleright_{A},\triangleleft_{A})$ be a \sapp and $r\in A\otimes A$.
Let linear maps $\eta_{r},\vartheta_{r},\theta_{r}:A\rightarrow A\otimes A$ be given by
\begin{equation}\label{eq:comul}
    \eta_{r}(x)=f(x)r,\;\theta_{r}(x)=g(x)r,\;
    \vartheta_{r}(x)=(\eta_{r}-\theta_{r})(x)=(f-g)(x)r,\;\forall x\in A.
\end{equation}
\begin{enumerate}
	\item\label{pro:co2} \eqref{eq:co2} holds if and only if the following equation holds:
{\small
		\begin{equation}\label{eq:pro:co2} \big(\mathrm{id}\otimes\mathrm{id}\otimes\mathcal{R}_{\circ_{A}}(x)\big)\Big(\sum_{j}f(u_{j})\big( r+\tau(r) \big)\otimes v_{j}-(\tau\otimes\mathrm{id})SA(r)\Big)-\big( \mathcal{L}_{\triangleleft_{A}}(x)\otimes\mathrm{id}\otimes\mathrm{id} \big)SA(r)=0.
		\end{equation}}
\item\label{pro:co3} \eqref{eq:co3} holds if and only if the following equation holds:
		\begin{align}
			&\big( \mathrm{id}\otimes\mathrm{id}\otimes\mathcal{R}_{\circ_{A}}(x)+\mathcal{L}_{\triangleleft_{A}}(x)\otimes\mathrm{id}\otimes\mathrm{id}  \big)\bigg(SA(r)-\sum_{j}\Big( \tau f(u_{j})\big(r+\tau(r)\big)\Big)\otimes v_{j} \bigg)\notag\\
			&+\sum_{j}\big(  \mathrm{id}\otimes\mathcal{L}_{\triangleleft_{A}}(u_{j})\otimes\mathrm{id} \big)\Big( f(x)\big( r+\tau(r)\big)\otimes v_{j} \Big)=0.\label{eq:pro:co3}
		\end{align}
	\item\label{pro:co1} \eqref{eq:co1} holds if and only if the following equation holds:
		\begin{equation}\label{eq:pro:co1}
			g(x)\big( r+\tau(r)\big)=0.
		\end{equation}
\item\label{pro:co4} \eqref{eq:co4} holds if and only if the following equation holds:
		\begin{equation}\label{eq:pro:co4}
			\big(  -\mathcal{L}_{\triangleleft_{A}}(x)\otimes\mathrm{id}\otimes\mathrm{id}- \mathrm{id}\otimes\mathrm{id}\otimes\mathcal{L}_{\circ_{A}}(x) \big)\Big((\tau\otimes\mathrm{id})SA(r)-\sum_{j}f(u_{j})\big( r+\tau(r) \big)\otimes v_{j}\Big)=0.
		\end{equation}
\item\label{pro:co5} \eqref{eq:co5} holds if and only if the following equation holds:
		\begin{equation}\label{eq:pro:co5} \big((\mathcal{L}_{\circ_{A}}+\mathcal{L}_{\triangleleft_{A}}) (x)\otimes\mathrm{id}\otimes\mathrm{id}\big)\Big((\tau\otimes\mathrm{id})SA(r)-\sum_{j}f(u_{j})\big( r+\tau(r) \big)\otimes v_{j}\Big)=0.
		\end{equation}
\end{enumerate}
\end{pro}

\begin{proof}
We only prove Item (\ref{pro:co2}), and other items are obtained similarly.
For all $x\in A$, we have
\begin{eqnarray*}
&&(  \eta_{r}\otimes\mathrm{id})  \eta_{r}(  x)-
(\mathrm{id}\otimes\eta_{r})  \eta_{r}(  x), \\
&=&\sum_{i,j}u_{j}\otimes v_{j}\circ_{A} u_{i}\otimes v_{i}\circ_{A} x+u_{i}\triangleleft_{A}
u_{j}\otimes v_{j}\otimes v_{i}\circ_{A} x+u_{j}\otimes v_{j}\circ_{A} (x\triangleleft_{A}
u_{i})\otimes v_{i}\\
&&\ \ +(x\triangleleft_{A} u_{i})\triangleleft_{A} u_{j}\otimes v_{j}\otimes v_{i}
-u_{i}\otimes u_{j}\otimes v_{j}\circ_{A} (v_{i}\circ_{A} x)-u_{i}\otimes
(v_{i}\circ_{A} x)\triangleleft_{A} u_{j}\otimes v_{j}\\
&&\ \ -x\triangleleft_{A} u_{i}\otimes u_{j}\otimes
v_{j}\circ_{A} v_{i}-x\triangleleft_{A} u_{i}\otimes v_{i}\triangleleft_{A} u_{j}\otimes v_{j}\\
&=&A(1)+A(2)+A(3),
\end{eqnarray*}
where \eqref{1987} holds and
\begin{align*}
A(1)=&\sum_{i,j}u_{j}\otimes v_{j}\circ_{A} u_{i}\otimes v_{i}\circ_{A}
x+u_{i}\triangleleft_{A} u_{j}\otimes v_{j}\otimes v_{i}\circ_{A} x-u_{i}\otimes
u_{j}\otimes v_{j}\circ_{A} (v_{i}\circ_{A} x)\\
=&\sum_{i,j}u_{j}\otimes v_{j}\circ_{A} u_{i}\otimes
v_{i}\circ_{A} x+u_{i}\triangleleft_{A} u_{j}\otimes v_{j}\otimes v_{i}\circ_{A} x-u_{i}\otimes
u_{j}\otimes (v_{i}\circ_{A} v_{j})\circ_{A} x \\
=&\big(\mathrm{id}\otimes \mathrm{id}\otimes \mathcal{R}_{\circ_{A}}(x)
\big)\Big(\sum_{i,j}u_{j}\otimes v_{j}\circ_{A} u_{i}\otimes v_{i}+u_{i}\triangleleft_{A}
u_{j}\otimes v_{j}\otimes v_{i}-u_{i}\otimes u_{j}\otimes v_{i}\circ_{A} v_{j}
\Big) \\
=&\big(\mathrm{id}\otimes \mathrm{id}\otimes \mathcal{R}_{\circ_{A}}(x)
\big)\Big(\sum_{j}f(u_{j})\big(r+\tau (r)
\big)\otimes v_{j}-(\tau \otimes \mathrm{id})SA(r)\Big), \\
A(2) =&\sum_{i,j}u_{j}\otimes v_{j}\circ_{A}(x\triangleleft_{A} u_{i})\otimes
v_{i}-u_{i}\otimes (v_{i}\circ_{A} x)\triangleleft_{A} u_{j}\otimes v_{j} \\
\overset{\eqref{eq:SDPP}}{=}&0, \\
A(3) =&\sum_{i,j}(x\triangleleft_{A} u_{i})\triangleleft_{A} u_{j}\otimes v_{j}\otimes v_{i}-x\triangleleft_{A}
u_{i}\otimes u_{j}\otimes v_{j}\circ_{A} v_{i}-x\triangleleft_{A} u_{i}\otimes v_{i}\triangleleft_{A}
u_{j}\otimes v_{j} \\
\overset{\eqref{eq:SDPP}}{=}&\sum_{i,j}-x\triangleleft_{A} (u_{j}\circ_{A} u_{i})\otimes
v_{j}\otimes v_{i}-x\triangleleft_{A} u_{i}\otimes u_{j}\otimes v_{j}\circ_{A} v_{i}-x\triangleleft_{A}
u_{i}\otimes v_{i}\triangleleft_{A} u_{j}\otimes v_{j} \\
=&\big(-\mathcal{L}_{\triangleleft_{A}}(x)\otimes \mathrm{id}\otimes \mathrm{id}
\big)SA(r).
\end{align*}
Hence the conclusion follows.
\end{proof}

\begin{pro}\label{pro:mp}
	Let $(A,\triangleright_{A},\triangleleft_{A})$ be a \sapp and $r\in A\otimes A$.
	Let linear maps $\vartheta_{r},\theta_{r}:A\rightarrow A\otimes A$ be given by \eqref{eq:comul}.
	\begin{enumerate}
		\item\label{pro:mp1}  \eqref{eq:bialg1} holds automatically.
		\item\label{pro:mp2} \eqref{eq:bialg2} holds automatically.
		\item\label{pro:mp3} \eqref{eq:bialg3} holds automatically.
		\item\label{pro:mp4} \eqref{eq:bialg4} holds if and only if the following equation holds:
		\begin{equation}\label{eq:pro:mp4} \big(\mathcal{L}_{\triangleleft_{A}}(y)\otimes\mathrm{id}\big)\tau\Big( f(x)\big(r+\tau(r)\big)\Big)=0,\;\forall x,y\in A.
		\end{equation}
		\item\label{pro:mp5} \eqref{eq:bialg5} holds if and only if the following equation holds:
		\begin{equation}\label{eq:pro:mp5}
			f(x\triangleleft_{A} y)\big( r+\tau(r)\big)=0.
		\end{equation}
		\item\label{pro:mp6} \eqref{eq:bialg6} holds automatically.
		\item\label{pro:mp7} \eqref{eq:bialg7} holds automatically.
	\end{enumerate}
\end{pro}

\begin{proof}
	We only prove Item (\ref{pro:mp4}), and other items are obtained similarly.
	For all $x,y\in A$, we have
		\begin{eqnarray*}
			&&\big(\mathrm{id}\otimes \mathcal{L}_{\triangleleft
				_{A}}(x)\big)\eta _{r}(y)+\tau \big(\mathrm{id}\otimes \mathcal{L}_{\triangleleft
				_{A}}(y)\big)\eta _{r}(x)-\eta _{r}(x\triangleleft_{A}y)\\
			&=&\sum_{i,j}-u_{i}\otimes v_{i}\circ _{A}(x\triangleleft_{A}y)-(x\triangleleft_{A}y)\triangleleft_{A}
			u_{i}\otimes v_{i}+u_{i}\otimes x\triangleleft_{A}(v_{i}\circ _{A}y) \\
			&&+y\triangleleft_{A}u_{i}\otimes x\triangleleft_{A}v_{i}+y\triangleleft_{A}(v_{i}\circ
			_{A}x)\otimes u_{i}+y\triangleleft_{A}v_{i}\otimes x\triangleleft_{A}u_{i} \\
			&\overset{\eqref{eq:SDPP}}{=}&\sum_{i}y\triangleleft _{A}(u_{i}\circ
			_{A}x)\otimes v_{i}+y\triangleleft_{A}(v_{i}\circ _{A}x)\otimes u_{i}+y\triangleleft
			_{A}u_{i}\otimes x\triangleleft_{A}v_{i}+y\triangleleft_{A}v_{i}\otimes x\triangleleft_{A}u_{i} \\
			&=&\big(\mathcal{L}_{\triangleleft_{A}}(y)\otimes \mathrm{id}\big)\tau \Big(f(x)
			\big(r+\tau (r)\big)\Big).
		\end{eqnarray*}
	Hence the conclusion follows.
\end{proof}

Combining Propositions \ref{pro:co} and \ref{pro:mp} together, we have the
following result.

\begin{thm}\label{thm:bialg}
	Let $(A,\triangleright_{A},\triangleleft_{A})$ be a \sapp and $r\in A\otimes A$.
	Let $\vartheta_{r},\theta_{r}:A\rightarrow A\otimes A$ be co-multiplications given by \eqref{eq:comul}.
	Then $(A,\triangleright_{A},\triangleleft_{A},\vartheta_{r},\theta_{r})$ is a \sappb if and only if \eqref{eq:pro:co1}-\eqref{eq:pro:mp5} hold.
	In particular, if $r$ is a solution of the SAPP-YBE and the symmetric part of $r$ is invariant on $(A,\triangleright_{A},\triangleleft_{A})$, that is,
	\eqref{eq:pro:co1} and
	the following equation  hold:
	\begin{equation}\label{eq:invariance}
		f(x)\big( r+\tau(r)\big )= 0,\;\forall x\in A,
	\end{equation}
	then $(A,\triangleright_{A},\triangleleft_{A},\vartheta_{r},\theta_{r})$ is a \sappb.
\end{thm}

\begin{defi}
	Suppose that $(A,\triangleright_{A},\triangleleft_{A})$ is a \sapp. If there
	exists a solution of the SAPP-YBE $r\in A\otimes A$ whose symmetric part is
	invariant on $(A,\triangleright_{A},\triangleleft_{A})$, then the resulting \sappb $(A,\triangleright_{A},
	\triangleleft_{A},\vartheta_{r},\theta_{r})$ by Theorem \ref{thm:bialg} is called
	\textbf{quasi-triangular}.
\end{defi}

\delete{
\begin{pro}\label{pro:4,2}
		Let $(A,\triangleright_{A},\triangleleft_{A})$ be a \sapp and $r=\sum\limits_{i}u_{i}\otimes v_{i}\in A\otimes A$.
		Then we have
		\begin{equation*}
			SD\big( -\tau(r) \big)=\tau_{13}\big(SA(r)\big),
		\end{equation*}
		where $\tau_{13}(x\otimes y\otimes z)=z\otimes y\otimes x$, for all $x,y,z\in A$.
		Consequently, $r$ is a solution of the SAPP-YBE in $(A,\triangleright_{A},\triangleleft_{A})$ if and only if $-\tau(r)$ is a solution of the SAPP-YBE in $(A,\triangleright_{A},\triangleleft_{A})$. Furthermore, if $(A,\triangleright_{A},\triangleleft_{A},\vartheta_{r},\theta_{r})$ is a quasi-triangular \sappb, then $(A,\triangleright_{A},\triangleleft_{A},\eta_{-\tau(r)},\theta_{-\tau(r)})$ is also a quasi-triangular \sappb.
\end{pro}
	
\begin{proof}
It is straightforward.
\end{proof}}

Next we show that quasi-triangular averaging commutative and cocommutative infinitesimal bialgebras render quasi-triangular \sappbs.

\begin{lem}\cite{Bai2024}\label{2474}
Let $(A,\cdot_{A},\Delta,P,Q)$ be an averaging commutative and cocommutative infinitesimal bialgebra.
Let  $\triangleright_{A},\triangleleft_{A}:A\otimes A\rightarrow A$ be multiplications given by
\begin{equation}\label{eq:com asso and SDPP}
	x\triangleright_{A}y=P(x)\cdot_{A}y+Q(x\cdot_{A}y),\; x\triangleleft_{A}y=-Q(x\cdot_{A}y),\;\forall x,y\in A,
\end{equation}
and $\vartheta,\theta:A\rightarrow A\otimes A$ be co-multiplications given by
\begin{equation}\label{eq:co re}
	\vartheta (x)=(Q\otimes\mathrm{id})\Delta (x)+\Delta (Px),\; \theta (x)=-\Delta (Px),\;\forall x\in A.
\end{equation}
Then $(A,\triangleright_{A},\triangleleft_{A},\vartheta ,\theta )$ is a \sappb.
\end{lem}

\begin{pro}\label{pro:quasi bialgebras}
	Let $(A,\cdot_{A},\Delta_{r} ,P,Q)$ be a quasi-triangular averaging commutative and cocommutative infinitesimal bialgebra and $(A,\triangleright_{A},\triangleleft_{A},\vartheta ,\theta )$ be the \sappb given in Lemma \ref{2474}.
	Then the following conditions hold:
	\begin{enumerate}
		\item\label{2492} $r+\tau(r)$ is invariant on $(A,\triangleright_{A},\triangleleft_{A})$.
		\item\label{2493}  $r$ satisfies the SAPP-YBE in $(A,\triangleright_{A},\triangleleft_{A})$.
		\item\label{2494} the following equations hold:
		\begin{eqnarray}
			\vartheta(x)=(f-g)(x)r=\vartheta_{r}(x),\;
			\theta(x)=g(x)r=\theta_{r}(x),\;\forall x\in A,
		\end{eqnarray}
	\end{enumerate}
such that $(A,\triangleright_{A},\triangleleft_{A},\vartheta=\vartheta_{r} ,\theta=\theta_{r} )$ is a quasi-triangular \sappb.
\delete{
	If there exists an $r\in A\otimes A$ such that  $(A,\cdot_{A},\Delta=\Delta_{r},P,Q)$ is quasi-triangular, then $(A,\triangleright_{A},\triangleleft_{A},\vartheta ,\theta )$
	Let  $\triangleright_{A},\triangleleft_{A}:A\otimes A\rightarrow A$ be given by
\begin{equation}\label{eq:com asso and SDPP}
x\triangleright_{A}y=P(x)\cdot_{A}y+Q(x\cdot_{A}y),\; x\triangleleft_{A}y=-Q(x\cdot_{A}y),\;\forall x,y\in A,
\end{equation}
and $\vartheta_{r},\theta_{r}:A\rightarrow A\otimes A$ be given by
\begin{equation}\label{eq:co re}
\vartheta_{r}(x)=(Q\otimes\mathrm{id})\Delta_{r}(x)+\Delta_{r}(Px),\; \theta_{r}(x)=-\Delta_{r}(Px),\;\forall x\in A.
\end{equation}
Then $(A,\triangleright_{A},\triangleleft_{A},\vartheta_{r},\theta_{r})$ is a quasi-triangular \sappb.}
\end{pro}
\begin{proof}
By the assumption, $r+\tau(r)$ is invariant on $(A,\cdot_{A})$ and $r$ satisfies the AAYBE.
Let $r=\sum\limits_{i}u_{i}\otimes v_{i}\in A\otimes A$ and $x\in A$. Then we have
	{\small
		\begin{eqnarray*}	f(x)\big(r+\tau(r)\big)&=&\big(\mathrm{id}\otimes\mathcal{R}_{\circ_{A}}(x)+\mathcal{L}_{\triangleleft_{A}}(x)\otimes\mathrm{id}\big)\big(r+\tau(r)\big)\\
			&=&\sum_{i} u_{i}\otimes v_{i}\circ_{A}x+x\triangleleft_{A}u_{i}\otimes v_{i}+v_{i}\otimes u_{i}\circ_{A}x+x\triangleleft_{A}v_{i}\otimes u_{i}\\
			&=&\sum_{i}u_{i}\otimes P(v_{i})\cdot_{A}x-Q(x\cdot_{A}u_{i})\otimes v_{i}+v_{i}\otimes P(u_{i})\cdot_{A}x-Q(x\cdot_{A}v_{i})\otimes u_{i}\\
			&\overset{\eqref{eq:AAYBE1},\eqref{eq:AAYBE2}}{=}&\sum_{i} Q(u_{i})\otimes v_{i}\cdot_{A}x-Q(x\cdot_{A}u_{i})\otimes v_{i}+Q(v_{i})\otimes u_{i}\cdot_{A}x-Q(x\cdot_{A}v_{i})\otimes u_{i}\\
			&=&(Q\otimes\mathrm{id})\big(\mathrm{id}\otimes\mathcal{L}_{\cdot_{A}}(x)-
			\mathcal{L}_{\cdot_{A}}(x)\otimes\mathrm{id}\big)\big(r+\tau(r)\big)\\
			&=&0,\\
			g(x)\big(r+\tau(r)\big)&=&\big(\mathcal{L}_{\circ_{A}}(x)\otimes\mathrm{id}-\mathrm{id}\otimes\mathcal{L}_{\circ_{A}}(x)\big)\big(r+\tau(r)\big)\\
			&=&\sum_{i} x\circ_{A}u_{i}\otimes v_{i}-u_{i}\otimes x\circ_{A}v_{i}+x\circ v_{i}\otimes u_{i}-v_{i}\otimes x\circ_{A}u_{i}\\
			&=&\sum_{i}P(x)\cdot_{A}u_{i}\otimes v_{i}-u_{i}\otimes P(x)\cdot_{A}v_{i}+P(x)\cdot_{A}v_{i}\otimes u_{i}-v_{i}\otimes P(x)\cdot_{A}u_{i}\\
			&=&-\big(\mathrm{id}\otimes\mathcal{L}_{\cdot_{A}}(Px)-\mathcal{L}_{\cdot_{A}}(Px)\otimes\mathrm{id}\big)\big(r+\tau(r)\big)\\
			&=&0.
	\end{eqnarray*}}Moreover, we have
	\begin{eqnarray*}
		SA(r)&=&\sum_{i,j} u_{i}\circ_{A}u_{j}\otimes v_{i}\otimes v_{j}+u_{i}\otimes v_{i}\triangleleft_{A}u_{j}\otimes v_{j}+u_{i}\otimes u_{j}\otimes v_{j}\circ_{A}v_{i}\\
		&=&\sum_{i,j} P(u_{i})\cdot_{A}u_{j}\otimes v_{i}\otimes v_{j}-u_{i}\otimes Q(v_{i}\cdot_{A}u_{j})\otimes v_{j}+u_{i}\otimes u_{j}\otimes P(v_{j})\cdot_{A}v_{i}\\
		&\overset{\eqref{eq:AAYBE1},\eqref{eq:AAYBE2}}{=}&\sum_{i,j} u_{i}\cdot_{A}u_{j}\otimes Q(v_{i})\otimes v_{j}-u_{i}\otimes Q(v_{i}\cdot_{A}u_{j})\otimes v_{j}+u_{i}\otimes Q(u_{j})\otimes v_{j}\cdot_{A}v_{i}\\
		&=&(\mathrm{id}\otimes Q\otimes\mathrm{id}){\bf A}(r)\\
		&=&0.
	\end{eqnarray*}
	Furthermore, we have
	\begin{eqnarray*}
		\theta(x)&=&-\Delta_{r}(Px)\overset{\eqref{eq:comul1}}{=}-\big( \mathrm{id}\otimes\mathcal{L}_{\cdot_{A}}(Px)-\mathcal{L}_{\cdot_{A}}(Px)\otimes\mathrm{id} \big)r\\
		&=&\sum_{i} P(x)\cdot_{A}u_{i}\otimes v_{i}-u_{i}\otimes P(x)\cdot_{A}v_{i}
		=\sum_{i} x\circ_{A} u_{i}\otimes v_{i}-u_{i}\otimes x\circ_{A}v_{i}
		\\
		&=&\big( \mathcal{L}_{\circ_{A}}(x)\otimes\mathrm{id}-\mathrm{id}\otimes\mathcal{L}_{\circ_{A}}(x) \big)r=g(x)r=\theta_{r}(x),
	\end{eqnarray*}
	and similarly $\vartheta(x)=(f-g)(x)r=\vartheta_{r}(x)$.
	In conclusion, conditions \eqref{2492}-\eqref{2494} hold, and thus $(A,\triangleright_{A},\triangleleft_{A},\vartheta_{r},\theta_{r})$ is a quasi-triangular \sappb.
\end{proof}

\begin{lem}\label{lem:TR}
	Let $(A,\triangleright_{A},\triangleleft_{A})$ be a \sapp and
$(A,\circ_{A})$ be the sub-adjacent perm algebra of $(A,\triangleright_{A},\triangleleft_{A})$. Suppose that
$r=\sum\limits_{i} u_{i}\otimes v_{i}\in A\otimes A$.
Let $\eta_{r},\vartheta_{r},\theta_{r}:A\rightarrow A\otimes A$ be linear maps given by \eqref{eq:comul}, and $\circ_{r},\triangleright_{r},\triangleleft_{r}:A^{*}\otimes A^{*}\rightarrow A^{*}$ be the linear duals of $\eta_{r}$, $\vartheta_{r}$ and $\theta_{r}$ respectively.
Then we have
	\begin{eqnarray}
		&& a^{*}\circ_{r}b^{*}=\mathcal{L}^{*}_{\circ_{A}}\big(r^{\sharp}(a^{*})\big)b^{*}+\mathcal{L}^{*}_{\triangleleft_{A}}\big({\tau(r)}^{\sharp}(b^{*})\big)a^{*},\label{eq:dual space mul1}\\
		&& a^{*}\triangleright_{r}b^{*}=(\mathcal{L}^{*}_{\circ_{A}}+\mathcal{R}^{*}_{\circ_{A}})\big(r^{\sharp}(a^{*})\big)b^{*}-\mathcal{R}^{*}_{\triangleright_{A}}\big({\tau(r)}^{\sharp}(b^{*})\big)a^{*},\label{eq:dual space mul3}\\
		&& a^{*}\triangleleft_{r}b^{*}=\mathcal{R}^{*}_{\circ_{A}}\big({\tau(r)}^{\sharp}(b^{*})\big)a^{*}-\mathcal{R}^{*}_{\circ_{A}}\big(r^{\sharp}(a^{*})\big)b^{*},\;\forall a^{*},b^{*}\in A^{*}.\label{eq:dual space mul2}
	\end{eqnarray}
	Moreover, we have
	{\small
		\begin{align}
			&\langle r^{\sharp}(a^{*})\circ_{A}r^{\sharp}(b^{*})-r^{\sharp}(a^{*}\circ_{r}b^{*}),c^{*}\rangle\notag\\
&=\langle a^{*}\otimes b^{*}\otimes c^{*},(\tau\otimes\mathrm{id})SA(r)- \sum_{j}f(u_{j})\big(r+\tau(r)\big)\otimes v_{j}  \rangle,\label{eq:t1}\\
			&\langle r^{\sharp}(a^{*})\triangleright_{A}r^{\sharp}(b^{*})-r^{\sharp}(a^{*}\triangleright_{r}b^{*}),c^{*}\rangle\notag\\
&=\langle a^{*}\otimes b^{*}\otimes c^{*},(\tau\otimes\mathrm{id}+\mathrm{id}\otimes\tau)SA(r)\notag\\
&\ \ 
-\sum_{j}u_{j}\otimes\tau\Big(f(v_{j})\big(r+\tau(r)\big)\Big)+f(u_{j})\big(r+\tau(r)\big)\otimes v_{j}\rangle,\label{eq:t1.5}\\
			&\langle r^{\sharp}(a^{*})\triangleleft_{A} r^{\sharp}(b^{*})-r^{\sharp}(a^{*}\triangleleft_{r}b^{*}),c^{*}\rangle\notag\\
			&=\langle a^{*}\otimes b^{*}\otimes c^{*},\sum_{j}u_{j}\otimes\tau\Big(f(v_{j})\big(r+\tau(r)\big)\Big)-(\mathrm{id}\otimes\tau)SA(r)\rangle,\label{eq:t2}\\
			&\langle \big(-\tau(r)\big)^{\sharp}(a^{*})\circ_{A}\big(-\tau(r)\big)^{\sharp}(b^{*})-\big(-\tau(r)\big)^{\sharp}(a^{*}\circ_{r}b^{*}),c^{*}\rangle\notag\\
			&=\langle a^{*}\otimes b^{*}\otimes c^{*},\xi SA(r) \rangle, \label{eq:t3}\\
			&\langle \big(-\tau(r)\big)^{\sharp}(a^{*})\triangleright_{A}\big(-\tau(r)\big)^{\sharp}(b^{*})-\big(-\tau(r)\big)^{\sharp}(a^{*}\triangleright_{r}b^{*}),c^{*}\rangle\notag\\
			&=\langle a^{*}\otimes b^{*}\otimes c^{*},(\xi+\mathrm{id}\otimes\tau)SA(r)-\xi\sum_{j} f(u_{j})\big( r+\tau(r) \big)\otimes v_{j}  \rangle,\label{eq:t4.5}\\
			&\langle \big(-\tau(r)\big)^{\sharp}(a^{*})\triangleleft_{A}\big(-\tau(r)\big)^{\sharp}(b^{*})-\big(-\tau(r)\big)^{\sharp}(a^{*}\triangleleft_{r}b^{*}),c^{*}\rangle\notag\\
			&=\langle a^{*}\otimes b^{*}\otimes c^{*},\xi\sum_{j} f(u_{j})\big( r+\tau(r) \big)\otimes v_{j}-(\mathrm{id}\otimes\tau)SA(r)  \rangle,\;\forall a^{*},b^{*},c^{*}\in A^{*}.\label{eq:t4}
	\end{align}}
\end{lem}

\begin{proof}
	Let $x\in A,a^{* },b^{* },c^{* }\in A^{* }$. We have
	\begin{eqnarray*}
		\langle a^{* }\circ _{r}b^{* },x\rangle  &=&\langle \eta_{r} (x),a^{*}\otimes b^{* }\rangle  \\
		&=&\langle \big(\mathrm{id}\otimes \mathcal{R}_{\circ _{A}}(x)+\mathcal{L}_{\triangleleft_{A}}(x)\otimes \mathrm{id}\big)r,a^{* }\otimes b^{* }\rangle
		\\
		&=&\langle r,a^{* }\otimes \mathcal{R}_{\circ _{A}}^{* }(x)b^{* }+
		\mathcal{L}_{\triangleleft_{A}}^{* }(x)a^{* }\otimes b^{* }\rangle  \\
		&=&\langle r^{\sharp}(a^{* }),\mathcal{R}_{\circ _{A}}^{* }(x)b^{*
		}\rangle +\langle {\tau(r)}^{\sharp}(b^{* }),\mathcal{L}_{\triangleleft_{A}}^{*
		}(x)a^{*}\rangle  \\
		&=&\langle r^{\sharp}(a^{* })\circ _{A}x,b^{* }\rangle +\langle x\triangleleft_{A}{\tau(r)}^{\sharp}(b^{* }),a^{* }\rangle  \\
		&=&\langle x,\mathcal{L}_{\circ _{A}}^{* }\big(r^{\sharp}(a^{* })\big)
		b^{* }+\mathcal{L}_{\triangleleft_{A}}^{* }\big({\tau(r)}^{\sharp}(b^{* })\big)
		a^{* }\rangle .
	\end{eqnarray*}
	Hence \eqref{eq:dual space mul1} holds. Similarly we get
	\eqref{eq:dual space mul2}, and \eqref{eq:dual space mul3} holds by substracting \eqref{eq:dual space mul2} from \eqref{eq:dual space mul1}. Moreover, we have
		\begin{eqnarray*}
			&&\langle r^{\sharp}(a^{* })\circ _{A}r^{\sharp}(b^{* }),c^{* }\rangle
			=\langle r^{\sharp}(a^{* }),\mathcal{R}_{\circ _{A}}^{* }\big(r^{\sharp}(b^{*
			})\big)c^{* }\rangle  \\
			&=&\langle r,a^{* }\otimes \mathcal{R}_{\circ _{A}}^{* }\big(
			r^{\sharp}(b^{* })\big)c^{* }\rangle =\sum_{i}\langle u_{i},a^{*
			}\rangle \langle v_{i}\circ _{A}r^{\sharp}(b^{* }),c^{* }\rangle  \\
			&=&\sum_{i}\langle u_{i},a^{* }\rangle \langle r^{\sharp}(b^{* }),\mathcal{L
			}_{\circ _{A}}^{* }(v_{i})c^{* }\rangle =\sum_{i}\langle u_{i},a^{*
			}\rangle \langle r,b^{* }\otimes \mathcal{L}_{\circ _{A}}^{*
			}(v_{i})c^{* }\rangle  \\
			&=&\sum_{i,j}\langle u_{i},a^{* }\rangle \langle u_{j}\otimes v_{i}\circ
			_{A}v_{j},b^{* }\otimes c^{* }\rangle =\sum_{i,j}\langle a^{*
			}\otimes b^{* }\otimes c^{* },u_{i}\otimes u_{j}\otimes v_{i}\circ
			_{A}v_{j}\rangle , \\
			&&\langle r^{\sharp}(a^{* }\circ _{r}b^{* }),c^{* }\rangle =\langle
			r,a^{* }\circ _{r}b^{* }\otimes c^{* }\rangle  \\
			&=&\sum_{i}\langle u_{i},a^{* }\circ _{r}b^{* }\rangle \langle
			v_{i},c^{* }\rangle =\sum_{i}\langle \eta_{r} (u_{i}),a^{* }\otimes
			b^{* }\rangle \langle v_{i},c^{* }\rangle  \\
			&=&\sum_{i,j}\langle u_{j}\otimes v_{j}\circ _{A}u_{i}+u_{i}\triangleleft_{A}u_{j}\otimes v_{j},a^{* }\otimes b^{* }\rangle \langle
			v_{i},c^{* }\rangle  \\
			&=&\sum_{i,j}\langle a^{* }\otimes b^{* }\otimes c^{*
			},u_{j}\otimes v_{j}\circ _{A}u_{i}\otimes v_{i}+u_{i}\triangleleft_{A}u_{j}\otimes
			v_{j}\otimes v_{i}\rangle .
		\end{eqnarray*}
Hence
	\begin{eqnarray*}
		&&\langle r^{\sharp}(a^{* })\circ _{A}r^{\sharp}(b^{* })-r^{\sharp}(a^{* }\circ
		_{r}b^{* }),c^{* }\rangle  \\
		&=&\sum_{i,j}\langle a^{* }\otimes b^{* }\otimes c^{*
		},u_{i}\otimes u_{j}\otimes v_{i}\circ _{A}v_{j}-u_{j}\otimes v_{j}\circ
		_{A}u_{i}\otimes v_{i}-u_{i}\triangleleft_{A}u_{j}\otimes v_{j}\otimes v_{i}\rangle
		\\
		&=&\langle a^{* }\otimes b^{* }\otimes c^{* },(\tau \otimes \mathrm{
			id})SA(r)-\sum_{j}f(u_{j})\big(r+\tau (r)\big)\otimes v_{j}\rangle,
	\end{eqnarray*}
	that is, \eqref{eq:t1} holds. Similarly we get \eqref{eq:t1.5}-\eqref{eq:t4}.
\end{proof}

\begin{thm}\label{thm:4.6}
	Let $(A,\triangleright_{A},\triangleleft_{A},\vartheta_{r},\theta_{r})$ be a quasi-triangular \sappb, and
	$\triangleright_{r},\triangleleft_{r}:A^{*}\otimes A^{*}\rightarrow A^{*}$ be the linear duals of $\vartheta_{r}$ and $\theta_{r}$ respectively.
	Then $r^{\sharp}:A^{*}\rightarrow A$ is a \sapp homomorphism, that is,
	\begin{eqnarray} &&r^{\sharp}(a^{*})\triangleright_{A}r^{\sharp}(b^{*})=r^{\sharp}(a^{*}\triangleright_{r} b^{*}),\;\label{eq:homo sdpp1}\\
&& r^{\sharp}(a^{*})\triangleleft_{A}r^{\sharp}(b^{*})=r^{\sharp}(a^{*}\triangleleft_{r} b^{*}),\;\forall a^{*},b^{*}\in A^{*}.\label{eq:homo sdpp2}
	\end{eqnarray}
Moreover, $\big(-\tau(r)\big)^{\sharp}:A^{*}\rightarrow A$ is also a \sapp homomorphism.
\end{thm}

\begin{proof}
By  Lemma \ref{lem:TR}, \eqref{eq:homo sdpp1} and \eqref{eq:homo sdpp2} hold. Thus  $r^{\sharp}$ is a \sapp homomorphism. Similarly $\big(-\tau(r)\big)^{\sharp}$ is also a \sapp homomorphism.
\end{proof}

Now we study the representation theory of \sapps.

\begin{defi}
Let $(A,\triangleright_{A},\triangleleft_{A})$ be a \sapp and $(A,\circ_{A})$ be
the sub-adjacent perm algebra. Let $V$ be a vector space and
$l_{\triangleright_{A}},r_{\triangleright_{A}},l_{\triangleleft_{A}}
:A\rightarrow \mathrm{End}_{\mathbb K}(V)$ be linear maps. Set
\begin{equation}\label{eq:sum linear}
l_{\circ_{A}}=l_{\triangleright_{A}}+l_{\triangleleft_{A}},\;\;
r_{\circ_{A}}=r_{\triangleright_{A}}+l_{\triangleleft_{A}}.
\end{equation}
If $(l_{\circ_{A}},r_{\circ_{A}},V)$ is a representation of $(A,\circ_{A})$, that is, 
\begin{eqnarray}
&&l_{\circ_{A}}(x\circ_{A}y)v
=l_{\circ_{A}}(x)l_{\circ_{A}}(y)v
=l_{\circ_{A}}(y)l_{\circ_{A}}(x)v,\label{eq:rep1}\\
&&r_{\circ_{A}}(x\circ_{A}y)v
=r_{\circ_{A}}(y)r_{\circ_{A}}(x)v
=r_{\circ_{A}}(y)l_{\circ_{A}}(x)v
=l_{\circ_{A}}(x)r_{\circ_{A}}(y)v,\;\forall x,y\in A,v\in V,\ \ \ \ \label{eq:rep2}
\end{eqnarray}
and the following equations hold:
\begin{eqnarray}
&&l_{\triangleleft_{A}}(x\circ_{A}y)v=l_{\circ_{A}}(x)
l_{\triangleleft_{A}}(y)v
=-l_{\triangleleft_{A}}(x)l_{\triangleleft_{A}}(y)v
=l_{\triangleleft_{A}}(y)l_{\circ_{A}}(x)v,\label{eq:sdpp rep1}\\
&&l_{\triangleleft_{A}}(x\triangleleft_{A}y)v
=-l_{\triangleleft_{A}}(y)r_{\circ_{A}}(x)v
=-r_{\circ_{A}}(x\triangleleft_{A}y)v,\;\forall x,y\in A, v\in V,
\end{eqnarray}
then we say $(l_{\triangleright_{A}},r_{\triangleright_{A}},
l_{\triangleleft_{A}},V)$ is a {\bf representation} of $(A,\triangleright_{A},\triangleleft_{A})$.
\end{defi}

\delete{
Two representations $(l _{\triangleright_{A}},r _{\triangleright_{A}},l _{\triangleleft_{A}},V)$ and $(l' _{\triangleright_{A}},r' _{\triangleright_{A}},l' _{\triangleleft_{A}},V')$ of $(A,\triangleright_{A},\triangleleft_{A})$ are called
\textbf{equivalent} if there exists a linear isomorphism $\phi:V \rightarrow V'$ such that the following equations hold:
\begin{eqnarray*}
\phi l_{\triangleright_{A}}(x)=l'_{\triangleright_{A}}(x)\phi,\;\phi r_{\triangleright_{A}}(x)=r'_{\triangleright_{A}}(x)\phi,\;\phi l_{\triangleleft_{A}} (x)=l'_{\triangleleft_{A}}(x)\phi, \;\forall x\in A.
\end{eqnarray*}
\textcolor{blue}{GL: Are equivalent reps used in this paper? If not, please delete this notion.}}

\begin{pro}
Let $(A,\triangleright_{A},\triangleleft_{A})$ be a \sapp, $V$ be a vector space and
$l_{\triangleright_{A}},r_{\triangleright_{A}},l_{\triangleleft_{A}}
:A\rightarrow \mathrm{End}_{\mathbb K}(V)$ be linear maps.
Then $(l_{\triangleright_{A}},r_{\triangleright_{A}},l_{\triangleleft_{A}},
V)$ is a representation of the \sapp $(A,\triangleright_{A},\triangleleft_{A})$ if and only
if there is a \sapp structure on $A\oplus V$ given by
\begin{eqnarray}
	&&(x+u)\triangleright_{d} (y+v)=x\triangleright_{A}y+l_{\triangleright_{A}}(x)v+r_{\triangleright_{A}}(y)u,\label{eq:sd SDPP1}\\
	&&(x+u)\triangleleft_{d} (y+v)=x\triangleleft_{A}y+l_{\triangleleft_{A}}(x)v+l_{\triangleleft_{A}}(y)u,\;\forall x,y\in A,u,v\in V.
	\label{eq:sd SDPP2}
\end{eqnarray}
In this case, we denote the \sapp structure on $A\oplus V$ by $A\ltimes
_{l_{\triangleright_{A}},r_{\triangleright_{A}},l_{\triangleleft_{A}}}V$ and call it the {\bf semi-direct product \sapp of $(A,\triangleright_{A},\triangleleft_{A})$ with respect to $(l_{\triangleright_{A}},r_{\triangleright_{A}},l_{\triangleleft_{A}},
	V)$}.
\end{pro}
\begin{proof}
	 It follows from Proposition \ref{2517} by taking the zero multiplications on $V$.
\end{proof}

\begin{ex}
Let $(A,\triangleright_{A},\triangleleft_{A})$ be a \sapp. Then
$(\mathcal{L}_{\triangleright_{A}},\mathcal{R}_{\triangleright_{A}},
\mathcal{L}_{\triangleleft_{A}},A)$
is a representation of $(A,\triangleright_{A},\triangleleft_{A} )$, which is called
the \textbf{adjoint representation}.
\end{ex}

\begin{pro}
Let $(A,\triangleright_{A},\triangleleft_{A})$ be a \sapp and $(A,\circ_{A})$ be the sub-adjacent perm algebra. If
$(l_{\triangleright_{A}},r_{\triangleright_{A}},l_{\triangleleft_{A}},V)$ is a representation of $(A,\triangleright_{A},\triangleleft_{A})$, then
$$(l^{*}_{\circ_{A}}+r^{*}_{\circ_{A}},r^{*}_{\triangleright_{A}},
-l^{*}_{\triangleleft_{A}}-r^{*}_{\triangleright_{A}},V^{*})
=(l^{*}_{\circ_{A}}+r^{*}_{\circ_{A}},r^{*}_{\triangleright_{A}},
-r^{*}_{\circ_{A}},V^{*})$$
is also a representation of $(A,\triangleright_{A},\triangleleft_{A})$.
In particular, $(\mathcal{L}^{*}_{\circ_{A}}+\mathcal{R}^{*}_{\circ_{A}},
\mathcal{R}^{*}_{\triangleright_{A}},-\mathcal{R}^{*}_{\circ_{A}},
A^{*})$ is a representation of $(A,\triangleright_{A},\triangleleft_{A})$, which is called the {\bf coadjoint representation} of $(A,\triangleright_{A},\triangleleft_{A})$.
\end{pro}

\begin{proof}
It is clear that $l^{*}_{\circ_{A}}$ satisfies \eqref{eq:rep1}.
For all $x,y\in A,u^{*}\in V^{*},v\in V$, we have
{\small
\begin{eqnarray*}
&&\langle -l^{*}_{\triangleleft_{A}}(x\circ _{A}y)u^{*},v\rangle =\langle u^{*
},-l _{\triangleleft_{A}}(x\circ _{A}y)v\rangle ,\;\langle l_{\triangleleft_{A}}^{* }(y)l_{\triangleleft_{A}}^{* }(x)u^{*
},v\rangle =\langle u^{* },l_{\triangleleft_{A}}(x)l_{\triangleleft_{A}}(y)v\rangle , \\
&&\langle -l_{\triangleleft_{A}}^{* }(y)l_{\circ_{A}}^{* }(x)u^{* },v\rangle =\langle u^{*
},-l _{\circ_{A}}(x)l_{\triangleleft_{A}}(y)v\rangle ,\;\langle -l_{\circ_{A}}^{* }(x)l_{\triangleleft_{A}}^{* }(y)u^{* },v\rangle
=\langle u^{* },-l_{\triangleleft_{A}}(y)l_{\circ_{A}}(x)v\rangle.
\end{eqnarray*}}Hence by \eqref{eq:sdpp rep1}, we have
\begin{equation*}
-l_{\triangleleft_{A}}^{* }(x\circ _{A}y)u^{* }=l_{\triangleleft_{A}}^{* }(y)l_{\triangleleft_{A}}^{* }(x)u^{* }=-l_{\triangleleft_{A}}^{*
}(y)l_{\circ_{A}}^{* }(x)u^{* }=-l_{\circ_{A}}^{* }(x)l_{\triangleleft_{A}}^{* }(y)u^{*}.
\end{equation*}
Thus 
$$ (l^{*}_{\circ_{A}},-l^{*}_{\triangleleft_{A}},V^{*})
=(l^{*}_{\circ_{A}}+r^{*}_{\circ_{A}}
-l^{*}_{\triangleleft_{A}}-r^{*}_{\triangleright_{A}},
r^{*}_{\triangleright_{A}}-l^{*}_{\triangleleft_{A}}
-r^{*}_{\triangleright_{A}},V^{*})$$
is a representation of $(A,\circ_{A})$.
Similarly we have
\begin{eqnarray*}
&&-r_{\circ_{A}}^{* }(x\circ _{A}y)u^{* }=-l_{\circ_{A}}^{* }(x)r_{\circ_{A}}^{* }(y)u^{* }=-r_{\circ_{A}}^{*
}(x)r_{\circ_{A}}^{* }(y)u^{* }=-r_{\circ_{A}}^{* }(y)l_{\circ_{A}}^{* }(x)u^{* }, \\
&&-r_{\circ_{A}}^{* }(x\triangleleft_{A}y)u^{* }=-r_{\circ_{A}}^{* }(y)l_{\triangleleft_{A}}^{* }(x)u^{* }=l_{\triangleleft_{A}}^{*
}(x\triangleleft_{A}y)u^{* }.
\end{eqnarray*}
Thus $(l^{*}_{\circ_{A}}+r^{*}_{\circ_{A}},r^{*}_{\triangleright_{A}},
-r^{*}_{\circ_{A}},V^{*})$ is a representation of $(A,\triangleright_{A},\triangleleft_{A})$.
\end{proof}

Now we introduce the notion of special apre-perm representation algebras.

\begin{defi}
Let $(A,\triangleright_{A},\triangleleft_{A})$ and $(V,\triangleright_{V},\triangleleft_{V})$ be \sapps and $l_{\triangleright_{A}},r_{\triangleright_{A}},$
$l_{\triangleleft_{A}}:
A\rightarrow \mathrm{End}_{\mathbb K}(V)$ be linear maps.
Let $l_{\circ_{A}},r_{\circ_{A}}:A\rightarrow \mathrm{End}_{\mathbb K}(V)$ be linear maps given by \eqref{eq:sum linear}.
If $(l_{\triangleright_{A}},r_{\triangleright_{A}},l_{\triangleleft_{A}},V)$ is a representation of $(A,\triangleright_{A},\triangleleft_{A})$ and the following equations hold:
	\begin{eqnarray}		&&l_{\circ_{A}}(x)(u\circ_{V}v)=l_{\circ_{A}}(x)u\circ_{V}v=r_{\circ_{A}}(x)u\circ_{V}v=u\circ_{V}l_{\circ_{A}}(x)v,\label{bim alg1}\\
		&&l_{\triangleleft_{A}}(x)(u\circ_{V}v)=u\circ_{V}l_{\triangleleft_{A}}(x)v=-u\triangleleft_{V}l_{\triangleleft_{A}}(x)v=r_{\circ_{A}}(x)u\triangleleft_{V}v,\label{bim alg2}\\
		&&r_{\circ_{A}}(x)(u\circ_{V}v)=r_{\circ_{A}}(x)(v\circ_{V}u)=u\circ_{V}r_{\circ_{A}}(x)v,\label{bim alg3}\\
		&&l_{\circ_{A}}(x)(u\triangleleft_{V}v)=l_{\circ_{A}}(x)u\triangleleft_{V}v=-l_{\triangleleft_{A}}(x)(u\triangleleft_{V}v),\;\forall x\in A, u,v\in V,\label{bim alg4}
	\end{eqnarray}
 then we say $(l_{\triangleright_{A}},r_{\triangleright_{A}},l_{\triangleleft_{A}},V,\triangleright_{V},\triangleleft_{V})$ is an \textbf{$(A,\triangleright_{A},\triangleleft_{A})$-representation algebra}.
\end{defi}

\begin{pro}\label{2517}
	Let $(A,\triangleright_{A},\triangleleft_{A})$ be a \sapp.
	Suppose that $V$ is a vector space with multiplications $\triangleright_{V},\triangleleft_{V}:V\otimes V\rightarrow V$ and
	$l_{\triangleright_{A}},r_{\triangleright_{A}},l_{\triangleleft_{A}}:A\rightarrow \mathrm{End}_{\mathbb K}(V) $
	are linear maps.
	Then $(l_{\triangleright_{A}},r_{\triangleright_{A}},l_{\triangleleft_{A}},V,\triangleright_{V},\triangleleft_{V})$ is an $(A,\triangleright_{A},\triangleleft_{A})$-representation algebra if and only if there is a \sapp structure on $A\oplus V$ given by
	\begin{eqnarray}
		&&(x+u)\triangleright_{d} (y+v)=x\triangleright_{A}y+l_{\triangleright_{A}}(x)v+r_{\triangleright_{A}}(y)u+u\triangleright_{V}v,\\
		&&(x+u)\triangleleft_{d} (y+v)=x\triangleleft_{A}y+l_{\triangleleft_{A}}(x)v+l_{\triangleleft_{A}}(y)u+u\triangleleft_{V}v, \;\forall x,y\in A, u,v\in V.\ \ \ \ \
	\end{eqnarray}
\end{pro}
\begin{proof}
	It follows from a straightforward computation.
	\delete{
Set a multiplication $\circ_{d}$ on $A\oplus V$ by
\begin{equation*}
 (x+u)\circ_{d} (y+v)=(x+u)\triangleright_{d} (y+v)+(x+u)\triangleleft_{d} (y+v).
\end{equation*}
Then it follows from a straightforward computation that   $(l_{\triangleright_{A}},r_{\triangleright_{A}},l_{\triangleleft_{A}},V,\triangleright_{V},\triangleleft_{V})$ is an  $(A,\triangleright_{A},$
$\triangleleft_{A})$-representation algebra if and only if the following conditions hold:
\begin{enumerate}
	\item $(A\oplus V,\circ_{d})$ is a perm algebra.
	\item $\triangleleft_{d}$ is commutative.
	\item \eqref{eq:SDPP} holds for the triple $(A\oplus V,\circ_{d},\triangleleft_{d})$.
\end{enumerate}
Hence the conclusion follows.}
\end{proof}

\begin{lem}\label{lem:6.3}
	Let $(A,\triangleright_{A},\triangleleft_{A})$ be a \sapp and $s\in A\otimes A$.
	Then  $s$ is invariant if and only if the following equations hold:
	\begin{eqnarray}
		&&s^{\sharp}(a^{*})\circ_{A}x+s^{\sharp}\big(\mathcal{L}^{*}_{\triangleleft_{A}}(x)a^{*}\big)=0,\label{eq:lem1}\\
		&&s^{\sharp}\big(\mathcal{L}^{*}_{\circ_{A}}(x)a^{*}\big)-x\circ_{A}s^{\sharp}(a^{*})=0,\;\forall x\in A, a^{*}\in A^{*}.\label{eq:lem2}
	\end{eqnarray}
	If in addition $s$  is symmetric, then $s$ is invariant if and only if the following equations hold:
	\begin{eqnarray}
		&&\mathcal{L}^{*}_{\circ_{A}}\big(s^{\sharp}(a^{*})\big)b^{*}+\mathcal{L}^{*}_{\triangleleft_{A}}\big(s^{\sharp}(b^{*})\big)a^{*}=0,\label{eq:lem3}\\
		&&\mathcal{R}^{*}_{\circ_{A}}\big(s^{\sharp}(b^{*})\big)a^{*}-\mathcal{R}^{*}_{\circ_{A}}\big(s^{\sharp}(a^{*})\big)b^{*}=0,\;\forall a^{*}, b^{*}\in A^{*}.\label{eq:lem4}
	\end{eqnarray}
\end{lem}
\begin{proof}
	For all $x\in A, a^{*},b^{*}\in A^{*}$, we have
		\begin{eqnarray*}
			\langle s^{\sharp}(a^{*})\circ_{A}x+s^{\sharp}\big(\mathcal{L}^{*}_{\triangleleft_{A}}(x)a^{*}\big),b^{*}\rangle&=&
			\langle s^{\sharp}(a^{*})\circ_{A}x, b^{*}\rangle+\langle s^{\sharp}\big(\mathcal{L}^{*}_{\triangleleft_{A}}(x)a^{*}\big),b^{*}\rangle\\
			&=&\langle s^{\sharp}(a^{*}),\mathcal{R}^{*}_{\circ_{A}}(x)b^{*}\rangle+\langle s,\mathcal{L}^{*}_{\triangleleft_{A}}(x)a^{*}\otimes b^{*}\rangle\\
			&=&\langle \big(\mathrm{id}\otimes\mathcal{R}_{\circ_{A}}(x)+\mathcal{L}_{\triangleleft_{A}}(x)\otimes\mathrm{id}\big)s, a^{*}\otimes b^{*}\rangle,\\
			\langle s^{\sharp}\big(\mathcal{L}^{*}_{\circ_{A}}(x)a^{*}\big)-x\circ_{A}s^{\sharp}(a^{*}),b^{*}\rangle
			&=&\langle s^{\sharp}\big(\mathcal{L}^{*}_{\circ_{A}}(x)a^{*}\big),b^{*}\rangle
-\langle s^{\sharp}(a^{*}),\mathcal{L}^{*}_{\circ_{A}}(x)b^{*}\rangle\\
			&=&\langle s,\mathcal{L}^{*}_{\circ_{A}}(x)a^{*}\otimes b^{*}\rangle-\langle s,a^{*}\otimes \mathcal{L}^{*}_{\circ_{A}}(x)b^{*}\rangle\\
			&=&\langle \big(\mathcal{L}_{\circ_{A}}(x)\otimes\mathrm{id}-\mathrm{id}\otimes\mathcal{L}_{\circ_{A}}(x)\big)s, a^{*}\otimes b^{*}\rangle.
		\end{eqnarray*}
	Hence $s$ is invariant if and only if \eqref{eq:lem1} and \eqref{eq:lem2} hold. Moreover, if $s$  is symmetric, then we have
		\begin{eqnarray*}
			\langle \mathcal{L}^{*}_{\circ_{A}}\big(s^{\sharp}(a^{*})\big)b^{*}+\mathcal{L}^{*}_{\triangleleft_{A}}\big(s^{\sharp}(b^{*})\big)a^{*},x\rangle
			&=&\langle b^{*},s^{\sharp}(a^{*})\circ_{A}x\rangle+\langle a^{*},s^{\sharp}(b^{*})\triangleleft_{A}x\rangle\\
			&=&\langle s^{\sharp}(a^{*}),\mathcal{R}^{*}_{\circ_{A}}(x)b^{*}\rangle+\langle s^{\sharp}(b^{*}),\mathcal{L}^{*}_{\triangleleft_{A}}(x)a^{*}\rangle\\
			&=&\langle s, a^{*}\otimes\mathcal{R}^{*}_{\circ_{A}}(x)b^{*}+b^{*}\otimes \mathcal{L}^{*}_{\triangleleft_{A}}(x)a^{*} \rangle\\
			&=&\langle \big(\mathrm{id}\otimes\mathcal{R}_{\circ_{A}}(x)+\mathcal{L}_{\triangleleft_{A}}(x)\otimes\mathrm{id}\big)s, a^{*}\otimes b^{*}\rangle,\\
			\langle \mathcal{R}^{*}_{\circ_{A}}\big(s^{\sharp}(b^{*})\big)a^{*}-\mathcal{R}^{*}_{\circ_{A}}\big(s^{\sharp}(a^{*})\big)b^{*},x\rangle
			&=&\langle x\circ_{A}s^{\sharp}(b^{*}),a^{*}\rangle-\langle x\circ_{A}s^{\sharp}(a^{*}),b^{*}\rangle\\
			&=&\langle s^{\sharp}(b^{*}),\mathcal{L}^{*}_{\circ_{A}}(x)a^{*}\rangle-\langle s^{\sharp}(a^{*}),\mathcal{L}^{*}_{\circ_{A}}(x)b^{*}\rangle\\
			&=&\langle s,b^{*}\otimes \mathcal{L}^{*}_{\circ_{A}}(x)a^{*}-a^{*}\otimes\mathcal{L}^{*}_{\circ_{A}}(x)b^{*}\rangle\\
			&=&\langle \big(\mathcal{L}_{\circ_{A}}(x)\otimes\mathrm{id}-\mathrm{id}\otimes\mathcal{L}_{\circ_{A}}(x)\big)s, a^{*}\otimes b^{*}\rangle.
		\end{eqnarray*}
Hence $s$ is invariant if and only if \eqref{eq:lem3} and \eqref{eq:lem4} hold.
\end{proof}

\begin{pro}\label{pro ss inv SDPP}
	Let $(A,\triangleright_{A},\triangleleft_{A})$ be a \sapp and $s\in A\otimes A$ be symmetric and invariant. Set multiplications $ \triangleright_{s}, \triangleleft_{s}:A^{*}\otimes A^{*}\rightarrow A^{*}$ by
	\begin{equation}	a^{*} \triangleright_{s}b^{*}=(\mathcal{L}^{*}_{\circ_{A}}+\mathcal{R}^{*}_{\circ_{A}})\big(s^{\sharp}(a^{*})\big)b^{*},\;
 a^{*} \triangleleft_{s}b^{*}=-\mathcal{R}^{*}_{\circ_{A}}\big(s^{\sharp}(a^{*})\big)b^{*},\;\forall a^{*},b^{*}\in A^{*}.
	\end{equation}
	\delete{
	\begin{eqnarray}
		&&a^{*}\bullet_{s}b^{*}=\mathcal{L}^{*}_{\circ_{A}}\big(s^{\sharp}(a^{*})\big)b^{*},\label{eq1:perm on A^{*}}\\
		&&a^{*}\lhd_{s}b^{*}=\mathcal{R}^{*}_{\circ_{A}}\big(s^{\sharp}(a^{*})\big)b^{*},\;\;\forall a^{*},b^{*}\in A^{*}.\label{eq2:perm on A^{*}}
	\end{eqnarray}}
	Then $(\mathcal{L}^{*}_{\circ_{A}}+\mathcal{R}^{*}_{\circ_{A}},\mathcal{R}^{*}_{\triangleright_{A}},-\mathcal{R}^{*}_{\circ_{A}}, A^{*}, \triangleright_{s}, \triangleleft_{s})$ is an $(A,\triangleright_{A},\triangleleft_{A})$-representation algebra.
\end{pro}

\begin{proof}
	By \eqref{eq:lem4}, $ \triangleleft_{s}$ is commutative.
	Set a multiplication $\bullet_{s}$ on $A^{*}$ by
	\begin{equation*}
		a^{*}\bullet_{s}b^{*}= a^{*} \triangleright_{s}b^{*}+a^{*} \triangleleft_{s}b^{*} =\mathcal{L}^{*}_{\circ_{A}}\big(s^{\sharp}(a^{*})\big)b^{*},\;\forall a^{*},b^{*}\in A^{*}.
	\end{equation*}
	For all $a^{*},b^{*},c^{*}\in A^{*}$, we have
	\begin{eqnarray*}
		&&a^{*}\bullet_{s}( b^{*}\bullet_{s}c^{*})=a^{*}\bullet_{s}\mathcal{L}^{*}_{\circ_{A}}\big(s^{\sharp}(b^{*})\big)c^{*}=\mathcal{L}^{*}_{\circ_{A}}\big(s^{\sharp}(a^{*})\big)\mathcal{L}^{*}_{\circ_{A}}\big(s^{\sharp}(b^{*})\big)c^{*},\\
		&&b^{*}\bullet_{s}( a^{*}\bullet_{s}c^{*})=b^{*}\bullet_{s}\mathcal{L}^{*}_{\circ_{A}}\big(s^{\sharp}(a^{*})\big)c^{*}=\mathcal{L}^{*}_{\circ_{A}}\big(s^{\sharp}(b^{*})\big)\mathcal{L}^{*}_{\circ_{A}}\big(s^{\sharp}(a^{*})\big)c^{*},\\
		&&(a^{*}\bullet_{s}b^{*})\bullet_{s}c^{*}=\mathcal{L}^{*}_{\circ_{A}}\bigg(s^{\sharp}\Big(\mathcal{L}^{*}_{\circ_{A}}\big(s^{\sharp}(a^{*})\big)b^{*}\Big)\bigg)c^{*}\overset{\eqref{eq:lem2}}{=}\mathcal{L}^{*}_{\circ_{A}}\big(s^{\sharp}(a^{*})\circ_{A}s^{\sharp}(b^{*})\big)c^{*},\\
		&&(a^{*}\bullet_{s}b^{*}) \triangleleft_{s}c^{*}=-\mathcal{R}^{*}_{\circ_{A}}\bigg(s^{\sharp}\Big(\mathcal{L}^{*}_{\circ_{A}}\big(s^{\sharp}(a^{*})\big)b^{*}\Big)\bigg)c^{*}\overset{\eqref{eq:lem2}}{=}-\mathcal{R}^{*}_{\circ_{A}}\big(s^{\sharp}(a^{*})\circ_{A}s^{\sharp}(b^{*})\big)c^{*},\\
		&&a^{*}\bullet_{s}( b^{*} \triangleleft_{s}c^{*})=-a^{*}\bullet_{s}\mathcal{R}^{*}_{\circ_{A}}\big(s^{\sharp}(b^{*})\big)c^{*}=-\mathcal{L}^{*}_{\circ_{A}}\big(s^{\sharp}(a^{*})\big)\mathcal{R}^{*}_{\circ_{A}}\big(s^{\sharp}(b^{*})\big)c^{*},\\
		&&a^{*} \triangleleft_{s}( b^{*} \triangleleft_{s}c^{*})=-a^{*} \triangleleft_{s}\mathcal{R}^{*}_{\circ_{A}}\big(s^{\sharp}(b^{*})\big)c^{*}=\mathcal{R}^{*}_{\circ_{A}}\big(s^{\sharp}(a^{*})\big)\mathcal{R}^{*}_{\circ_{A}}\big(s^{\sharp}(b^{*})\big)c^{*}.
	\end{eqnarray*}
Hence by \eqref{eq:rep1} and \eqref{eq:sdpp rep1}, we have
\begin{eqnarray*}
 &&a^{*}\bullet_{s}( b^{*}\bullet_{s}c^{*})=b^{*}\bullet_{s}( a^{*}\bullet_{s}c^{*})=(a^{*}\bullet_{s}b^{*})\bullet_{s}c^{*},\\
 &&(a^{*}\bullet_{s}b^{*}) \triangleleft_{s}c^{*}=a^{*}\bullet_{s}( b^{*} \triangleleft_{s}c^{*})=-a^{*} \triangleleft_{s}( b^{*} \triangleleft_{s}c^{*}).
\end{eqnarray*}
Thus $(A^{*},\bullet_{s})$ is a perm algebra and
 $(A^{*}, \triangleright_{s}, \triangleleft_{s})$ is a compatible \sapp of $(A^{*},\bullet_{s})$. Since $s$ is symmetric and invariant, we have
\begin{eqnarray*}
&&\langle x\triangleleft_{A}s^{\sharp}(a^{*}),b^{*}\rangle=\langle s^{\sharp}(a^{*}),\mathcal{L}^{*}_{\triangleleft_{A}}(x)b^{*}\rangle=\langle s,a^{*}\otimes\mathcal{L}^{*}_{\triangleleft_{A}}(x)b^{*}\rangle\\	&&=\langle\big(\mathrm{id}\otimes\mathcal{L}_{\triangleleft_{A}}(x)\big)s, a^{*}\otimes b^{*}\rangle =-\langle \big(\mathcal{R}_{\circ_{A}}(x)\otimes\mathrm{id}\big)s,a^{*}\otimes b^{*}\rangle
=-\langle s^{\sharp}\big(\mathcal{R}^{*}_{\circ_{A}}(x)a^{*}\big),b^{*}\rangle,
\end{eqnarray*}
that is,
	\begin{equation}\label{eq:lem5} x\triangleleft_{A}s^{\sharp}(a^{*})=-s^{\sharp}\big(\mathcal{R}^{*}_{\circ_{A}}(x)a^{*}\big),\;\forall x\in A, a^{*}\in A^{*}.
	\end{equation}
	Moreover, for all $x,y\in A, a^{*},b^{*}\in A^{*}$, we have
	\begin{eqnarray*} \langle\mathcal{L}^{*}_{\triangleleft_{A}}(x)(a^{*}\bullet_{s}b^{*}),y\rangle&=&\langle \mathcal{L}^{*}_{\circ_{A}}\big(s^{\sharp}(a^{*})\big)b^{*},x\triangleleft_{A}y\rangle=\langle b^{*},s^{\sharp}(a^{*})\circ_{A}(x\triangleleft_{A}y)\rangle,\\
		\langle \mathcal{L}^{*}_{\triangleleft_{A}}(x)(b^{*}\bullet_{s}a^{*}),y\rangle&=&\langle\mathcal{L}^{*}_{\circ_{A}}\big(s^{\sharp}(b^{*})\big)a^{*},x\triangleleft_{A}y\rangle=\langle s^{\sharp}(b^{*}),\mathcal{R}^{*}_{\circ_{A}}(x\triangleleft_{A}y)a^{*}\rangle\\
		&=&\langle s^{\sharp}\big(\mathcal{R}^{*}_{\circ_{A}}(x\triangleleft_{A}y)a^{*}\big),b^{*}\rangle\overset{\eqref{eq:lem5}}{=}-\langle b^{*},(x\triangleleft_{A}y)\triangleleft_{A}s^{\sharp}(a^{*})\rangle,\\
		\langle a^{*}\bullet_{s}\mathcal{L}^{*}_{\triangleleft_{A}}(x)(b^{*}),y\rangle&=&\langle \mathcal{L}^{*}_{\circ_{A}}\big(s^{\sharp}(a^{*})\big)\mathcal{L}^{*}_{\triangleleft_{A}}(x)(b^{*}),y\rangle
		=\langle b^{*},x\triangleleft_{A}(s^{\sharp}(a^{*})\circ_{A}y)\rangle.
	\end{eqnarray*}
	Hence by \eqref{eq:SDPP}, \eqref{bim alg3} holds for $(\mathcal{L}^{*}_{\circ_{A}}+\mathcal{R}^{*}_{\circ_{A}},\mathcal{R}^{*}_{\triangleright_{A}},-\mathcal{R}^{*}_{\circ_{A}}, A^{*}, \triangleright_{s}, \triangleleft_{s})$. Similarly \eqref{bim alg1}, \eqref{bim alg2} and \eqref{bim alg4} hold for $(\mathcal{L}^{*}_{\circ_{A}}+\mathcal{R}^{*}_{\circ_{A}},\mathcal{R}^{*}_{\triangleright_{A}},-\mathcal{R}^{*}_{\circ_{A}}, A^{*}, \triangleright_{s}, \triangleleft_{s})$,
	and thus $(\mathcal{L}^{*}_{\circ_{A}}+\mathcal{R}^{*}_{\circ_{A}},\mathcal{R}^{*}_{\triangleright_{A}},-\mathcal{R}^{*}_{\circ_{A}}, A^{*}, \triangleright_{s}, \triangleleft_{s})$ is an $(A,\triangleright_{A},\triangleleft_{A})$-representation algebra.
\end{proof}

Next we introduce the notion of $\calo$-operators with weights of \sapps.

\begin{defi}
	Let $(A,\triangleright_{A},\triangleleft_{A})$ be a \sapp and $(l_{\triangleright_{A}},r_{\triangleright_{A}},l_{\triangleleft_{A}},V,\triangleright_{V},\triangleleft_{V})$ be an $(A,\triangleright_{A},\triangleleft_{A})$-representation algebra. A linear map $T:V\rightarrow A$ is called an {\bf $\calo$-operator of weight $\lambda\in\mathbb{K}$  of $(A,\triangleright_{A},\triangleleft_{A})$ associated to $(l_{\triangleright_{A}},r_{\triangleright_{A}},
l_{\triangleleft_{A}},V,\triangleright_{V},\triangleleft_{V})$} if $T$ satisfies
	\begin{eqnarray} &&Tu\triangleright_{A}Tv=T\big(l_{\triangleright_{A}}(Tu)v+r_{\triangleright_{A}}(Tv)u+\lambda u\triangleright_{V}v\big),\\ &&Tu\triangleleft_{A}Tv=T\big(l_{\triangleleft_{A}}(Tu)v+l_{\triangleleft_{A}}(Tv)u+\lambda u\triangleleft_{V}v\big),\;\;\forall u,v\in V.
	\end{eqnarray}
	In particular, if $V$ is equipped with zero multiplications, that is
	\begin{equation*}
		u\triangleright_{V} v=u\triangleleft_{V} v=0,
	\end{equation*}
	then we simply say $T:V\rightarrow A$ is an {\bf $\calo$-operator of $(A,\triangleright_{A},\triangleleft_{A})$ associated to $(l_{\triangleright_{A}},r_{\triangleright_{A}},l_{\triangleleft_{A}},$
$V)$}.
\end{defi}

\begin{ex}
	Let $(A,\triangleright_{A},\triangleleft_{A})$ be a \sapp. Suppose that  $R:A\rightarrow A$ is a \textbf{Rota-Baxter operator of weight $\lambda$}, that is,
	\begin{eqnarray}
		&&R(x)\triangleright_{A}R(y)=R\big(x\triangleright_{A}R(y)+R(x)\triangleright_{A}y+\lambda x\triangleright_{A}y\big),\\
		&&R(x)\triangleleft_{A}R(y)=R\big(x\triangleleft_{A}R(y)+R(x)\triangleleft_{A}y+\lambda x\triangleleft_{A}y\big),\;\;\forall x, y\in A.
	\end{eqnarray}
	Then $R$ is an $\calo$-operator of $(A,\triangleright_{A},\triangleleft_{A})$ of weight $\lambda$ associated to $(\mathcal{L}_{\triangleright_{A}},\mathcal{R}_{\triangleright_{A}},\mathcal{L}_{\triangleleft_{A}},A,\triangleright_{A},
\triangleleft_{A})$. In particular, if $\lambda=0$, then $R$ is simply an $\calo$-operator of $(A,\triangleright_{A},\triangleleft_{A})$ associated to $(\mathcal{L}_{\triangleright_{A}},\mathcal{R}_{\triangleright_{A}},$
$\mathcal{L}_{\triangleleft_{A}},A)$.
\end{ex}

Solutions of the SAPP-YBE whose symmetric parts are invariant can be interpreted in terms of $\mathcal{O}$-operators with weights as follows.

\begin{thm}\label{pro:rrb}
	Let $(A,\triangleright_{A},\triangleleft_{A})$ be a \sapp and $r\in A\otimes A$ such that $r+\tau(r)$ is invariant. Then the following conditions are equivalent:
	\begin{enumerate}
		\item\label{rrb0}
		$r$ is a solution of the SAPP-YBE in $(A,\triangleright_{A},\triangleleft_{A})$ such that $(A,\triangleright_{A},\triangleleft_{A},\vartheta_{r},\theta_{r})$ with $\vartheta_{r}$ and $\theta_{r}$ defined by \eqref{eq:comul} is a quasi-triangular \sappb.
		\item\label{rrb1} $r^{\sharp}$ is an $\calo$-operator of weight $-1$ of $(A,\triangleright_{A},\triangleleft_{A})$ associated to the
$(A,\triangleright_{A},\triangleleft_{A})$-represen-
tation algebra $(\mathcal{L}^{*}_{\circ_{A}}+\mathcal{R}^{*}_{\circ_{A}},
\mathcal{R}^{*}_{\triangleright_{A}},-\mathcal{R}^{*}_{\circ_{A}}, A^{*}, \triangleright_{r+\tau(r)}, \triangleleft_{r+\tau(r)})$, where the multiplications $ \triangleright_{r+\tau(r)}$ and $ \triangleleft_{r+\tau(r)}$ are given by
		\begin{eqnarray}
			&&a^{*} \triangleright_{r+\tau(r)} b^{*}=(\mathcal{L}^{*}_{\circ_{A}}+\mathcal{R}^{*}_{\circ_{A}})\Big(\big(r+\tau(r)\big)^{\sharp} (a^{*})\Big)b^{*},\label{eq:thm:6,1}\\
			&&a^{*} \triangleleft_{r+\tau(r)} b^{*}=-\mathcal{R}^{*}_{\circ_{A}}\Big(\big(r+\tau(r)\big)^{\sharp} (a^{*})\Big)b^{*},\;\;\forall a^{*},b^{*}\in A^{*}.\label{eq:thm:6,2}
		\end{eqnarray}
		That is, the following equations hold:
\begin{eqnarray} &&r^{\sharp}(a^{*})\triangleright_{A}r^{\sharp}(b^{*})
=r^{\sharp}\Big((\mathcal{L}^{*}_{\circ_{A}}
+\mathcal{R}^{*}_{\circ_{A}})\big(r^{\sharp}(a^{*})\big)b^{*}+ \mathcal{R}^{*}_{\triangleright_{A}}\big(r^{\sharp}(b^{*})\big)a^{*}
-a^{*} \triangleright_{r+\tau(r)}b^{*}\Big),\ \ \ \ \label{o-op-1}\\ &&r^{\sharp}(a^{*})\triangleleft_{A}r^{\sharp}(b^{*})
=r^{\sharp}\Big(-\mathcal{R}^{*}_{\circ_{A}}\big(r^{\sharp}(a^{*})\big)b^{*}
-\mathcal{R}^{*}_{\circ_{A}}\big(r^{\sharp}(b^{*})\big)a^{*}
-a^{*} \triangleleft_{r+\tau(r)}b^{*}\Big).\label{o-op-2}
\end{eqnarray}
	\end{enumerate}
\end{thm}

\begin{proof}
	By Proposition \ref{pro ss inv SDPP}, $(\mathcal{L}^{*}_{\circ_{A}}+\mathcal{R}^{*}_{\circ_{A}},
\mathcal{R}^{*}_{\triangleright_{A}},-\mathcal{R}^{*}_{\circ_{A}}, A^{*}, \triangleright_{r+\tau(r)}, \triangleleft_{r+\tau(r)})$ is an $(A,\triangleright_{A},\triangleleft_{A})$-representation algebra. On the other hand, we have
	\begin{eqnarray*}
		&&-r^{\sharp}\Big(\mathcal{R}^{*}_{\circ_{A}}\big(r^{\sharp}(a^{*})\big)b^{*}\Big)-r^{\sharp}\Big( \mathcal{R}^{*}_{\circ_{A}}\big(r^{\sharp}(b^{*})\big)a^{*}\Big)-r^{\sharp}(a^{*} \triangleleft_{r+\tau(r)}b^{*})\\
		&\overset{\eqref{eq:thm:6,2}}{=}&-r^{\sharp}\Big(\mathcal{R}^{*}_{\circ_{A}}\big(r^{\sharp}(a^{*})\big)b^{*}\Big)-r^{\sharp}\Big( \mathcal{R}^{*}_{\circ_{A}}\big(r^{\sharp}(b^{*})\big)a^{*}\Big)+r^{\sharp}\bigg(\mathcal{R}^{*}_{\circ_{A}}\Big(\big(r+\tau(r)\big)^{\sharp}(a^{*})\Big)b^{*}\bigg)\\
		&\overset{\eqref{eq:lem4}}{=}&-r^{\sharp}\Big(\mathcal{R}^{*}_{\circ_{A}}\big(r^{\sharp}(a^{*})\big)b^{*}\Big)-r^{\sharp}\Big( \mathcal{R}^{*}_{\circ_{A}}\big(r^{\sharp}(b^{*})\big)a^{*}\Big)+r^{\sharp}\bigg(\mathcal{R}^{*}_{\circ_{A}}\Big(\big(r+\tau(r)\big)^{\sharp}(b^{*})\Big)a^{*}\bigg)\\
		&=&
		r^{\sharp}\Big(\mathcal{R}^{*}_{\circ_{A}}\big( \tau(r)^{\sharp}(b^{*}) \big)a^{*}-\mathcal{R}^{*}_{\circ_{A}}\big(r^{\sharp}(a^{*})\big)b^{*} \Big)\\
		&\overset{\eqref{eq:dual space mul2}}{=}&r^{\sharp}(a^{*}\triangleleft_{r}b^{*}).	
	\end{eqnarray*}
	Since $r+\tau(r)$ is invariant, by \eqref{eq:t2}, \eqref{o-op-2} holds if and only if $SA(r)=0$. Similarly, 
\eqref{o-op-1} holds if and only if $SA(r)=0$. Hence the conclusion follows from Lemma~\ref{lem:TR}.
\end{proof}

Recall \cite{Bai2024} that a {\bf quadratic \sapp} is a quadruple $(A,\triangleright_{A},\triangleleft_{A},\mathcal{B})$, where $(A,\triangleright_{A},\triangleleft_{A})$ is a \sapp and $\mathcal{B}$ is a nondegenerate symmetric bilinear form on $A$ satisfying the following equation:
\begin{equation}\label{eq:cor4}
\mathcal{B}(x\triangleleft_{A}y,z)=-\mathcal{B}(x,z\circ_{A} y),\;\forall x,y,z\in A.
\end{equation}
Now we investigate the tensor
form of the bilinear form $\mathcal{B}$ in a quadratic \sapp
$(A,\triangleright_{A},\triangleleft_{A},\mathcal{B})$. 
\delete{Let $\mathcal{B}$ be a
nondegenerate bilinear form on a vector space $A$ which
corresponds to a linear map $\mathcal{B}^{\natural}:A\rightarrow
A^{*}$ by \eqref{eq:B}. Define a 2-tensor $\phi_{\mathcal{B}}\in A\otimes A$ to be
the tensor form of ${\mathcal{B}^{\natural}}^{-1}$, that is,
\begin{equation}\label{eq:2-tensor}
\langle \phi_{\mathcal{B}},a^{*}\otimes b^{*}\rangle=\langle {\mathcal{B}
^{\natural}}^{-1}(a^{*}),b^{*}\rangle,\;\forall a^{*},b^{*}\in A^{*}.
\end{equation}}

\begin{pro}\label{pro:2.24}
Let $(A,\triangleright_{A},\triangleleft_{A})$ be a \sapp and $\mathcal{B}$ be a
nondegenerate bilinear form on $A$. Then
$(A,\triangleright_{A},\triangleleft_{A},\mathcal{B})$ is a quadratic \sapp if and
only if $\phi_{\mathcal{B}}$ is symmetric and satisfies
$f(x)\phi_{\mathcal{B}}=0$ for all $x\in A$. Moreover, in this
case we have $g(x)\phi_{\mathcal{B}}=0$ for all $x\in A$, and thus
$\phi_{\mathcal{B}}$ is invariant on $(A,\triangleright_{A},\triangleleft_{A})$.
\end{pro}

\begin{proof}
It is clear that $\mathcal{B}$ is symmetric if and only if $\phi _{\mathcal{B}}$ is symmetric.
Let $x,y,z\in A$ and $a^{*}=\mathcal{B}^{\natural}(y),b^{*}=\mathcal{B}^{\natural}(z)$.
Under the assumptions, we have
\begin{eqnarray*}
&&\langle \big(\mathrm{id}\otimes \mathcal{R}_{\circ _{A}}(x)\big)\phi _{
    \mathcal{B}},a^{* }\otimes b^{* }\rangle =\langle \phi _{\mathcal{B}
},a^{* }\otimes \mathcal{R}_{\circ _{A}}^{* }(x)b^{* }\rangle
=\langle {\mathcal{B}^{\natural}}^{-1}(a^{* }),\mathcal{R}_{\circ
    _{A}}^{* }(x)b^{* }\rangle  \\
&&= \langle {\mathcal{B}^{\natural}}^{-1}(a^{* })\circ _{A}x,b^{*
}\rangle =\langle y\circ _{A}x,\mathcal{B}^{\natural}(z)\rangle =\mathcal{B}
(y\circ _{A}x,z), \\
&&-\langle \big(\mathcal{L}_{\triangleleft_{A}}(x)\otimes \mathrm{id}\big)\phi _{\mathcal{B}},a^{* }\otimes b^{* }\rangle =-\langle \phi _{\mathcal{B}},
\mathcal{L}_{\triangleleft_{A}}^{* }(x)a^{* }\otimes b^{* }\rangle
=-\langle \mathcal{L}_{\triangleleft_{A}}^{* }(x)a^{* },{\mathcal{B}^{\natural}}
^{-1}(b^{* })\rangle  \\
&&=-\langle a^{* },x\triangleleft_{A}{\mathcal{B}^{\natural}}^{-1}(b^{*
})\rangle =-\langle \mathcal{B}^{\natural}(y),x\triangleleft_{A}z\rangle =-\mathcal{B}(y,x\triangleleft_{A}z)=-\mathcal{B}(y,z\triangleleft_{A}x).
\end{eqnarray*}
Hence \eqref{eq:cor4} holds if and only if $f(x)\phi _{\mathcal{B}}=0$
for all $x\in A$. In this case, by \cite{Bai2024}, we obtain
\begin{equation}
\mathcal{B}(x\circ_{A}y,z)=\mathcal{B}(y,x\circ_{A}z),\;\forall x,y,z\in A,
\end{equation}
which indicates that $g(x)\phi _{\mathcal{B}}=0$ for all $x\in A$.
\end{proof}

Next we apply Theorem~\ref{pro:rrb} to the case of quadratic \sapps.

\begin{pro}\label{pro1-6}
	Let $(A,\triangleright_{A},\triangleleft_{A},\mathcal{B})$ be a quadratic \sapp.
	Suppose that there exists  $r\in A\otimes A$ such that $r+\tau(r)$ is invariant.
	Define a linear map $R:A\rightarrow A$ by \eqref{Br,Brt}.
	Then  $r$ is a solution of the SAPP-YBE in $(A,\triangleright_{A},\triangleleft_{A})$ if and only if the following equations hold:
		\begin{eqnarray}
			&&R(x)\triangleright_{A}R(y)=R \Big(R(x)\triangleright_{A}y+ x\triangleright_{A}R (y)- x\triangleright_{A}\big(r+\tau(r)\big)^{\sharp} \mathcal{B}^{\natural}(y)\Big),\label{eq:pro3,1}\\
			&&R(x)\triangleleft_{A}R(y)=R \Big(R(x)\triangleleft_{A}y+ x\triangleleft_{A}R (y)- x\triangleleft_{A}\big(r+\tau(r)\big)^{\sharp} \mathcal{B}^{\natural}(y)\Big),\;\forall x,y\in A.\ \ \ \ \label{eq:pro3,2}
		\end{eqnarray}
\end{pro}
\begin{proof}
	Let $x,y\in A$ and  $a^{*}=\mathcal{B}^{\natural}(x), b^{*}=\mathcal{B}^{\natural}(y)$. By Proposition \ref{pro:2.24}, $\phi_{\mathcal{B}}$ is symmetric and invariant. Then we have
\delete{
	\begin{small}
		\begin{eqnarray*}
			&&\langle x\triangleleft_{A}\phi_{\mathcal{B}}^{\sharp}(a^{*}),b^{*}\rangle=\langle \phi_{\mathcal{B}}^{\sharp}(a^{*}),\mathcal{L}^{*}_{\triangleleft_{A}}(x)b^{*}\rangle=\langle \phi_{\mathcal{B}},a^{*}\otimes\mathcal{L}^{*}_{\triangleleft_{A}}(x)b^{*}\rangle\\
			&&=\langle\big(\mathrm{id}\otimes\mathcal{L}_{\triangleleft_{A}}(x)\big)\phi_{\mathcal{B}}, a^{*}\otimes b^{*}\rangle =-\langle \big( \mathcal{R}_{\circ_{A}} (x)\otimes\mathrm{id}\big)\phi_{\mathcal{B}},a^{*}\otimes b^{*}\rangle=-\langle \phi_{\mathcal{B}}^{\sharp}\big( \mathcal{R}^{*}_{\circ_{A}} (x)a^{*}\big),b^{*}\rangle,
		\end{eqnarray*}
	\end{small}that is,
		\begin{equation}\label{eq:lem5-}		x\triangleleft_{A}\mathcal{B}^{{\natural}^{-1}}(a^{*})=-\mathcal{B}^{{\natural}^{-1}}\big( \mathcal{R}^{*}_{\circ_{A}} (x)a^{*}\big).
		\end{equation}}
		\begin{eqnarray*}
&&\langle y\triangleright_{A}\mathcal{B}^{{\natural}^{-1}}(a^{*}),b^{*}\rangle=\langle \mathcal{B}^{{\natural}^{-1}}(a^{*}),\mathcal{L}^{*}_{\triangleright_{A}}(y)b^{*}\rangle=\langle \phi_{\mathcal{B}},a^{*}\otimes\mathcal{L}^{*}_{\triangleright_{A}}(y)b^{*}\rangle\\
			&&=\langle\big(\mathrm{id}\otimes\mathcal{L}_{\triangleright_{A}}(y)\big)\phi_{\mathcal{B}}, a^{*}\otimes b^{*}\rangle =\langle\big(\mathrm{id}\otimes(\mathcal{L}_{\circ_{A}}-\mathcal{L}_{\triangleleft_{A}})(y)\big)\phi_{\mathcal{B}}, a^{*}\otimes b^{*}\rangle\\
&&=
\langle \big( (\mathcal{L}_{\circ_{A}}+\mathcal{R}_{\circ_{A}}) (y)\otimes\mathrm{id}\big)\phi_{\mathcal{B}},a^{*}\otimes b^{*}\rangle=\langle \mathcal{B}^{{\natural}^{-1}}\big( (\mathcal{L}^{*}_{\circ_{A}}+\mathcal{R}^{*}_{\circ_{A}}) (y)a^{*}\big),b^{*}\rangle,
		\end{eqnarray*}
that is,
\begin{equation}\label{eq:lem7}
y\triangleright_{A}\mathcal{B}^{{\natural}^{-1}}(a^{*})=\mathcal{B}^{{\natural}^{-1}}\big((\mathcal{L}^{*}_{\circ_{A}}+\mathcal{R}^{*}_{\circ_{A}})(y)a^{*}\big).
\end{equation}
Similarly, we have
\begin{equation}\label{eq:lem6} \mathcal{B}^{{\natural}^{-1}}(a^{*})\triangleright_{A}y=\mathcal{B}^{{\natural}^{-1}}\big(\mathcal{R}^{*}_{\triangleright_{A}}(y)a^{*}\big).
\end{equation}
Moreover, we have
\begin{align*}
			R(x)\triangleright_{A}R(y)&=r^{\sharp}(a^{*})\triangleright_{A}r^{\sharp}(b^{*}),\\
			R\big(R(x)\triangleright_{A}y\big)&=r^{\sharp}\mathcal{B}^{\natural}\big(r^{\sharp}(a^{*})\triangleright_{A}\mathcal{B}^{{\natural}^{-1}}(b^{*})\big)
			\overset{\eqref{eq:lem7}}{=}r^{\sharp}\Big((\mathcal{L}^{*}_{\circ_{A}}+\mathcal{R}^{*}_{\circ_{A}})\big(r^{\sharp}(a^{*})\big)b^{*}\Big),\\
			R\big(x\triangleright_{A}R(y)\big)&=r^{\sharp}\mathcal{B}^{\natural}\big(\mathcal{B}^{{\natural}^{-1}}(a^{*})\triangleright_{A}r^{\sharp}(b^{*})\big)
			\overset{\eqref{eq:lem6}}{=}r^{\sharp}\Big(\mathcal{R}^{*}_{\triangleright_{A}}\big(r^{\sharp}(b^{*})\big)a^{*}\Big),\\
			-R\Big(x\triangleright_{A}\big(r+\tau(r)\big)^{\sharp}\mathcal{B}^{\natural}(y)\Big)&=
			-r^{\sharp}\mathcal{B}^{\natural}\Big(\mathcal{B}^{{\natural}^{-1}}(a^{*})\triangleright_{A}\big(r+\tau(r)\big)^{\sharp}(b^{*})\Big)\\&\overset{\eqref{eq:lem6}}{=}
			-r^{\sharp}\bigg(\mathcal{R}^{*}_{\triangleright_{A}}\Big(\big(r+\tau(r)\big)^{\sharp}(b^{*})\Big)a^{*}\bigg)\\
			&\overset{\eqref{eq:lem3},\eqref{eq:lem4}}{=}-r^{\sharp}\bigg((\mathcal{L}^{*}_{\circ_{A}}+\mathcal{R}^{*}_{\circ_{A}})\Big(\big(r+\tau(r)\big)^{\sharp}(a^{*})\Big) b^{*} \bigg)\\
&=-r^{\sharp}(a^{*} \triangleright_{r+\tau(r)}b^{*}).
\end{align*}
Hence \eqref{eq:pro3,1} holds if and only if \eqref{o-op-1} holds. Similarly, \eqref{eq:pro3,2} holds if and only if \eqref{o-op-2} holds. Thus the conclusion follows from Theorem~\ref{pro:rrb}.
\end{proof}

\delete{
\begin{proof}
	Let $x,y\in A$ and  $a^{*}=\mathcal{B}^{\natural}(x), b^{*}=\mathcal{B}^{\natural}(y)$. By Proposition \ref{pro:2.24}, $\phi_{\mathcal{B}}$ is symmetric and invariant. Then we have
\begin{eqnarray*}
		R(x) \triangleleft_{A}R(y)&=&r^{\sharp}(a^{*}) \triangleleft_{A}r^{\sharp}(b^{*}),\\
		R\big(R(x) \triangleleft_{A}y\big)&=&r^{\sharp}\mathcal{B}^{\natural}\big(r^{\sharp}(a^{*}) \triangleleft_{A}\mathcal{B}^{{\natural}^{-1}}(b^{*})\big)\overset{\eqref{eq:lem5}}{=}r^{\sharp}\Big(\mathcal{R}^{*}_{\circ_{A}}\big(r^{\sharp}(a^{*})\big)b^{*}\Big),\\
		R\big(x \triangleleft_{A}R(y)\big)&=&r^{\sharp}\mathcal{B}^{\natural}\big(\mathcal{B}^{{\natural}^{-1}}(a^{*}) \triangleleft_{A}r^{\sharp}(b^{*})\big)\overset{\eqref{eq:lem5}}{=}r^{\sharp}\Big(\mathcal{R}^{*}_{\circ_{A}}\big(r^{\sharp}(b^{*})\big)a^{*}\Big),\\
		-R\Big(x \triangleleft_{A}\big(r+\tau(r)\big)^{\sharp}\mathcal{B}^{\natural}(y)\Big)&=&
		-r^{\sharp}\mathcal{B}^{\natural}\Big(\mathcal{B}^{{\natural}^{-1}}(a^{*}) \triangleleft_{A}\big(r+\tau(r)\big)^{\sharp}(b^{*})\Big)\\&\overset{\eqref{eq:lem1}}{=}&
		-r^{\sharp}\bigg(\mathcal{R}^{*}_{\circ_{A}}\Big(\big(r+\tau(r)\big)^{\sharp}(b^{*})\Big)a^{*}\bigg)\\
		&\overset{\eqref{eq:lem4}}{=}&-r^{\sharp}\bigg(\mathcal{R}^{*}_{\circ_{A}}\Big(\big(r+\tau(r)\big)^{\sharp}(a^{*})\Big) b^{*} \bigg)=-r^{\sharp}(a^{*}\lhd_{r+\tau(r)}b^{*}).
\end{eqnarray*}Hence \eqref{eq:pro3,2} holds if and only if \eqref{o-op-2} holds. Similarly, \eqref{eq:pro3,1} holds if and only if \eqref{o-op-1} holds. Thus the conclusion follows from Theorem~\ref{pro:rrb}.
\end{proof}}

\delete{
	
	In this section, we study triangular and factorizable SDPP bialgebras, which are subclasses of quasi-triangular SDPP bialgebras $(A,\circ_{A}, \triangleleft_{A},\eta_{r}, \theta_{r})$ for the distinct cases that $r+\tau(r)=0$ and $\big(r+\tau(r)\big)^{\sharp}$ is nondegenerate respectively.
	On the one hand, triangular SDPP bialgebras can be realized from $\mathcal{O}$-operators of SDPP algebras where the multiplications on the representation spaces are simplified to be trivial.
	Moreover, such $\mathcal{O}$-operators can be provided by pre-SDPP algebras and hence the latter give rise to triangular SDPP bialgebras.
	On the other hand,  a  factorizable SDPP bialgebra  leads to a factorization of the underlying SDPP algebra, and an arbitrary Manin triple of SDPP algebras admits a factorizable SDPP bialgebra structure.}

\subsection{Triangular \sappbs}\label{sec5.1}\

Let $(A,\triangleright_{A},\triangleleft_{A})$ be a \sapp and $r\in A\otimes A$. If $r$ is a skew-symmetric solution of the SAPP-YBE in $(A,\triangleright_{A},\triangleleft_{A})$, then by Theorem \ref{thm:bialg}, $(A,\triangleright_{A},\triangleleft_{A},\vartheta_{r},\theta_{r})$ is a \sappb, where $\vartheta_{r}$ and $\theta_{r}$ are defined by \eqref{eq:comul}. In this case, we say $(A,\triangleright_{A},\triangleleft_{A},\vartheta_{r},\theta_{r})$ is \textbf{triangular}.


\begin{pro}\label{thm:T1}
	Let $(A,\triangleright_{A},\triangleleft_{A})$ be a \sapp and $r\in A\otimes A$ be skew-symmetric. Then $r$ is a solution of the SAPP-YBE in $(A,\triangleright_{A},\triangleleft_{A})$ if and only if $r^{\sharp}$ is an $\calo$-operator of $(A,\triangleright_{A},\triangleleft_{A})$ associated to $(\mathcal{L}^{*}_{\circ_{A}}+\mathcal{R}^{*}_{\circ_{A}},\mathcal{R}^{*}_{\triangleright_{A}}, -\mathcal{R}^{*}_{\circ_{A}},A^{*})$.
\end{pro}
\begin{proof}
It follows from Theorem~\ref{pro:rrb} by observing $r+\tau(r)=0$.
\end{proof}

\begin{defi}
	A \textbf{quadratic Rota-Baxter \sapp of weight $\lambda$} is a quintuple  $(A,\triangleright_{A},\triangleleft_{A},R,\mathcal{B})$, such that $(A,\triangleright_{A},\triangleleft_{A},\mathcal{B})$ is a quadratic \sapp, $R$ is a Rota-Baxter operator of weight $\lambda$ on $(A,\triangleright_{A},\triangleleft_{A})$ and  \eqref{strong} holds.
\end{defi}

Next we establish the relationship between symmetric Rota-Baxter Frobenius commutative algebras and quadratic Rota-Baxter \sapps with the same weights.

\begin{pro}\label{pro:6.31}
	Let $(A,\cdot_{A},P,R,\mathcal{B})$ be a symmetric averaging Rota-Baxter Frobenius commutative algebra of weight $\lambda$.
	\delete{
	Suppose that $P$ is an averaging operator of $(A,\cdot_{A})$ which commutes with $R$ and $\hat{P}$ is the adjoint map of $P$ with respect to $\mathcal{B}$ defined by
\begin{equation}
\mathcal{B}\big(P(x),y\big)=\mathcal{B}\big(x,\hat{P}(y)\big),\;\forall x,y\in A.
\end{equation}}Then $(A,\triangleright_{A},\triangleleft_{A},R,\mathcal{B})$ is a quadratic Rota-Baxter \sapp of weight $\lambda$,
where the multiplications $\triangleright_{A},\triangleleft_{A}:A\otimes A\rightarrow A$ are given by
\begin{equation}\label{eq:90}
x\triangleright_{A}y=P(x)\cdot_{A}y+\hat{P}(x\cdot_{A}y),\; x\triangleleft_{A}y=-\hat{P}(x\cdot_{A}y),\;\forall x,y\in A.
\end{equation}
\end{pro}
\begin{proof}
	By \cite{Bai2024}, $(A,\triangleright_{A},\triangleleft_{A},\mathcal{B})$ is a quadratic \sapp. By Theorem \ref{thm:commute}, $\hat{P}$ also commutes with $R$. For all $x,y\in A$, we have
	\begin{eqnarray*}
&&R(x)\triangleleft_{A}R(y)-R\big( R(x)\triangleleft_{A}y+x\triangleleft_{A} R(y)+\lambda x\triangleleft_{A}y \big)\\
		&&=-\hat{P}\big( R(x)\cdot_{A}R(y) \big)+R\hat{P}\big(R(x)\cdot_{A}y+x\cdot_{A}R(y)+\lambda x\cdot_{A}y  \big)\\
		&&=-\hat{P}\Big( R(x)\cdot_{A}R(y)-R\big(R(x)\cdot_{A}y+x\cdot_{A}R(y)+\lambda x\cdot_{A}y\big)  \Big)\\
		&&=0, 
	\end{eqnarray*}
and similarly
\begin{eqnarray*}
R(x)\triangleright_{A}R(y)-R\big( R(x)\triangleright_{A}y+x\triangleright_{A}R(y)+\lambda x\triangleright_{A}y \big)=0.
\end{eqnarray*}
	Thus $(A,\triangleright_{A},\triangleleft_{A},R,\mathcal{B})$ is a quadratic Rota-Baxter \sapp of weight $\lambda $.
\end{proof}

\delete{
	\begin{pro}\label{pro:4.2}
		Let $\beta$ be a Rota-Baxter operator of   weight $\lambda$ on a SDPP algebra $(A, \triangleright_{A}, \triangleleft_{A})$.
		Then $$\big(A\ltimes_{\mathcal{L}^{*}_{\circ_{A}},\mathcal{L}^{*}_{ \triangleleft_{A}},\mathcal{R}^{*}_{\circ_{A}}} A^{*},\mathcal{B}_{d},\beta-(\beta+\lambda\mathrm{id}_{A} )^{*}\big)$$ is a quadratic Rota-Baxter SDPP algebra of weight $\lambda$, where the bilinear form $\mathcal{B}_{d}$ on $A\oplus A^{*}$ is given by \eqref{eq:bfds}.
	\end{pro}
	\begin{proof}
		It follows from a straightforward computation.
	\end{proof}

	\begin{ex}\label{ex:6.13}
		Let $(A, \triangleright_{A}, \triangleleft_{A})$ be a SDPP algebra.
		Then the identity map $\mathrm{id}_{A}$ is a Rota-Baxter operator of $(A, \triangleright_{A}, \triangleleft_{A})$ of weight $-1$.
		By Proposition \ref{pro:4.2},
		\begin{equation*}
			\big(A\ltimes_{\mathcal{L}^{*}_{\circ_{A}},\mathcal{L}^{*}_{ \triangleleft_{A}},\mathcal{R}^{*}_{\circ_{A}}} A^{*},\mathcal{B}_{d},\mathrm{id}_{A}-(\mathrm{id}_{A}-\mathrm{id}_{A})^{*}\big)= (A\ltimes_{\mathcal{L}^{*}_{\circ_{A}},\mathcal{L}^{*}_{ \triangleleft_{A}},\mathcal{R}^{*}_{\circ_{A}}} A^{*},\mathcal{B}_{d},\mathrm{id}_{A})
		\end{equation*}
		is a quadratic Rota-Baxter SDPP algebra of weight $-1$.
\end{ex}}

\begin{lem}\label{lem:bf}
	Let $A$ be a vector space and $\mathcal{B}$ be a nondegenerate symmetric bilinear form.
	Let $r\in A\otimes A$ and $R:A\rightarrow A$ satisfy \eqref{Br,Brt}.
	Then $r$satisfies
	\begin{equation}\label{rw}
		r+\tau(r)=-\lambda\phi_{\mathcal{B}},\; \lambda\in\mathbb{K}
	\end{equation}
	if and only if $R$  satisfies \eqref{strong}.
\end{lem}
\begin{proof}
It is similar to the proof of \cite[Lemma 4.17]{BLST}.
\end{proof}

\delete{
\begin{proof}
	Let $x,y\in A$ and $a^{*}=\mathcal{B}^{\natural}(x), b^{*}=\mathcal{B}^{\natural}(y)$.
	Then
	\begin{eqnarray*}
		\mathcal{B}\big(R (x),y\big)&=& \langle \mathcal{B}^{\natural}(y),r^{\sharp}\mathcal{B}^{\natural}(x)\rangle=\langle b^{*}, r^{\sharp}(a^{*})\rangle=\langle r, a^{*}\otimes b^{*}\rangle,\\
		\mathcal{B}\big(x,R (y)\big)&=& \langle\mathcal{B}^{\natural}(x), r^{\sharp}\mathcal{B}^{\natural}(y)\rangle=\langle a^{*}, r^{\sharp}(b^{*})\rangle=\langle \tau(r), a^{*}\otimes b^{*}\rangle,\\
		\lambda\mathcal{B}(x,y)&=& \lambda \langle \mathcal{B}^{\natural^{-1}}\mathcal{B}^{\natural}(x),\mathcal{B}^{\natural}(y)\rangle=\lambda  \langle  \mathcal{B}^{\natural^{-1}} (a^{*}),b^{*}\rangle=\lambda\langle\phi_{\mathcal{B}}, a^{*}\otimes b^{*}\rangle.
	\end{eqnarray*}
	Hence $r$ satisfies  \eqref{rw} if and only if $R$ satisfies \eqref{strong}.
\end{proof}}

\begin{pro}\label{pro:triangular Leib}
	Let $(A,\triangleright_{A},\triangleleft_{A},R,\mathcal{B})$ be a quadratic Rota-Baxter \sapp of weight $0$.
	Then there is a triangular \sappb $(A,\triangleright_{A},\triangleleft_{A},\vartheta_{r},\theta_{r})$ with $\vartheta_r$ and $\theta_{r}$ defined by \eqref{eq:comul},
where $r\in A\otimes A$ is given through the operator form $r^{\sharp}$ by \eqref{eq:thm:quadratic to fact}.
\end{pro}
\begin{proof}
It follows from Proposition \ref{pro1-6} and Lemma \ref{lem:bf} by observing $r+\tau(r)=0$.
\end{proof}

\begin{thm}\label{thm:TY3}
	Let $(A,\triangleright_{A},\triangleleft_{A})$ be a \sapp with a representation $(l_{\triangleright_{A}},$
$r_{\triangleright_{A}},l_{\triangleleft_{A}},V)$. Let $T:V\rightarrow A$ be a linear map which is identified as
	\begin{equation*}
		T_{\sharp}\in V^{*}\otimes A \subset (A\ltimes_{l^{*}_{\circ_{A}}+r^{*}_{\circ_{A}},r^{*}_{\triangleright_{A}}, -r^{*}_{\circ_{A}}}V^{*})\otimes (A\ltimes_{l^{*}_{\circ_{A}}+r^{*}_{\circ_{A}},r^{*}_{\triangleright_{A}}, -r^{*}_{\circ_{A}}}V^{*}) .
	\end{equation*}
	\delete{an element in $(A\ltimes
		_{l^{*},h^{*},r^{*}}V^{*})\otimes (A\ltimes_{l^{*},h^{*},r^{*}}V^{*})$(through Hom$(V,A)\cong V^{*}\otimes A\subset (A\ltimes_{l^{*},h^{*},r^{*}}V^{*})\otimes (A\ltimes_{l^{*},h^{*},r^{*}}V^{*}))$.}
	Then $r=T_{\sharp}-\tau(T_{\sharp})$ is a skew-symmetric solution of the SAPP-YBE in $A\ltimes_{l^{*}_{\circ_{A}}+r^{*}_{\circ_{A}},r^{*}_{\triangleright_{A}}, -r^{*}_{\circ_{A}}}V^{*}$ if and only if $T$ is an $\calo$-operator of $(A,\triangleright_{A},\triangleleft_{A})$ associated to $(l_{\triangleright_{A}},r_{\triangleright_{A}},l_{\triangleleft_{A}},V)$.
\end{thm}
\begin{proof}
Let $\{v_{1}$, $\cdots$, $v_{n}\}$ be a basis of $V$ and  $\{v^{*}_{1},\cdots,v^{*}_{n}\}$ be the dual basis. We have
	\begin{equation*}
		T_{\sharp}=\sum\limits_{i=1}^{n}v^{*}_{i}\otimes T(v_{i}),\;r=\sum\limits_{i=1}^{n}v^{*}_{i}\otimes T(v_{i}) -T(v_{i})\otimes v^{*}_{i}.
	\end{equation*}
Notice that
		\begin{eqnarray*}
			\sum_{i,k} T(v_{i})\otimes l^{*}_{\triangleright_{A}}\big(T(v_{k})\big)v^{*}_{i}=\sum_{i,k} T\Big(l_{\triangleright_{A}}\big(T(v_{k})\big)v_{i}\Big)\otimes v^{*}_{i},
	\end{eqnarray*}
and similarly for linear maps $r_{\triangleright_{A}}$ and $l_{\triangleleft_{A}}$. Hence we obtain
	\delete{
		\begin{eqnarray}
			T(v_{i})\otimes l^{*}\big(T(v_{k})\big)v^{*}_{i}
			&=&T\Big(l\big(T(v_{k})\big)v_{i}\Big)\otimes v^{*}_{i},\label{Eq:condition1}\\
			l^{*}\big(T(v_{k})\big)v^{*}_{i}\otimes T(v_{i})&=&v^{*}_{i}\otimes T\Big(l\big(T(v_{k})\big)v_{i}\Big),\label{Eq:condition4}
	\end{eqnarray}}
{\small
		\begin{align*}
			SA(r)=&\sum_{i,j}T(v_{i})\circ_{A} T(v_{j})\otimes v^{*}_{i}\otimes v^{*}_{j}-v^{*}_{i}\circ T(v_{j})\otimes T(v_{i})\otimes v^{*}_{j}-T(v_{i})\circ v^{*}_{j}\otimes v^{*}_{i}\otimes T(v_{j})\\
			&+v^{*}_{i}\otimes T(v_{i})\triangleleft v^{*}_{j}\otimes T(v_{j})-v^{*}_{i}\otimes T(v_{i})\triangleleft_{A} T(v_{j})\otimes v^{*}_{j}+T(v_{i})\otimes v^{*}_{i}\triangleleft T(v_{j})\otimes v^{*}_{j}\\
			&+v^{*}_{j}\otimes v^{*}_{i}\otimes T(v_{i})\circ_{A} T(v_{j})-T(v_{j})\otimes v^{*}_{i}\otimes T(v_{i})\circ v^{*}_{j}-v^{*}_{j}\otimes T(v_{i})\otimes v^{*}_{i}\circ T(v_{j})\\
			=&\sum_{i,j}T(v_{i})\circ_{A} T(v_{j})\otimes v^{*}_{i}\otimes v^{*}_{j}+l_{\triangleleft_{A}}^{*}\big(T(v_{j})\big)v^{*}_{i}\otimes T(v_{i})\otimes v^{*}_{j}
		-l^{*}_{\circ_{A}}\big(T(v_{i})\big)v^{*}_{j}\otimes v^{*}_{i}\otimes T(v_{j})\\
			&-v^{*}_{i}\otimes r^{*}_{\circ_{A}}\big(T(v_{i})\big)v^{*}_{j}\otimes T(v_{j})
			-v^{*}_{i}\otimes T(v_{i})\triangleleft_{A} T(v_{j})\otimes v^{*}_{j}-T(v_{i})\otimes r^{*}_{\circ_{A}}\big(T(v_{j})\big)v^{*}_{i}\otimes v^{*}_{j}\\
			&+v^{*}_{j}\otimes v^{*}_{i}\otimes T(v_{i})\circ_{A} T(v_{j})-T(v_{j})\otimes v^{*}_{i}\otimes l^{*}_{\circ_{A}}\big(T(v_{i})\big)v^{*}_{j}+v^{*}_{j}\otimes T(v_{i})\otimes l^{*}_{\triangleleft_{A}}\big(T(v_{j})\big)v^{*}_{i}\\
			=&\sum_{i,j}\bigg(T(v_{i})\circ_{A} T(v_{j})-T\Big( l_{\circ_{A}} \big(T(v_{i})\big)v_{j}+r_{\circ_{A}}\big(T(v_{j})\big)v_{i}\Big)\bigg)\otimes v^{*}_{i}\otimes v^{*}_{j}\\
			&+v^{*}_{j}\otimes v^{*}_{i}\otimes \bigg(T(v_{i})\circ_{A} T(v_{j})-T\Big(l_{\circ_{A}}\big(T(v_{i})\big)v_{j}+r_{\circ_{A}}\big(T(v_{j})\big)v_{i}\Big)\bigg)\\
			&-v^{*}_{i}\otimes \bigg(T(v_{i})\triangleleft_{A} T(v_{j})-T\Big(l_{\triangleleft_{A}}\big(T(v_{i})\big)v_{j}+l_{\triangleleft_{A}}\big(T(v_{j})\big)v_{i}\Big)\bigg)\otimes v^{*}_{j}.
		\end{align*}}Therefore $SA(r)=0$ if and only if the following equations hold:
	\begin{eqnarray*}
		&&T(v_{i}) \triangleright_{A} T(v_{j})=T\Big( l_{\triangleright_{A}}\big(T(v_{i})v_{j}\big)+r_{\triangleright_{A}}\big(T(v_{j})\big)v_{i}  \Big),\\
		&&T(v_{i}) \triangleleft_{A} T(v_{j})=T\Big( l_{\triangleleft_{A}}\big(T(v_{i})v_{j}\big)+l_{\triangleleft_{A}}\big(T(v_{j})\big)v_{i}  \Big),\;\forall i,j\in\{1,\cdots, n\},
	\end{eqnarray*}
	that is,
	$T$ is an $\calo$-operator of $(A,\triangleright_{A},\triangleleft_{A})$ associated to $(l_{\triangleright_{A}},r_{\triangleright_{A}},l_{\triangleleft_{A}},V)$.
\end{proof}

\delete{
	\begin{pro}
		Let $(A,\cdot_{A})$ be a commutative  algebra with a representation $(\mu,V)$ and $T:V\rightarrow A$ be a linear map.
		\begin{enumerate}
			\item\label{r1}
			There is a perm symmetric algebra $(A,\circ_{A},\star_{A})$ given by $x\circ_{A}y=x\cdot_{A} y,\; x\star_{A}y=0$. Moreover, $(\mu^{*},0,\mu^{*},V^{*})$ is a representation of $(A,\circ_{A},\star_{A})$. Consequently there is a perm symmetric algebra $A\ltimes _{\mu^{*},0,\mu^{*}}V^{*}$ on the direct sum $A\oplus V^{*}$ of vector spaces, whose multiplications are given by
			\begin{eqnarray}
				&&(x+u^{*})\circ(y+v^{*})=x\cdot_{A} y+\mu^{*}(x)v^{*},\\
				&&(x+u^{*})\star(y+v^{*})=\mu^{*}(x)v^{*}+\mu^{*}(y)u^{*},\;\forall x,y\in A, u^{*},v^{*}\in V^{*}.
			\end{eqnarray}
			\item\label{r2}
			$r=T_{\sharp}-\tau(T_{\sharp})=\sum\limits_{i}v^{*}_{i}\otimes T(v_{i})-T(v_{i})\otimes v^{*}_{i}$ is a skew-symmetric solution of PSYBE in $A\ltimes _{\mu^{*},0,\mu^{*}}V^{*}$ if and only if $T$ is an $\calo$-operator of $(A,\cdot_{A})$ associated to $(\mu,V)$.
		\end{enumerate}
	\end{pro}
	
	\begin{proof}
		item \ref{r1} follows from a straightforward computation. We only prove item \ref{r2}. By \eqref{Eq:condition1} and \eqref{Eq:condition4} for $l=\mu$, we have
		\begin{small}
			\begin{eqnarray*}
				SA(r)&=&\sum_{i,j}-T(v_{i})\circ v^{*}_{j}\otimes v^{*}_{i}\otimes T(v_{j})+T(v_{i})\circ T(v_{j})\otimes v^{*}_{i}\otimes v^{*}_{j}-v^{*}_{i}\otimes T(v_{i})\star v^{*}_{j}\otimes T(v_{j})\\
				&&-T(v_{i})\otimes v^{*}_{i}\star T(v_{j})\otimes v^{*}_{j}+v^{*}_{j}\otimes v^{*}_{i}\otimes T(v_{i})\circ T(v_{j})-T(v_{j})\otimes v^{*}_{i}\otimes T(v_{i})\circ v^{*}_{j}\\
				&=&\sum_{i,j}-\mu^{*}\big(T(v_{i})\big)v^{*}_{j}\otimes v^{*}_{i}\otimes T(v_{j})+T(v_{i})\circ T(v_{j})\otimes v^{*}_{i}\otimes v^{*}_{j}-v^{*}_{i}\otimes \mu^{*}\big(T(v_{i})\big)v^{*}_{j}\otimes T(v_{j})\\
				&&-T(v_{i})\otimes \mu^{*}\big(T(v_{j})\big)v^{*}_{i}\otimes v^{*}_{j}+v^{*}_{j}\otimes v^{*}_{i}\otimes T(v_{i})\circ T(v_{j})-T(v_{j})\otimes v^{*}_{i}\otimes \mu^{*}\big(T(v_{i})\big)v^{*}_{j}\\
				&=&\sum_{i,j}\bigg(T(v_{i})\circ_{A} T(v_{j})-T\Big(\mu\big(T(v_{i})\big)v_{j}+\mu\big(T(v_{j})\big)v_{i}\Big)\bigg)\otimes v^{*}_{i}\otimes v^{*}_{j}\\
				&&+v^{*}_{j}\otimes v^{*}_{i}\otimes \bigg(T(v_{i})\circ_{A} T(v_{j})-T\Big(\mu\big(T(v_{i})\big)v_{j}+\mu\big(T(v_{j})\big)v_{i}\Big)\bigg).
			\end{eqnarray*}
		\end{small}Hence the conclusion follows.
\end{proof}}

\delete{
	\begin{defi}\cite{BLZ}
		A \textbf{pre-perm algebra} is a triple$(A, \triangleright_{A}, \triangleleft_{A})$, where $A$ is a vector space, and $ \triangleright_{A}$ and $ \triangleleft_{A}$ are multiplications such that
		\begin{small}
			\begin{eqnarray}
				&&x \triangleleft_{A}(y \triangleright_{A} z+y \triangleleft_{A} z)=(x \triangleleft_{A} y) \triangleleft_{A} z=(y \triangleright_{A} x) \triangleleft_{A} z=y \triangleright_{A}(x \triangleleft_{A} z),\\
				&&(x \triangleright_{A} y+x \triangleleft_{A} y) \triangleright_{A} z=x \triangleright_{A} (y \triangleright_{A} z)=y \triangleright_{A}(x \triangleright_{A} z),\;\forall x,y,z\in A.
			\end{eqnarray}
		\end{small}
\end{defi}}

Recall \cite{LZB} that a {\bf pre-perm algebra} $(A,\succ_{A},\prec_{A})$ is a vector space $A$ together with multiplications $ \succ_{A},\prec_{A}:A\otimes A\rightarrow A $ such that the following equations hold:
\begin{eqnarray*}
&&x\succ_{A}(y\succ_{A} z)=y\succ_{A} (x\succ_{A} z)=(x\circ_{A}y)\succ_{A} z,\\
&&x\prec_{A}(y\circ_{A}z)=(x\prec_{A} y)\prec_{A} z=(y \succ_{A} x)\prec_{A} z=y \succ_{A} (x \prec_{A} z),\;\forall x,y,z\in A,
\end{eqnarray*}
where $x\circ_{A}y=x\succ_{A}y+x\prec_{A}y$. Next we intoduce the notion of pre-\sapps to construct skew-symmetric solutions of the SAPP-YBE.

\begin{defi}
	Let $A$ be a vector space with multiplications $ \smallfrown_{A},\smallsmile_{A},\diamond_{A}:A\otimes A\rightarrow A$.
	Let multiplications $\triangleright_{A},\triangleleft_{A},\succ_{A},\prec_{A},\circ_{A}:A\otimes A\rightarrow A$  be given by
	\begin{eqnarray}
&&x\triangleright_{A}y=x\smallfrown_{A}y+x\smallsmile_{A}y,\;\;
x\triangleleft_{A}y=x\diamond_{A}y+y\diamond_{A}x,\label{eq:pre-sdpp}\\
&&x\succ_{A}y=x\smallfrown_{A}y+x\diamond_{A} y,\;\;\;
x\prec_{A} y=x\smallsmile_{A}y+y\diamond_{A} x,\\
&&x\circ_{A}y=x\triangleright_{A}y+x\triangleleft_{A}y=x\succ_{A}y+x\prec_{A} y,\;
\forall x,y\in A.
	\end{eqnarray}
If $(A,\succ_{A},\prec_{A})$ is a pre-perm algebra, and the following equations hold:
\begin{eqnarray}
	&&(x\circ_{A}y)\diamond_{A} z
	=x\succ_{A}(y\diamond_{A}z)=-x\diamond_{A}(y\diamond_{A}z)=y\diamond_{A}(x\succ_{A}z),\\
	&&(x\triangleleft_{A}y)\diamond_{A}z=-y\diamond_{A}(z\prec_{A} x)=-z\prec_{A}(x\triangleleft_{A}y),\;\forall x,y,z\in A,
\end{eqnarray}
then we say $(A,\smallfrown_{A},\smallsmile_{A},\diamond_{A})$ is a {\bf pre-\sapp}.
\end{defi}

\delete{
	For a pre-perm symmetric algebra $(A, \triangleright_{A}, \triangleleft_{A},\diamond_{A})$, define three linear maps $\mathcal{L}_{ \triangleright_{A}}, \mathcal{R}_{ \triangleleft_{A}}, \mathcal{L}_{\diamond_{A}}:A\rightarrow$ End$(A)$ respectively by
	\begin{equation*}
		\mathcal{L}_{ \triangleright_{A}}(x)y=x \triangleright_{A} y,\; \mathcal{R}_{ \triangleleft_{A}}(x)y=y \triangleleft_{A} x,\; \mathcal{L}_{\diamond_{A}}(x)y=x\diamond_{A} y,\;\forall x,y\in A.
	\end{equation*}
	By a direct computation, we have the following conclusion.}

\begin{pro}\label{pro:TY2}
	Let $(A,\smallfrown_{A},\smallsmile_{A},\diamond_{A})$ be a pre-\sapp. Then  $(A,\triangleright_{A},\triangleleft_{A})$ given by \eqref{eq:pre-sdpp} is a \sapp, which is called the \textbf{sub-adjacent \sapp} of $(A,\smallfrown_{A},$
	$\smallsmile_{A},\diamond_{A})$. Moreover, $(\mathcal{L}_{\smallfrown_{A}}, \mathcal{R}_{\smallsmile_{A}}, \mathcal{L}_{\diamond_{A}},A)$ is a representation of $(A,\triangleright_{A},\triangleleft_{A})$ and the identity map $\mathrm{id}$ is an $\calo$-operator of $(A,\triangleright_{A},\triangleleft_{A})$ associated to $(\mathcal{L}_{\smallfrown_{A}}, \mathcal{R}_{\smallsmile_{A}}, \mathcal{L}_{\diamond_{A}},A)$.
\end{pro}
\begin{proof}
It follows from a straightforward computation.
\end{proof}

\begin{pro}\label{pro:3.34}
	Let $(A,\star_{A},P,Q)$ be an admissible averaging Zinbiel algebra.
	Then there is a pre-\sapp $(A,\smallfrown_{A},\smallsmile_{A},\diamond_{A})$ given by
		\begin{equation*}
			x\smallfrown_{A} y=P(x)\star_{A} y+Q(x\star_{A} y),\;
			x\smallsmile_{A}  y=y\star_{A} P(x)+Q(y\star_{A}x),\;
			x\diamond_{A} y=-Q(x\star_{A} y), 
		\end{equation*}
for all $x,y\in A$.
\end{pro}

\begin{proof}
	It is straightforward.
\end{proof}

\begin{pro}\label{pro:3.35}
	Let $(A,\smallfrown_{A},\smallsmile_{A},\diamond_{A})$ be a pre-\sapp and $(A,\triangleright_{A},\triangleleft_{A})$ be the sub-adjacent \sapp of $(A,\smallfrown_{A},\smallsmile_{A},\diamond_{A})$. Let $\{e_{1}$, $\cdots$, $e_{n}\}$ be a basis of $A$ and  $\{e^{*}_{1},\cdots,e^{*}_{n}\}$ be the dual basis. Then
	$r$ given by \eqref{1109}
	is a skew-symmetric solution of the SAPP-YBE in the \sapp $$A\ltimes A^{*}:=A\ltimes_{\mathcal{L}^{*}_{\smallfrown_{A}}
+\mathcal{R}^{*}_{\smallsmile_{A}}+2\mathcal{L}^{*}_{\diamond_{A}},
\mathcal{R}^{*}_{\smallsmile_{A}},-\mathcal{L}^{*}_{\diamond_{A}}
-\mathcal{R}^{*}_{\smallsmile_{A}}}A^{*}.$$
Therefore there is a triangular \sappb $(A\ltimes A^{*},\vartheta_{r},\theta_{r})$, where the linear maps $\vartheta_{r}$ and $\theta_{r}$ are defined by \eqref{eq:comul} with the above $r$.
\end{pro}

\begin{proof}
	By Proposition \ref{pro:TY2}, the identity map $\mathrm{id}$ is an $\calo$-operator of $(A,\triangleright_{A},\triangleleft_{A})$ associated to $(\mathcal{L}_{\smallfrown_{A}}, \mathcal{R}_{\smallsmile_{A}}, \mathcal{L}_{\diamond_{A}}, A)$. Hence by Theorem \ref{thm:TY3},
	$r$ given by \eqref{1121}
	is a skew-symmetric solution of the SAPP-YBE in  $A\ltimes A^{*}$.
Hence the conclusion follows.
\end{proof}

Combing Propositions \ref{pro:2.28}, \ref{pro:quasi bialgebras}, \ref{pro:3.34} and \ref{pro:3.35}, we have the following commutative diagram which has already been shown in the Introduction.
 {\small
    	\begin{equation*}
    		\begin{split}
    			\xymatrix{
    				\txt{admissible averaging Zinbiel\\ algebras}
    				\ar@{=>}[r]^-{{\rm Prop.}~\ref{pro:2.28}
    				}
    				\ar@{=>}[d]^-{{\rm Prop.}~\ref{pro:3.34}
    				}
    				& \txt{triangular averaging commutative  and\\ cocommutative  
    					infinitesimal bialgebras
    				}
    				\ar@{=>}[d]^-{{\rm Prop.}~\ref{pro:quasi bialgebras}
    				}
    				\\  
    				\txt{pre-special apre-perm algebras}
    				\ar@{=>}[r]^-{{\rm Prop.}~\ref{pro:3.35} 
    				}
    				& \txt{triangular special apre-perm bialgebras}
    			}
    		\end{split}
    	\end{equation*}
    }
\subsection{Factorizable \sappbs}\label{sec5.2}\

Let $(A,\triangleright_{A},\triangleleft_{A},\vartheta_{r},\theta_{r})$ be a quasi-triangular \sappb. \delete{
	Recall that the operator $\big(r+\tau(r)\big)^{\sharp}$ is defined by
	\begin{equation}\label{factorizable operator}
		\big(r+\tau(r)\big)^{\sharp}=r^{\sharp}+{\tau(r)}^{\sharp}:A^{*}\rightarrow A.
\end{equation}}If $r+\tau(r)=0$, then $r$ is skew-symmetric and  $(A,\triangleright_{A},\triangleleft_{A},\vartheta_{r},\theta_{r})$ is triangular.
Factorizable \sappbs are however concerned with the opposite case that $\big(r+\tau(r)\big)^{\sharp}$ is bijective.

\begin{defi}
	A quasi-triangular \sappb $(A,\triangleright_{A},\triangleleft_{A},\vartheta_{r},\theta_{r})$ is called \textbf{factorizable} if the linear map $\big(r+\tau(r)\big)^{\sharp}:A^{*}\rightarrow A$  is a bijection.
\end{defi}

\begin{rmk}\label{rmk:7.8}
	In the definition of a factorizable \sappb $(A,\triangleright_{A},\triangleleft_{A},\vartheta_{r},$
$\theta_{r})$,
the condition \eqref{eq:pro:co1} is superfluous, since by Proposition \ref{pro:2.24}, it can be indicated from the nondegeneracy of $\big(r+\tau(r)\big)^{\sharp}$ and \eqref{eq:invariance}.
\end{rmk}


\delete{
	As a subclass of quasi-triangular \sappbs, factorizable \sappbs also appear in pairs.
	
	\begin{cor}\label{cor:fact pair}
		If $(A,\circ_{A},\star_{A},\theta_{r},\delta_{r})$ is a factorizable perm symmetric bialgebra, then $(A,\circ_{A}$,
		$\star_{A},\theta_{-\tau(r)},\delta_{-\tau(r)})$ is also a factorizable perm symmetric bialgebra.
	\end{cor}
	\begin{proof}
		If $\big(r+\tau(r)\big)^{\sharp}$ is a linear isomorphism, then $\Big(-\tau(r)+\tau\big(-\tau(r)\big)\Big)^{\sharp}=-\big(r+\tau(r)\big)^{\sharp}$ is also a linear isomorphism. Hence the conclusion follows from Proposition \ref{pro:4,2}.
\end{proof}}


\begin{defi}\cite{Bai2024}
Let $(A,\triangleright_{A},\triangleleft_{A})$ and
        $(A^{*},\triangleright_{A^{*}},\triangleleft_{A^{*}})$ be \sapps. If there is a
        quadratic \sapp structure $(A\oplus
        A^{*},\triangleright_{d},\triangleleft_{d},\mathcal{B}_{d})$ on $A\oplus A^{*}$
        such that $(A\oplus
        A^{*},\triangleright_{d},\triangleleft_{d} )$ contains $(A,\triangleright_{A},\triangleleft_{A})$ and $(
        A^{*},\triangleright_{A^{*}},\triangleleft_{A^{*}})$ as \sappsubs, then we say
        $\big(  (  A\oplus
        A^{*},\triangleright_{d},\triangleleft_{d},\mathcal{B}_{d}),(A,\triangleright_{A},\triangleleft_{A}),(A^{*},\triangleright_{A^{*}},\triangleleft_{A^{*}})\big)
        $ is a \textbf{Manin triple of \sapps}.
\end{defi}
\delete{
\begin{thm}\label{thm:Manin triple}
	Let $(A, \triangleright_{A}, \triangleleft_{A})$ and $(A^{*}, \triangleright_{A^{*}}, \triangleleft_{A^{*}})$ be SDPP algebras.
Then there is a  Manin triple of SDPP algebras $\big((A\oplus A^{*}, \triangleright_{d}, \triangleleft_{d},\mathcal{B}_{d}),(A, \triangleright_{A}, \triangleleft_{A}),(A^{*}, \triangleright_{A^{*}}, \triangleleft_{A^{*}})\big) $ if and only if there is a perm algebra structure $(A\oplus A^{*},\circ_{d})$ on $A\oplus A^{*}$ given by
\begin{equation}\label{eq:A ds}	(x+a^{*})\circ_{d}(y+b^{*})=x\circ_{A}y+\mathcal{L}^{*}_{\circ_{A^{*}}}(a^{*})y+\mathcal{L}^{*}_{ \triangleleft_{A^{*}}}(b^{*})x+a^{*}\circ_{A^{*}}b^{*}+
	\mathcal{L}^{*}_{\circ_{A}}(x)b^{*}+\mathcal{L}^{*}_{ \triangleleft_{A}}(y)a^{*},
\end{equation}
for all $x,y\in A, a^{*},b^{*}\in A^{*}$.
Moreover, in this case the SDPP algebra structure $(A\oplus A^{*}, \triangleright_{d},$
$ \triangleleft_{d})$ on $A\oplus A^{*}$ is given by
\begin{eqnarray}
(x+a^{*}) \triangleright_{d}(y+b^{*})&=&\;x \triangleright_{A}y+(\mathcal{L}^{*}_{ \triangleright_{A^{*}}}-\mathcal{R}^{*}_{ \triangleright_{A^{*}}})(a^{*})y
-\mathcal{R}^{*}_{ \triangleright_{A^{*}}}(b^{*})x\nonumber\\
&&\;+a^{*} \triangleright_{A^{*}}b^{*}+(\mathcal{L}^{*}_{ \triangleright_{A}}-\mathcal{R}^{*}_{ \triangleright_{A}})(x)b^{*}-\mathcal{R}^{*}_{ \triangleright_{A}}(y)a^{*},\label{eq:A ds1}\\
(x+a^{*}) \triangleleft_{d}(y+b^{*})&=&\;x \triangleleft_{A}y+\mathcal{R}^{*}_{\circ_{A^{*}}}(a^{*})y
+\mathcal{R}^{*}_{\circ_{A^{*}}}(b^{*})x\nonumber\\
&&\;+a^{*} \triangleleft_{A^{*}}b^{*}+\mathcal{R}^{*}_{\circ_{A }}(x)b^{*}
+\mathcal{R}^{*}_{\circ_{A }}(y)a^{*}.\label{eq:A ds2}
\end{eqnarray}
\end{thm}}

\begin{lem}\cite{Bai2024}\label{3039}
	Let
	$(A,\triangleright_{A},\triangleleft_{A})$ and
	$(A^{*},\triangleright_{A^{*}},\triangleleft_{A^{*}})$ be \sapps.
Then  there is a Manin triple of SDPP algebras $\big(  (  A\oplus
A^{*},\triangleright_{d},\triangleleft_{d},\mathcal{B}_{d}),(A,\triangleright_{A},\triangleleft_{A}),(A^{*},\triangleright_{A^{*}},\triangleleft_{A^{*}})\big)$  if and only if $(A,\triangleright_{A},\triangleleft_{A},\vartheta ,\theta )$ is a \sappb, where $\vartheta$ and $\theta$ are the linear duals of $\triangleright_{A^{*}}$ and $\triangleleft_{A^{*}}$ respectively. Moreover, in this case we have
\begin{eqnarray}
	(x+a^{*})\triangleright_{d}(y+b^{*})&=&\;x\triangleright_{A}y+(\mathcal{L}^{*}_{\circ_{A^{*}}}+\mathcal{R}^{*}_{\circ_{A^{*}}})(a^{*})y
	+\mathcal{R}^{*}_{\triangleright_{A^{*}}}(b^{*})x\nonumber\\
	&&\;+a^{*}\triangleright_{A^{*}}b^{*}+(\mathcal{L}^{*}_{\circ_{A}}+\mathcal{R}^{*}_{\circ_{A}})(x)b^{*}+\mathcal{R}^{*}_{\triangleright_{A}}(y)a^{*},\label{3048}\\
	(x+a^{*})\triangleleft_{d}(y+b^{*})&=&\;x\triangleleft_{A}y-\mathcal{R}^{*}_{\circ_{A^{*}}}(a^{*})y
	-\mathcal{R}^{*}_{\circ_{A^{*}}}(b^{*})x\nonumber\\
	&&\;+a^{*}\triangleleft_{A^{*}}b^{*}-\mathcal{R}^{*}_{\circ_{A }}(x)b^{*}
	-\mathcal{R}^{*}_{\circ_{A }}(y)a^{*},\label{3052}
\end{eqnarray}
for all $x,y\in A, a^{*},b^{*}\in A^{*}$.
\end{lem}

\delete{
Let $(A,\triangleright_{A},\triangleleft_{A})$ be a \sapp and $r\in A\otimes A$. Define a linear map $r^{\sharp}\oplus \big(-\tau(r)\big)^{\sharp}:A^{*}\rightarrow A\oplus A$ by
\begin{equation*}
	r^{\sharp}\oplus \big(-\tau(r)\big)^{\sharp}(a^{*})=\Big(r^{\sharp}(a^{*}),\big(-\tau(r)\big)^{\sharp}(a^{*})\Big),\;\forall a^{*}\in A^{*}.
\end{equation*}
Then}
We have the following proposition which justifies the terminology of factorizable \sappbs.

\begin{pro}\label{3038}
	Let $(A,\triangleright_{A},\triangleleft_{A},\vartheta_{r},\theta_{r})$ be a factorizable \sappb, and the corresponding Manin triple be denoted by $\big((D=A\oplus A^{*},\triangleright_{d},\triangleleft_{d},\mathcal{B}_{d}),
(A,\triangleright_{A},\triangleleft_{A}),$
$(A^{*},\triangleright_{r},\triangleleft_{r})\big)$. 
Define a linear map $\psi:D=A\oplus A^{*}\rightarrow A\oplus A$  by
\eqref{eq:factorizable}.
\delete{
\begin{equation*}
	\psi(x)=(x,x), \;\psi(a^{*})=\Big(r^{\sharp}(a^{*}),\big(-\tau(r)\big)^{\sharp}(a^{*})\Big),\;\forall x\in A, a^{*}\in A^{*}.
\end{equation*}}Then $\psi$ gives the \sapp isomorphism between  $(D,\triangleright_{d},\triangleleft_{d})$ and the direct sum $A\oplus A$ of \sapps. In particular, $\psi|_{A^{*}}$ gives  the \sapp isomorphism between $(A^{*},\triangleright_{r},\triangleleft_{r})$   and $\mathrm{Im}(r^{\sharp}\oplus \big(-\tau(r)\big)^{\sharp})$ as a special apre-perm subalgebra of $A\oplus A$. Moreover, for any $x\in A$, there is a unique decomposition
	$x=x_{1}-x_{2},$
	where $(x_{1},x_{2})\in\mathrm{Im}(r^{\sharp}\oplus \big(-\tau(r)\big)^{\sharp}).$
\end{pro}
\begin{proof}
	Since $\big(r+\tau(r)\big)^{\sharp}$ is a bijection, we have $\mathrm{Ker}(\psi|_{A^{*}})=0$.
	By Theorem \ref{thm:4.6}, we have
	\begin{eqnarray*}
		\psi(a^{*})\triangleright\psi(b^{*})&=&\Big( r^{\sharp}(a^{*})\triangleright_{A}r^{\sharp}(b^{*}),\big(-\tau(r)\big)^{\sharp}(a^{*})\triangleright_{A}\big(-\tau(r)\big)^{\sharp}(b^{*}) \Big)\\
		&=&\Big( r^{\sharp}(a^{*} \triangleright_{r} b^{*}),\big(-\tau(r)\big)^{\sharp}(a^{*} \triangleright_{r} b^{*}) \Big)\\
		&=&\psi(a^{*}\triangleright_{r}b^{*}),\;\forall a^{*},b^{*}\in A^{*}.
	\end{eqnarray*}
Similarly, we also have
\begin{eqnarray*}
\psi(a^{*})\triangleleft\psi(b^{*})=\psi(a^{*}\triangleleft_{r}b^{*}),\;\forall a^{*},b^{*}\in A^{*}.
\end{eqnarray*}
Then $\psi|_{A^{*}}$ is a homomorphism of \sapps. Therefore, $\mathrm{Im}(r^{\sharp}\oplus \big(-\tau(r)\big)^{\sharp})$ is isomorphic to $(A^{*},\triangleright_{r},\triangleleft_{r})$ as special apre-perm subalgebras. For all $x\in A, a^{*},b^{*}\in A^{*}$, we have
{\small
	\begin{align*}
		&\langle r^{\sharp}\big((\mathcal{L}^{*}_{\circ_{A}}+\mathcal{R}^{*}_{\circ_{A}})(x)a^{*}\big)
+\mathcal{R}^{*}_{\triangleright_{r}}(a^{*})x,b^{*}\rangle\\
		&=\langle r, (\mathcal{L}^{*}_{\circ_{A}}+\mathcal{R}^{*}_{\circ_{A}})(x)a^{*}\otimes b^{*}\rangle+\langle x, b^{*}\triangleright_{r}a^{*}\rangle\\
		&=\langle r, (\mathcal{L}^{*}_{\circ_{A}}+\mathcal{R}^{*}_{\circ_{A}})(x)a^{*}\otimes b^{*}\rangle+\langle \vartheta_{r}(x), b^{*}\otimes a^{*}\rangle\\
		&=\langle \big((\mathcal{L}_{\circ_{A}}+\mathcal{R}_{\circ_{A}})(x)\otimes\mathrm{id}\big)r, a^{*}\otimes b^{*}\rangle-\langle \big(-\mathcal{L}_{\triangleright_{A}}(x)\otimes\mathrm{id}+\mathrm{id}\otimes(\mathcal{L}_{\circ_{A}}+\mathcal{R}_{\circ_{A}})(x)\big)\tau (r),b^{*}\otimes a^{*}\rangle\\
		&=\langle \big((\mathcal{L}_{\circ_{A}}+\mathcal{R}_{\circ_{A}})(x)\otimes\mathrm{id}+\mathrm{id}\otimes\mathcal{L}_{\triangleright_{A}}(x)	-(\mathcal{L}_{\circ_{A}}+\mathcal{R}_{\circ_{A}})(x)\otimes\mathrm{id}\big)r, a^{*}\otimes b^{*}\rangle\\
		&=\langle \big(\mathrm{id}\otimes\mathcal{L}_{\triangleright_{A}}(x)\big)r,a^{*}\otimes b^{*}\rangle\\
		&=\langle x\triangleright_{A}r^{\sharp}(a^{*}), b^{*}\rangle,
	\end{align*}}that is,
	\begin{equation}\label{eq:106}
		r^{\sharp}\big((\mathcal{L}^{*}_{\circ_{A}}+\mathcal{R}^{*}_{\circ_{A}})(x)a^{*}\big)
+\mathcal{R}^{*}_{\triangleright_{r}}(a^{*})x= x\triangleright_{A}r^{\sharp}(a^{*}).
	\end{equation}
	Similarly we have
	\begin{equation}\label{eq:107}	\big(-\tau(r)\big)^{\sharp}\big((\mathcal{L}^{*}_{\circ_{A}}+\mathcal{R}^{*}_{\circ_{A}})(x)a^{*}\big)
+\mathcal{R}^{*}_{\triangleright_{r}}(a^{*})x=x\triangleright_{A}\big(-\tau(r)\big)^{\sharp}(a^{*}).
	\end{equation}
	Thus we have
{\small
		\begin{align*}
			\psi(x\triangleright_{D} a^{*})&=\psi\big((\mathcal{L}^{*}_{\circ_{A}}+\mathcal{R}^{*}_{\circ_{A}})(x)a^{*}+\mathcal{R}^{*}_{\triangleright_{r}}(a^{*})x\big)\\	&=\Big(r^{\sharp}\big((\mathcal{L}^{*}_{\circ_{A}}+\mathcal{R}^{*}_{\circ_{A}})(x)a^{*}\big)+\mathcal{R}^{*}_{\triangleright_{r}}(a^{*})x,
\big(-\tau(r)\big)^{\sharp}\big((\mathcal{L}^{*}_{\circ_{A}}+\mathcal{R}^{*}_{\circ_{A}})(x)a^{*}\big)+\mathcal{R}^{*}_{\triangleright_{r}}(a^{*})x\Big)\\ &\overset{\eqref{eq:106},\eqref{eq:107}}{=}\big(x\triangleright_{A}r^{\sharp}(a^{*}), x\triangleright_{A}\big(-\tau(r)\big)^{\sharp}(a^{*})\big)\\
			&=\psi(x)\triangleright \psi(a^{*}),
		\end{align*}}and similarly
	\begin{equation*} \psi(a^{*}\triangleright_{D}x)=\psi(a^{*})\triangleright\psi(x), \;\psi(x\triangleleft_{D}a^{*})=\psi(x)\triangleleft\psi(a^{*}).
	\end{equation*}
	In conclusion, $\psi:D\rightarrow A\oplus A$ is a homomorphism of \sapps and is clearly bijective, and hence $\psi$ is a \sapp isomorphism. Moreover, 
	any element $x\in A$ can be expressed as
	\begin{small}
		\begin{equation*} x=\big(r+\tau(r)\big)^{\sharp}{\big(r+\tau(r)\big)^{\sharp}}^{-1}(x)=r^{\sharp}{\big(r+\tau(r)\big)^{\sharp}}^{-1}(x)-\big(-\tau(r)\big)^{\sharp}{\big(r+\tau(r)\big)^{\sharp}}^{-1}(x)=x_{1}-x_{2},
		\end{equation*}
	\end{small}where $x_{1}=r^{\sharp}{\big(r+\tau(r)\big)^{\sharp}}^{-1}(x)\in\mathrm{Im}  r^{\sharp} ,\; x_{2}=\big(-\tau(r)\big)^{\sharp}{\big(r+\tau(r)\big)^{\sharp}}^{-1}(x)\in\mathrm{Im} \big(-\tau(r)\big)^{\sharp} $.
	Since $\big(r+\tau(r)\big)^{\sharp}$ is a bijection, the decomposition is unique.
	Hence the proof is finished.
\end{proof}

\delete{
	\begin{proof}
		Define a linear map $\theta:D=A\oplus A^{*}\rightarrow A\oplus A$  by
		\begin{equation*}
			\theta(x)=(x,x), \;\theta(a^{*})=\Big(r^{\sharp}(a^{*}),\big(-\tau(r)\big)^{\sharp}(a^{*})\Big),\;\forall x\in A, a^{*}\in A^{*}.
		\end{equation*}
		Since $\big(r+\tau(r)\big)^{\sharp}$ is a linear isomorphism, $\theta$ is bijective.
		It is clear that $\theta|_{A}$ is an embedding of perm symmetric algebras, and by Theorem \ref{thm:4.6}, $\theta|_{A^{*}}$ is a homomorphism of perm symmetric algebras. Therefore, $\mathrm{Im}\Big(r^{\sharp}\oplus \big(-\tau(r)\big)^{\sharp}\Big)$ is isomorphic to $(A^{*},\circ_{r},\star_{r})$ as perm symmetric algebras. For all $x\in A, a^{*}\in A^{*}$, we have
		\begin{eqnarray*}
			\theta(x\circ_{D} a^{*})&=&\theta\big(\mathcal{L}^{*}_{\star_{A^{*}}}(a^{*})x+\mathcal{L}^{*}_{\circ_{A}}(x)a^{*}\big)\\	&=&\Big(r^{\sharp}\big(\mathcal{L}^{*}_{\circ_{A}}(x)a^{*}\big)+\mathcal{L}^{*}_{\star_{A^{*}}}(a^{*})x,\big(-\tau(r)\big)^{\sharp}\big(\mathcal{L}^{*}_{\circ_{A}}(x)a^{*}\big)+\mathcal{L}^{*}_{\star_{A^{*}}}(a^{*})x\Big),\\
			\theta(x\star_{D} a^{*})&=&\theta\big(\mathcal{R}^{*}_{\circ_{A^{*}}}(a^{*})x+\mathcal{R}^{*}_{\circ_{A}}(x)a^{*}\big)\\
			&=&\Big(r^{\sharp}\big(\mathcal{R}^{*}_{\circ_{A}}(x)a^{*}\big)+\mathcal{R}^{*}_{\circ_{A^{*}}}(a^{*})x,\big(-\tau(r)\big)^{\sharp}\big(\mathcal{R}^{*}_{\circ_{A}}(x)a^{*}\big)+\mathcal{R}^{*}_{\circ_{A^{*}}}(a^{*})x\Big).
		\end{eqnarray*}
		Thus for all $b^{*}\in A^{*}$, we have
		\begin{eqnarray*}
			\langle r^{\sharp}\big(\mathcal{L}^{*}_{\circ_{A}}(x)a^{*}\big),b^{*}\rangle&=&\langle r,\mathcal{L}^{*}_{\circ_{A}}(x)a^{*}\otimes b^{*}\rangle=\langle \big(\mathcal{L}_{\circ_{A}}(x)\otimes\mathrm{id}\big)r, a^{*}\otimes b^{*}\rangle,\\
			\langle\mathcal{L}^{*}_{\star_{A^{*}}}(a^{*})x,b^{*}\rangle&=&\langle x,a^{*}\star _{A^{*}}b^{*}\rangle=\langle\delta_{r}(x),a^{*}\otimes b^{*}\rangle\\
			&=&\langle \big(\mathrm{id}\otimes\mathcal{L}_{\circ_{A}}(x)-\mathcal{L}_{\circ_{A}}(x)\otimes\mathrm{id}\big)r, a^{*}\otimes b^{*}\rangle.
		\end{eqnarray*}
		Hence we have
		\begin{eqnarray*}
			\langle r^{\sharp}\big(\mathcal{L}^{*}_{\circ_{A}}(x)a^{*}\big)+\mathcal{L}^{*}_{\star_{A^{*}}}(a^{*})x, b^{*}\rangle&=&\langle\big(\mathrm{id}\otimes\mathcal{L}_{\circ_{A}}(x)\big)r,a^{*}\otimes b^{*})\rangle=\langle r,a^{*}\otimes\mathcal{L}^{*}_{\circ_{A}}(x)b^{*}\rangle\\
			&=&\langle r^{\sharp}(a^{*}), \mathcal{L}^{*}_{\circ_{A}}(x)b^{*}\rangle=\langle x\circ_{A} r^{\sharp}(a^{*}), b^{*}\rangle.
		\end{eqnarray*}
		That is,
		\begin{equation*}
			r^{\sharp}\big(\mathcal{L}^{*}_{\circ_{A}}(x)a^{*}\big)+\mathcal{L}^{*}_{\star_{A^{*}}}(a^{*})x=x\circ_{A}r^{\sharp}(a^{*}).
		\end{equation*}
		Similarly, we also obtain
		\begin{equation*}
			r^{\sharp}\big(\mathcal{R}^{*}_{\circ_{A}}(x)a^{*}\big)+\mathcal{R}^{*}_{\circ_{A^{*}}}(a^{*})x=x\star_{A}r^{\sharp}(a^{*}).
		\end{equation*}
		In the same way, we have
		\begin{small}
			\begin{eqnarray*}	&&\big(-\tau(r)\big)^{\sharp}\big(\mathcal{L}^{*}_{\circ_{A}}(x)a^{*}\big)+\mathcal{L}^{*}_{\star_{A^{*}}}(a^{*})x=x\circ_{A}\big(-\tau(r)\big)^{\sharp}(a^{*}),\\
				&&\big(-\tau(r)\big)^{\sharp}\big(\mathcal{R}^{*}_{\circ_{A}}(x)a^{*}\big)+\mathcal{R}^{*}_{\circ_{A^{*}}}(a^{*})x=x\star_{A}\big(-\tau(r)\big)^{\sharp}(a^{*}).
			\end{eqnarray*}
		\end{small}Thus we have
		\begin{eqnarray*}
			&&\theta(x\circ_{D}a^{*})=\big(x\circ_{A}r^{\sharp}(a^{*}), x\circ_{A}\big(-\tau(r)\big)^{\sharp}(a^{*})\big)=\theta(x)\circ \theta(a^{*}),\\
			&&\theta(x\star_{D}a^{*})=\big(x\star_{A}r^{\sharp}(a^{*}), x\star_{A}\big(-\tau(r)\big)^{\sharp}(a^{*})\big)=\theta(x)\star \theta(a^{*}),
		\end{eqnarray*}
		and similarly
		\begin{equation*}
			\theta(a^{*}\circ_{D}x)=\theta(a^{*})\circ\theta(x).
		\end{equation*}
		In conclusion, $\theta:D\rightarrow A\oplus A$ is a homomorphism of perm symmetric algebras and is bijective, and hence $\theta$ is a perm symmetric algebra isomorphism. Moreover, 
		any element $x\in A$ can be uniquely expressed as
		\begin{small}
			\begin{equation*}
				x=\big(r+\tau(r)\big)^{\sharp}{\big(r+\tau(r)\big)^{\sharp}}^{-1}(x)=r^{\sharp}{\big(r+\tau(r)\big)^{\sharp}}^{-1}(x)-\big(-\tau(r)\big)^{\sharp}{\big(r+\tau(r)\big)^{\sharp}}^{-1}(x)=x_{1}-x_{2},
			\end{equation*}
		\end{small}where $x_{1}=r^{\sharp}{\big(r+\tau(r)\big)^{\sharp}}^{-1}(x)\in\mathrm{Im}(r^{\sharp}),\; x_{2}=\big(-\tau(r)\big)^{\sharp}{\big(r+\tau(r)\big)^{\sharp}}^{-1}(x)\in\mathrm{Im}\Big(\big(-\tau(r)\big)^{\sharp}\Big)$.
		Hence the conclusion follows.
\end{proof}}

In the following, we show that there is a factorizable \sappb structure on an arbitrary Manin triple of \sapps.

\begin{thm}\label{3172}
	Let $\big((D=A\oplus A^{*},\triangleright_{d},\triangleleft_{d},\mathcal{B}_{d}),
(A,\triangleright_{A},\triangleleft_{A}),(A^{*},\triangleright_{A^{*}},\triangleleft_{A^{*}})\big) $ be a Manin triple of \sapps. Let
	$\{e_{1}$, $\cdots$, $e_{n}\}$ be a basis of $A$,  $\{e^{*}_{1},\cdots,e^{*}_{n}\}$ be the dual basis and $r$ be given by \eqref{1306}. Then $(D,\triangleright_{d},\triangleleft_{d},\vartheta_{r},\theta_{r})$ with $\theta_r$ and $\delta_r$ defined by \eqref{eq:comul} is a factorizable \sappb.
\end{thm}

\begin{proof}
	We first observe that
	\begin{equation*}
		\big( r+\tau(r)\big)^{\sharp}(x+a^{*})=x+a^{*},\;\forall x\in A, a^{*}\in A^{*}.
	\end{equation*}
	Hence $\big(r+\tau(r)\big)^{\sharp}$ is a linear isomorphism.
By  \eqref{3048} and \eqref{3052}, we have
	\begin{eqnarray*}
		f(x)\big(r+\tau(r)\big)&=&\sum_{i}e^{*}_{i}\otimes e_{i}\circ_{A}x+e_{i}\otimes e^{*}_{i}\circ_{d}x+x\triangleleft_{d}e^{*}_{i}\otimes e_{i}+x\triangleleft_{A}e_{i}\otimes e^{*}_{i}\\
		&=&\sum_{i}e^{*}_{i}\otimes e_{i}\circ_{A}x+e_{i}\otimes\mathcal{L}^{*}_{\circ_{A^{*}}}(e^{*}_{i})x
-e_{i}\otimes\mathcal{L}^{*}_{\triangleleft_{A}}(x)e^{*}_{i}
-\mathcal{R}^{*}_{\circ_{A^{*}}}(e^{*}_{i})x\otimes e_{i}\\
		&&-\mathcal{R}^{*}_{\circ_{A}}(x)e^{*}_{i}\otimes e_{i}+x\triangleleft_{A}e_{i}\otimes e^{*}_{i},
	\end{eqnarray*}
	for all $x\in A$.
	Observing that
	\begin{eqnarray*}		&&\sum_{i}e_{i}\otimes\mathcal{L}^{*}_{\circ_{A^{*}}}(e^{*}_{i})x-\mathcal{R}^{*}_{\circ_{A^{*}}}(e^{*}_{i})x\otimes e_{i}=0,\;\sum_{i}x\triangleleft_{A}e_{i}\otimes e^{*}_{i}-e_{i}\otimes\mathcal{L}^{*}_{\triangleleft_{A}}(x)e^{*}_{i}=0,\\
		&&\sum_{i}e^{*}_{i}\otimes e_{i}\circ_{A}x-\mathcal{R}^{*}_{\circ_{A}}(x)e^{*}_{i}\otimes e_{i}=0,
	\end{eqnarray*}
	we finally get $f(x)\big(r+\tau(r)\big)= 0$. By duality, we also have $f(a^{*})\big(r+\tau(r)\big)= 0$, for all $a^{*}\in A^{*}$. By Remark \ref{rmk:7.8}, $r+\tau(r)$ is invariant on $(D,\triangleright_{d},\triangleleft_{d})$. Furthermore,
		\begin{eqnarray*}		SA(r)&=&\sum_{i,j}e^{*}_{i}\circ_{A^{*}}e^{*}_{j}\otimes e_{i}\otimes e_{j}+e^{*}_{i}\otimes e_{i}\triangleleft_{d}e^{*}_{j}\otimes e_{j}+e^{*}_{j}\otimes e^{*}_{i}\otimes e_{i}\circ_{A} e_{j} \\
			&=&\sum_{i,j}e^{*}_{i}\circ_{A^{*}}e^{*}_{j}\otimes e_{i}\otimes e_{j}-e^{*}_{i}\otimes\mathcal{R}^{*}_{\circ_{A}}(e_{i})e^{*}_{j}\otimes e_{j}\\
			&&-e^{*}_{i}\otimes \mathcal{R}^{*}_{\circ_{A^{*}}}(e^{*}_{j})e_{i}\otimes e_{j}+e^{*}_{j}\otimes e^{*}_{i}\otimes e_{i}\circ_{A} e_{j}\\
			&=&0,
		\end{eqnarray*}
that is, $r$ is a solution of the SAPP-YBE in $(D,\triangleright_{d},\triangleleft_{d})$. Hence $(D,\triangleright_{d},\triangleleft_{d},\vartheta_{r},\theta_{r})$ is a factorizable \sappb.
\end{proof}

\begin{ex}\label{ex:6.29}
	Let $(A\ltimes_{\mathcal{L}^{*}_{\cdot_{A}}}A^{*},\Delta_{r},P,\hat{P})$ be the factorizable averaging commutative and cocommutative infinitesimal bialgebra given by Example \ref{ex:3.25}.
	Then by Proposition \ref{pro:quasi bialgebras}, there is a factorizable \sappb $(D=A\oplus A^{*},\triangleright_{d},\triangleleft_{d},\vartheta_{r},\theta_{r})$ with multiplications and co-multiplications given by \eqref{eq:com asso and SDPP} and \eqref{eq:co re} respectively.
	Explicitly, the non-zero multiplications and co-multiplications on $D$ are given by
{\small
		\begin{align}
			&e_{1}\triangleright_{A}e_{1}=e_{1},\; e_{1}\triangleright_{A}e_{2}=e_{2}\triangleright_{A}e_{1}=e_{2},
\;e_{1}\triangleright_{d}e^{*}_{2}=e^{*}_{2}\triangleright_{d}e_{1}
=2e^{*}_{2},\label{eq:ex:fact0}\\
&e_{1}\triangleright_{d}e^{*}_{1}=e_{2}\triangleright_{d}e^{*}_{2}
=e^{*}_{1}\triangleright_{d}e_{1}=e^{*}_{2}\triangleright_{d}e_{2}
=2e^{*}_{1},\;e_{1}\triangleleft_{d}e^{*}_{2}=-e^{*}_{2},\; e_{1}\triangleleft_{d}e^{*}_{1}=e_{2}\triangleleft_{d}e^{*}_{2}=-e^{*}_{1},\label{eq:ex:fact1}\\
&\vartheta_{r}(e^{*}_{1})=e^{*}_{1}\otimes e^{*}_{1},\;  \vartheta_{r}(e^{*}_{2})=e^{*}_{1}\otimes e^{*}_{2}+e^{*}_{2}\otimes e^{*}_{1}.\label{eq:ex:fact2}
	\end{align}}
\end{ex}

\delete{
	\begin{ex}\label{ex:5.14}
		Consider the 4-dimensional perm symmetric algebra $(A,\circ_{A},\star_{A})$ defined with respect to a basis $\{ e_{1},e_{2},e_{3},e_{4}\}$ by the following nonzero products:
		\begin{eqnarray*}
			&&e_{1}\circ e_{1}=e_{1},\; e_{1}\circ e_{2}=e_{2}\circ e_{1}=e_{2},\; e_{1}\circ e_{4}=e_{1}\star e_{4}=e_{4},\\
			&&e_{1}\circ e_{3}=e_{2}\circ e_{4}=e_{1}\star e_{3}=e_{2}\star e_{4}=e_{3}.
		\end{eqnarray*}
		Then it is straightforward to check
		\begin{equation*}
			r=e_{3}\otimes e_{1}+e_{4}\otimes e_{2}
		\end{equation*}
		is a solution of the PSYBE in $(A,\circ_{A},\star_{A})$ and  $r+\tau(r)$ is invariant,
		and thus there is a quasi-triangular perm symmetric bialgebra $(A,\circ_{A},\star_{A},\theta_{r},\delta_{r})$ with $\theta_{r}$ and $\delta_{r}$ defined by \eqref{eq:comul}. Moreover we have
		\begin{equation*}
			r+\tau(r)=e_{1}\otimes e_{3}+e_{2}\otimes e_{4}+e_{3}\otimes e_{1}+e_{4}\otimes e_{2},
		\end{equation*}
		which gives
		\begin{equation*}
			\big(r+\tau(r)\big)^{\sharp}(e^{*}_{1})=e_{3},\;
			\big(r+\tau(r)\big)^{\sharp}(e^{*}_{2})=e_{4},\;
			\big(r+\tau(r)\big)^{\sharp}(e^{*}_{3})=e_{1},\;
			\big(r+\tau(r)\big)^{\sharp}(e^{*}_{4})=e_{2}.
		\end{equation*}
		Hence $\big(r+\tau(r)\big)^{\sharp}$ is nondegenerate and  $(A,\circ_{A},\star_{A},\theta_{r},\delta_{r})$ is factorizable.
\end{ex}}

Next we establish a one-to-one correspondence between quadratic Rota-Baxter \sapps of weight $-1$ and factorizable \sappbs.

\begin{thm}\label{thm:quadratic to fact}
	Let $(A,\triangleright_{A},\triangleleft_{A},R,\mathcal{B})$ be a quadratic Rota-Baxter \sapp of weight $-1$.
	Then there is a factorizable \sappb $(A,\triangleright_{A},\triangleleft_{A},\vartheta_{r},\theta_{r})$ with $r$ given through the operator form $r^{\sharp}$ by \eqref{eq:thm:quadratic to fact}.
	Conversely, let $(A,\triangleright_{A},\triangleleft_{A},\vartheta_{r},\theta_{r})$ be a factorizable \sappb.
	Then there is a quadratic Rota-Baxter \sapp $(A,\triangleright_{A},\triangleleft_{A},R,\mathcal{B})$ of weight $-1$ with $R$ given by  \eqref{eq:fact to quadratic0}
	and $\mathcal{B}$ given by
	\eqref{eq:fact to quadratic}.
\end{thm}

\begin{proof}
	Let $(A,\triangleright_{A},\triangleleft_{A},R,\mathcal{B})$ be a quadratic Rota-Baxter \sapp of weight $-1$.
	Then by Lemma \ref{lem:bf}, we have
	\begin{equation}\label{rw2}
		r+\tau(r)=\phi_{\mathcal{B}}.
	\end{equation}
	Since $\phi_{\mathcal{B}}\in A\otimes A$ is invariant and $\mathcal{B}^ \natural:A\rightarrow A^{*}$ is a linear isomorphism, we see that $r+\tau(r)$ is invariant on $(A,\triangleright_{A},\triangleleft_{A})$ and $\big(r+\tau(r)\big)^{\sharp}$ is a bijection. Moreover, since $R$ is a Rota-Baxter operator of weight $-1$, \eqref{eq:pro3,1} and \eqref{eq:pro3,2} hold in Proposition \ref{pro1-6} such that $SA(r)=0$.  Hence $(A,\triangleright_{A},\triangleleft_{A},\vartheta_{r},\theta_{r})$ is a factorizable \sappb with $\theta_{r}$ and $\delta_{r}$ defined by \eqref{eq:comul}.
	
	Conversely, let $(A,\triangleright_{A},\triangleleft_{A},\vartheta_{r},\theta_{r})$ be a factorizable \sappb.
	We first observe that \eqref{eq:fact to quadratic} equivalently gives \eqref{rw2} and
	\begin{equation}\label{eq:113}
	\big(r+\tau(r)\big)^{\sharp}=\mathcal{B}^{{\natural}^{-1}}.
	\end{equation}
	Since $r+\tau(r)$ is invariant and $\big(r+\tau(r)\big)^{\sharp}$ is a linear isomorphism, it follows from Proposition  \ref{pro:2.24} that $\mathcal{B}$ given by \eqref{eq:fact to quadratic} contributes a quadratic \sapp $(A,\triangleright_{A},\triangleleft_{A},\mathcal{B})$. Taking \eqref{eq:113} into \eqref{eq:pro3,1} and \eqref{eq:pro3,2} in Proposition \ref{pro1-6}, we see that $R$ given by \eqref{Br,Brt} is a Rota-Baxter operator of weight $-1$.
	Furthermore, \eqref{strong} holds for $\lambda=-1$ by Lemma \ref{lem:bf}. Therefore $(A,\triangleright_{A},\triangleleft_{A},R,\mathcal{B})$ is a quadratic Rota-Baxter \sapp of weight $-1$.
\end{proof}

\delete{ Let $\mathcal{B}$ be a bilinear form on $A$ given by \eqref{rw}, that is,
	\begin{equation}\label{eq:fact to quadratic}
		\mathcal{B}(x,y)=-\lambda\langle {\big(r+\tau(r)\big)^{\sharp}}^{-1}(x),y\rangle,\;\lambda\neq0,\;\forall x,y\in A,
	\end{equation}
	and $\beta$ be a linear map given by \eqref{Br,Brt}. By \eqref{eq:2-tensor} and \eqref{eq:fact to quadratic}, we have $\mathcal{B}^{\natural}=-\lambda {\big(r+\tau(r)\big)^{\sharp}}^{-1}$, and thus $\big(r+\tau(r)\big)^{\sharp}=-\lambda \mathcal{B}^{{\natural}^{-1}}$. Since $r$ satisfies \eqref{eq:pro:co1}, \eqref{eq:invariance} and $\big(r+\tau(r)\big)^{\sharp}$ is nondegenerate, it follows that $\mathcal{B}$ given by \eqref{eq:fact to quadratic} is a nondegenerate symmetric invariant bilinear form on $(A,\circ_{A},\star_{A})$. Moreover, taking $\big(r+\tau(r)\big)^{\sharp}=-\lambda \mathcal{B}^{{\natural}^{-1}}$ into \eqref{eq:pro3,1} and \eqref{eq:pro3,2} in Proposition \ref{pro1-6}, we see that $\beta$ is a Rota-Baxter operator of weight $\lambda$. Again from $\big(r+\tau(r)\big)^{\sharp}=-\lambda \mathcal{B}^{{\natural}^{-1}}$, \eqref{rw} holds, and by Lemma \ref{lem:bf}, \eqref{eq:lem:bf2} holds. Thus $(A,\circ_{A},\star_{A},\mathcal{B},\beta)$ is a quadratic Rota-Baxter perm symmetric algebra of weight $\lambda$.}

\delete{
	\begin{cor}\label{cor:4.20}
		Let $\beta$ be a Rota-Baxter operator of weight $\lambda\neq0$ on a perm symmetric algebra $(A,\circ_{A},\star_{A})$. Let $\{e_{1},\cdots,e_{n}\}$ be a basis of $A$ and $\{e^{*}_{1},\cdots,e^{*}_{n}\}$ be the dual basis.
		Then $(A\ltimes_{\mathcal{L}^{*}_{\circ_{A}},\mathcal{L}^{*}_{\star_{A}},\mathcal{R}^{*}_{\circ_{A}}} A^{*},\theta_{r},\delta_{r})$ is a factorizable perm symmetric bialgebra with $\theta_r$ and $\delta_{r}$ defined by \eqref{eq:comul}, where
		\begin{equation}\label{eq:cor:4.20}
			r=\sum_{i=1}^{n}-(\beta+\lambda)e_{i}\otimes e^{*}_{i}+e^{*}_{i}\otimes\beta(e_{i}).
		\end{equation}
	\end{cor}
	
	\begin{proof}
		By Proposition \ref{pro:4.2}, 	$\big(A\ltimes_{\mathcal{L}^{*}_{\circ_{A}},\mathcal{L}^{*}_{\star_{A}},\mathcal{R}^{*}_{\circ_{A}}} A^{*},\mathcal{B}_{d},\beta-( \beta+\lambda\mathrm{id}_{A} )^{*}\big)$ is a quadratic Rota-Baxter perm symmetric algebra of weight $\lambda$.
		For all $x\in A, a^{*}\in A^{*}$, we have
		\begin{equation*}
			\mathcal{B}_{d}^{\natural}(x+a^{*})=x+a^{*},
		\end{equation*}
		and following \eqref{eq:thm:quadratic to fact}, we obtain a linear transformation $r^{\sharp}$ on $A\oplus A^{*}$ by
		\begin{equation*}
			r^{\sharp}(x+a^{*})=\big(\beta-( \beta+\lambda\mathrm{id}_{A} )^{*}\big)\mathcal{B}_{d}^{{\natural}^{-1}}(x+a^{*})=\beta(x)-(\beta+\lambda\mathrm{id}_{A})^{*}a^{*}.
		\end{equation*}
		Then we have
		\begin{eqnarray*}
			&&\sum_{i,j}\langle r,e_{i}\otimes e^{*}_{j}\rangle=\sum_{i,j}\langle r^{\sharp}(e_{i}),e^{*}_{j} \rangle=\sum_{i,j}\langle \beta(e_{i}),e^{*}_{j} \rangle,\\
			&&\sum_{i,j}\langle r,e^{*}_{i}\otimes e_{j} \rangle=\sum_{i,j}\langle r^{\sharp}(e^{*}_{i}),e_{j} \rangle=\sum_{i,j}\langle -(\beta+\lambda\mathrm{id}_{A})^{*}e^{*}_{i},e_{j} \rangle=\sum_{i,j}\langle e^{*}_{i},-(\beta+\lambda )e_{j} \rangle.
		\end{eqnarray*}
		Hence \eqref{eq:cor:4.20} holds, and by Theorem \ref{thm:quadratic to fact},
		$(A\ltimes_{\mathcal{L}^{*}_{\circ_{A}},\mathcal{L}^{*}_{\star_{A}},\mathcal{R}^{*}_{\circ_{A}}} A^{*},\theta_{r},\delta_{r})$ is a factorizable perm symmetric bialgebra.
\end{proof}}

Synthesizing Corollary \ref{cor:2.24}, Propositions \ref{pro:quasi bialgebras}, \ref{pro:6.31}, \ref{pro:triangular Leib} and Theorems \ref{thm:commute}, \ref{thm:quadratic to fact}, we obtain the following commutative diagram which has already been shown in the Introduction.
{\footnotesize
 	\begin{equation*}
 		\begin{split}
 			\xymatrix{
 				\txt{triangular averaging\\ commutative  and\\ cocommutative  
 					infinitesimal\\ bialgebras}
 				\ar@{<=}[r]^-{{\rm Cor.}~\ref{cor:2.24}
 				}_-{\lambda=0}
 				\ar@{=>}[d]^-{{\rm Prop.}~\ref{pro:quasi bialgebras}
 				}
 				& \txt{symmetric averaging\\ Rota-Baxter Frobenius\\ commutative algebras}
 				\ar@<.4ex>@{<=>}[r]^-{{\rm Thm.}~\ref{thm:commute}
 				}_-{\lambda=-1}
 				\ar@{=>}[d]^-{{\rm Prop.}~\ref{pro:6.31}
 				}
 				& \txt{factorizable averaging\\ commutative  and\\ cocommutative  
 					infinitesimal\\ bialgebras
 				}
 				\ar@{=>}[d]^-{ {\rm Prop.}~\ref{pro:quasi bialgebras}
 				}
 				\\  
 				\txt{triangular special \\  apre-perm bialgebras}
 				\ar@{<=}[r]^-{{\rm Prop.}~\ref{pro:triangular Leib} 
 				}_-{\lambda=0}
 				&\txt{quadratic Rota-Baxter\\ special apre-perm algebras}	\ar@{<=>}[r]^-{{\rm Thm.}~\ref{thm:quadratic to fact} 
 				}_-{\lambda=-1}  
                & \txt{factorizable  special \\ apre-perm bialgebras}
 			}
 		\end{split}
 	\end{equation*}
 }

\begin{ex}
	Let $(A\ltimes_{\mathcal{L}^{*}_{\cdot_{A}}}A^{*},R,\mathcal{B}_{d})$ be the symmetric Rota-Baxter commutative Frobenius algebra of weight $-1$ and $P$ be the averaging operator which commutes with $R$ given in Example \ref{ex:3.25}.
	Then by Proposition \ref{pro:6.31}, there is a  quadratic Rota-Baxter \sapp $(D=A\oplus A^{*},\triangleright_{d},\triangleleft_{d},R,\mathcal{B})$ of weight $-1$ which is exactly given by \eqref{eq:ex:fact0} and \eqref{eq:ex:fact1}. Moreover by Theorem \ref{thm:quadratic to fact}, there is a factorizable \sappb $(A,\triangleright_{A},\triangleleft_{A},\vartheta_{r},\theta_{r})$ whose non-zero co-multiplications are given by \eqref{eq:ex:fact2}. That is, the factorizable \sappb $(A,\triangleright_{A},\triangleleft_{A},\vartheta_{r},\theta_{r})$ is exactly the one given in Example \ref{ex:6.29}.
\end{ex}

\delete{
	
	In this section, we introduce the notion of quadratic Rota-Baxter perm symmetric algebras. We show that a quadratic Rota-Baxter perm symmetric algebra of weight $0$ gives rise to a triangular perm symmetric bialgebra, and there is a one-to-one correspondence between quadratic Rota-Baxter perm symmetric algebras of nonzero weights and factorizable perm symmetric bialgebras.
	
	\begin{defi}
		A \textbf{quadratic Rota-Baxter perm symmetric algebra of weight $\lambda$} is a quintuple  $(A,\circ_{A},\star_{A},\mathcal{B},\beta)$, such that $(A,\circ_{A},\star_{A},\mathcal{B})$ is a quadratic perm symmetric algebra, $\beta$ is a Rota-Baxter operator of weight $\lambda$ on $(A,\circ_{A},\star_{A})$ and the following condition holds:
		\begin{equation}\label{strong2}
			\mathcal{B}\big(\beta (x),y\big)+\mathcal{B}\big(x,\beta (y)\big)+\lambda\mathcal{B}(x,y)=0,\;\;\forall x,y\in A.
		\end{equation}
	\end{defi}
	
	\begin{pro}\label{mir}
		Let $(A,\circ_{A},\star_{A},\mathcal{B})$ be a quadratic perm symmetric algebra and $\beta:A\rightarrow A$ be a linear map. Then $(A,\circ_{A},\star_{A},\mathcal{B},\beta)$ is a quadratic Rota-Baxter perm symmetric algebra of weight $\lambda$ if and only if $\big(A,\circ_{A},\star_{A},-\mathcal{B},-(\lambda\mathrm{id}+\beta)\big)$ is a quadratic Rota-Baxter perm symmetric algebra of weight $\lambda$.
	\end{pro}
	
	\begin{proof}
		It is well-known that $\beta$ is a Rota-Baxter operator of weight $\lambda$ if and only if $-(\lambda\mathrm{id}+\beta)$ is a Rota-Baxter operator of weight $\lambda$. For all $x,y\in A$, we have
		\begin{equation*}
			-\mathcal{B}\big(-(\lambda\mathrm{id}+\beta)x,y\big)-\mathcal{B}\big(x,-(\lambda\mathrm{id}+\beta)y\big)-\lambda\mathcal{B}(x,y)= \mathcal{B}\big(\beta(x),y\big)+\mathcal{B}\big(x,\beta(y)\big)+\lambda\mathcal{B}(x,y).
		\end{equation*}
		Thus $(\mathcal{B},\beta)$ satisfies \eqref{strong} if and only if $\big(-\mathcal{B},-(\lambda\mathrm{id}+\beta)\big)$ satisfies \eqref{strong}.
		Hence the conclusion follows.
	\end{proof}

	\begin{pro}\label{pro:4.2}
		Let $\beta$ be a Rota-Baxter operator of   weight $\lambda$ on a perm symmetric algebra $(A,\circ_{A},\star_{A})$.
		Then $$\big(A\ltimes_{\mathcal{L}^{*}_{\circ_{A}},\mathcal{L}^{*}_{\star_{A}},\mathcal{R}^{*}_{\circ_{A}}} A^{*},\mathcal{B}_{d},\beta-(\beta+\lambda\mathrm{id}_{A} )^{*}\big)$$ is a quadratic Rota-Baxter perm symmetric algebra of weight $\lambda$, where the bilinear form $\mathcal{B}_{d}$ on $A\oplus A^{*}$ is given by \eqref{eq:bfds}.
	\end{pro}
	\begin{proof}
		It follows from a straightforward computation.
	\end{proof}

	\begin{ex}\label{ex:6.13}
		Let $(A,\circ_{A},\star_{A})$ be a perm symmetric algebra.
		Then the identity map $\mathrm{id}_{A}$ is a Rota-Baxter operator of $(A,\circ_{A},\star_{A})$ of weight $-1$.
		By Proposition \ref{pro:4.2},
		\begin{equation*}
			\big(A\ltimes_{\mathcal{L}^{*}_{\circ_{A}},\mathcal{L}^{*}_{\star_{A}},\mathcal{R}^{*}_{\circ_{A}}} A^{*},\mathcal{B}_{d},\mathrm{id}_{A}-(\mathrm{id}_{A}-\mathrm{id}_{A})^{*}\big)= (A\ltimes_{\mathcal{L}^{*}_{\circ_{A}},\mathcal{L}^{*}_{\star_{A}},\mathcal{R}^{*}_{\circ_{A}}} A^{*},\mathcal{B}_{d},\mathrm{id}_{A})
		\end{equation*}
		is a quadratic Rota-Baxter perm symmetric algebra of weight $-1$.
	\end{ex}

	\begin{lem}\label{lem:bf}
		Let $A$ be a vector space and $\mathcal{B}$ be a nondegenerate symmetric bilinear form.
		Let $r\in A\otimes A$ and $\beta:A\rightarrow A$ satisfy \eqref{Br,Brt}.
		Then $r$satisfies
		\begin{equation}\label{rw}
			r+\tau(r)=-\lambda\phi_{\mathcal{B}},\; \lambda\in\mathbb{K}
		\end{equation}
		if and only if $\beta$  satisfies \eqref{strong}.
	\end{lem}
	\begin{proof}
		Let $x,y\in A$ and $a^{*}=\mathcal{B}^{\natural}(x), b^{*}=\mathcal{B}^{\natural}(y)$.
		Then
		\begin{eqnarray*}
			\mathcal{B}\big(\beta (x),y\big)&=& \langle \mathcal{B}^{\natural}(y),r^{\sharp}\mathcal{B}^{\natural}(x)\rangle=\langle b^{*}, r^{\sharp}(a^{*})\rangle=\langle r, a^{*}\otimes b^{*}\rangle,\\
			\mathcal{B}\big(x,\beta (y)\big)&=& \langle\mathcal{B}^{\natural}(x), r^{\sharp}\mathcal{B}^{\natural}(y)\rangle=\langle a^{*}, r^{\sharp}(b^{*})\rangle=\langle \tau(r), a^{*}\otimes b^{*}\rangle,\\
			\lambda\mathcal{B}(x,y)&=& \lambda \langle \mathcal{B}^{\natural^{-1}}\mathcal{B}^{\natural}(x),\mathcal{B}^{\natural}(y)\rangle=\lambda  \langle (\mathcal{B}^{\natural^{-1}})a^{*},b^{*}\rangle=\lambda\langle\phi_{\mathcal{B}}, a^{*}\otimes b^{*}\rangle.
		\end{eqnarray*}
		Hence $r$ satisfies  \eqref{rw} if and only if $\beta$ satisfies \eqref{strong}.
	\end{proof}
	
	\begin{pro}
		Let $(A,\circ_{A},\star_{A},\mathcal{B},\beta)$ be a quadratic Rota-Baxter perm symmetric algebra of weight $0$.
		Then there is a triangular perm symmetric bialgebra $(A,\circ_{A},\star_{A},\theta_{r},\delta_{r})$ with $\theta_r$ and $\delta_{r}$ defined by \eqref{eq:comul},
		where $r\in A\otimes A$ is given through the operator form $r^{\sharp}$ by
		\eqref{Br,Brt}, that is,
		\begin{equation}\label{eq:thm:quadratic to fact}
			r^{\sharp}(a^{*})=\beta\mathcal{B}^{{\natural}^{-1}}(a^{*}),\;\;\forall a^{*}\in A^{*}.
		\end{equation}
	\end{pro}
	\begin{proof}
		It follows from Proposition \ref{pro1-6} and Lemma \ref{lem:bf} by observing $r+\tau(r)=0$.
	\end{proof}

	Next we establish a one-to-one correspondence between quadratic Rota-Baxter perm symmetric algebras of nonzero weights and factorizable perm symmetric bialgebras.

	\begin{thm}\label{thm:quadratic to fact}
		Let $(A,\circ_{A},\star_{A},\mathcal{B},\beta)$ be a quadratic Rota-Baxter perm symmetric algebra of weight $-1$.
		Then there is a factorizable perm symmetric bialgebra $(A,\circ_{A},\star_{A},\theta_{r},\delta_{r})$ with $r$ given  through the operator form $r^{\sharp}$ by \eqref{eq:thm:quadratic to fact}.
		Conversely, let $(A,\circ_{A},\star_{A},\theta_{r},\delta_{r})$ be a factorizable perm symmetric bialgebra.
		Then there is a quadratic Rota-Baxter perm symmetric algebra $(A,\circ_{A},\star_{A},\mathcal{B},\beta)$ of weight $-1$ with $\beta$ given by \eqref{Br,Brt}
		\begin{equation}\label{eq:fact to quadratic0}
			\beta(x)=r^{\sharp}\big(r+\tau(r)\big)^{\sharp^{-1}}(x)
		\end{equation}
		and $\mathcal{B}$ given by
		\begin{equation}\label{eq:fact to quadratic}
			\mathcal{B}(x,y)=\langle {\big(r+\tau(r)\big)^{\sharp}}^{-1}(x),y\rangle,\; \forall x,y\in A.
		\end{equation}
	\end{thm}

	\begin{proof}
		Let $(A,\circ_{A},\star_{A},\mathcal{B},\beta)$ be a quadratic Rota-Baxter perm symmetric algebra of weight $-1$.
		Then by Lemma \ref{lem:bf} we have
		\begin{equation}\label{rw2}
			r+\tau(r)=\phi_{\mathcal{B}}.
		\end{equation}
		Since $\phi_{\mathcal{B}}\in A\otimes A$ is invariant and $\mathcal{B}^ \natural:A\rightarrow A^{*}$ is a linear isomorphism, we see that $r+\tau(r)$ is invariant and $\big(r+\tau(r)\big)^{\sharp}$ is a linear isomorphism. Moreover, since $\beta$ is a Rota-Baxter operator of weight $-1$, \eqref{eq:pro3,1} and \eqref{eq:pro3,2} hold in Proposition \ref{pro1-6} such that $SA(r)=0$. In conclusion, $r$ is a solution of the PSYBE which moreover satisfies that $r+\tau(r)$ is invariant and $\big(r+\tau(r)\big)^{\sharp}$ is a linear isomorphism. Hence   $(A,\circ_{A},\star_{A},\theta_{r},\delta_{r})$ is a factorizable perm symmetric bialgebra with $\theta_{r}$ and $\delta_{r}$ defined by \eqref{eq:comul}.
		is a linear isomorphism.
		
		Conversely,  let $(A,\circ_{A},\star_{A},\theta_{r},\delta_{r})$ be a factorizable perm symmetric bialgebra.
		We first observe that \eqref{eq:fact to quadratic} equivalently gives \eqref{rw2} and
		\begin{equation}\label{eq:113}
			\big(r+\tau(r)\big)^{\sharp}=\mathcal{B}^{{\natural}^{-1}}.
		\end{equation}
		Since $r+\tau(r)$ is invariant and $\big(r+\tau(r)\big)^{\sharp}$ is a linear isomorphism, it follows from Proposition  \ref{pro:2.24} that $\mathcal{B}$ given by \eqref{eq:fact to quadratic} contributes a quadratic perm symmetric algebra $(A,\circ_{A},\star_{A},\mathcal{B})$. Taking \eqref{eq:113} into \eqref{eq:pro3,1} and \eqref{eq:pro3,2} in Proposition \ref{pro1-6}, we see that $\beta$ given by \eqref{Br,Brt} is a Rota-Baxter operator of weight $-1$.
		Furthermore, \eqref{strong} holds for $\lambda=-1$ by Lemma \ref{lem:bf}. Therefore $(A,\circ_{A},\star_{A},\mathcal{B},\beta)$ is a quadratic Rota-Baxter perm symmetric algebra of weight $-1$.
	\end{proof}
	
	\delete{ Let $\mathcal{B}$ be a bilinear form on $A$ given by \eqref{rw}, that is,
		\begin{equation}\label{eq:fact to quadratic}
			\mathcal{B}(x,y)=-\lambda\langle {\big(r+\tau(r)\big)^{\sharp}}^{-1}(x),y\rangle,\;\lambda\neq0,\;\forall x,y\in A,
		\end{equation}
		and $\beta$ be a linear map given by \eqref{Br,Brt}. By \eqref{eq:2-tensor} and \eqref{eq:fact to quadratic}, we have $\mathcal{B}^{\natural}=-\lambda {\big(r+\tau(r)\big)^{\sharp}}^{-1}$, and thus $\big(r+\tau(r)\big)^{\sharp}=-\lambda \mathcal{B}^{{\natural}^{-1}}$. Since $r$ satisfies \eqref{eq:pro:co1}, \eqref{eq:invariance} and $\big(r+\tau(r)\big)^{\sharp}$ is nondegenerate, it follows that $\mathcal{B}$ given by \eqref{eq:fact to quadratic} is a nondegenerate symmetric invariant bilinear form on $(A,\circ_{A},\star_{A})$. Moreover, taking $\big(r+\tau(r)\big)^{\sharp}=-\lambda \mathcal{B}^{{\natural}^{-1}}$ into \eqref{eq:pro3,1} and \eqref{eq:pro3,2} in Proposition \ref{pro1-6}, we see that $\beta$ is a Rota-Baxter operator of weight $\lambda$. Again from $\big(r+\tau(r)\big)^{\sharp}=-\lambda \mathcal{B}^{{\natural}^{-1}}$, \eqref{rw} holds, and by Lemma \ref{lem:bf}, \eqref{eq:lem:bf2} holds. Thus $(A,\circ_{A},\star_{A},\mathcal{B},\beta)$ is a quadratic Rota-Baxter perm symmetric algebra of weight $\lambda$.}

	\begin{cor}\label{cor:4.20}
		Let $\beta$ be a Rota-Baxter operator of weight $\lambda\neq0$ on a perm symmetric algebra $(A,\circ_{A},\star_{A})$. Let $\{e_{1},\cdots,e_{n}\}$ be a basis of $A$ and $\{e^{*}_{1},\cdots,e^{*}_{n}\}$ be the dual basis.
		Then $(A\ltimes_{\mathcal{L}^{*}_{\circ_{A}},\mathcal{L}^{*}_{\star_{A}},\mathcal{R}^{*}_{\circ_{A}}} A^{*},\theta_{r},\delta_{r})$ is a factorizable perm symmetric bialgebra with $\theta_r$ and $\delta_{r}$ defined by \eqref{eq:comul}, where
		\begin{equation}\label{eq:cor:4.20}
			r=\sum_{i=1}^{n}-(\beta+\lambda)e_{i}\otimes e^{*}_{i}+e^{*}_{i}\otimes\beta(e_{i}).
		\end{equation}
	\end{cor}
	
	\begin{proof}
		By Proposition \ref{pro:4.2}, 	$\big(A\ltimes_{\mathcal{L}^{*}_{\circ_{A}},\mathcal{L}^{*}_{\star_{A}},\mathcal{R}^{*}_{\circ_{A}}} A^{*},\mathcal{B}_{d},\beta-( \beta+\lambda\mathrm{id}_{A} )^{*}\big)$ is a quadratic Rota-Baxter perm symmetric algebra of weight $\lambda$.
		For all $x\in A, a^{*}\in A^{*}$, we have
		\begin{equation*}
			\mathcal{B}_{d}^{\natural}(x+a^{*})=x+a^{*},
		\end{equation*}
		and following \eqref{eq:thm:quadratic to fact}, we obtain a linear transformation $r^{\sharp}$ on $A\oplus A^{*}$ by
		\begin{equation*}
			r^{\sharp}(x+a^{*})=\big(\beta-( \beta+\lambda\mathrm{id}_{A} )^{*}\big)\mathcal{B}_{d}^{{\natural}^{-1}}(x+a^{*})=\beta(x)-(\beta+\lambda\mathrm{id}_{A})^{*}a^{*}.
		\end{equation*}
		Then we have
		\begin{eqnarray*}
			&&\sum_{i,j}\langle r,e_{i}\otimes e^{*}_{j}\rangle=\sum_{i,j}\langle r^{\sharp}(e_{i}),e^{*}_{j} \rangle=\sum_{i,j}\langle \beta(e_{i}),e^{*}_{j} \rangle,\\
			&&\sum_{i,j}\langle r,e^{*}_{i}\otimes e_{j} \rangle=\sum_{i,j}\langle r^{\sharp}(e^{*}_{i}),e_{j} \rangle=\sum_{i,j}\langle -(\beta+\lambda\mathrm{id}_{A})^{*}e^{*}_{i},e_{j} \rangle=\sum_{i,j}\langle e^{*}_{i},-(\beta+\lambda )e_{j} \rangle.
		\end{eqnarray*}
		Hence \eqref{eq:cor:4.20} holds, and by Theorem \ref{thm:quadratic to fact},
		$(A\ltimes_{\mathcal{L}^{*}_{\circ_{A}},\mathcal{L}^{*}_{\star_{A}},\mathcal{R}^{*}_{\circ_{A}}} A^{*},\theta_{r},\delta_{r})$ is a factorizable perm symmetric bialgebra.
	\end{proof}

	\begin{ex}
		Let $(A,\cdot_{A})$ be a 2-dimensional commutative  algebra defined with respect to a basis $\{e_{1}$, $e_{2}\}$ by the following nonzero products:
		\begin{equation*}
			e_{1}\cdot_{A} e_{1}=e_{1},\; e_{1}\cdot_{A} e_{2}=e_{2}.
		\end{equation*}
		Then the perm symmetric algebra $A\ltimes _{\mathcal{L}^{*}_{\cdot_{A}},0,\mathcal{L}^{*}_{\cdot_{A}}}A^{*}$ given by \eqref{eq:11} and \eqref{eq:2.21} is defined by the following nonzero products:
		\begin{eqnarray*}
			&&e_{1}\circ e_{1}=e_{1},\; e_{1}\circ e_{2}=e_{2}\circ e_{1}=e_{2},\; e_{1}\circ e^{*}_{2}=e_{1}\star e^{*}_{2}=e^{*}_{2},\\
			&&e_{1}\circ e^{*}_{1}=e_{2}\circ e^{*}_{2}=e_{1}\star e^{*}_{1}=e_{2}\star e^{*}_{2}=e^{*}_{1}.
		\end{eqnarray*}
		Then by Example \ref{ex:6.13}, $(A\ltimes _{\mathcal{L}^{*}_{\cdot_{A}},0,\mathcal{L}^{*}_{\cdot_{A}}}A^{*},\mathcal{B}_{d},\mathrm{id}_{A})$ is a quadratic Rota-Baxter perm symmetric algebra of weight $-1$. By Corollary \ref{cor:4.20}, there is a factorizable perm symmetric bialgebra $(A\ltimes _{\mathcal{L}^{*}_{\cdot_{A}},0,\mathcal{L}^{*}_{\cdot_{A}}}A^{*},\theta_{r},\delta_{r})$, where
		\begin{equation*}
			r=e^{*}_{1}\otimes e_{1}+e^{*}_{2}\otimes e_{2}.
		\end{equation*}
		Note that this perm symmetric bialgebra is isomorphic to the one in Example \ref{ex:5.14} by identifying $e^{*}_{1}$ with $e_{3}$ and $e^{*}_{2}$ with $e_{4}$.
	\end{ex}

	\begin{pro}
		Let $(A,\circ_{A},\star_{A},\theta_{r},\delta_{r})$ be a factorizable perm symmetric bialgebra which corresponds to a  quadratic Rota-Baxter perm symmetric algebra $(A,\circ_{A},\star_{A},\mathcal{B},\beta)$ of weight $\lambda\neq 0$ via Theorems \ref{thm:quadratic to fact}. Then the factorizable perm symmetric bialgebra $(A,\circ_{A},\star_{A},\theta_{-\tau(r)}$,
		$\delta_{-\tau(r)})$ corresponds to the quadratic Rota-Baxter perm symmetric algebra $\big(A,\circ_{A},\star_{A},-\mathcal{B}$,
		$-(\lambda\mathrm{id}+\beta)\big)$ of weight $\lambda$. In conclusion, we have the following commutative diagram.
		
		\begin{equation*}
				\xymatrix@C=3cm@R=2.5cm{
					(A,\circ_{A},\star_{A},\theta_{r},\delta_{r}) \ar@{<->}[r]^-{ 
						{\rm Cor.}~\ref{cor:fact pair}}
					\ar@<-1ex>@{->}[d]_-{ 
						{\rm Thm.}~\ref{thm:quadratic to fact}}
					& (A,\circ_{A},\star_{A},\theta_{-\tau(r)},\delta_{-\tau(r)})
					\ar@<-1ex>@{->}[d]_-{{\rm Thm.}~\ref{thm:quadratic to fact}}\\
					(A,\circ_{A},\star_{A},\mathcal{B},\beta)
					\ar@{<->}[r]^-{{\rm Prop.}~\ref{mir}}
					\ar@<-1ex>[u]_-{{\rm Thm.}~\ref{thm:quadratic to fact}}
					&
					\big(A,\circ_{A},\star_{A},-\mathcal{B},-(\lambda\mathrm{id}+\beta)\big)
					\ar@<-1ex>[u]_-{{\rm Thm.}~\ref{thm:quadratic to fact}}
				}
		\end{equation*}
	\end{pro}
	\begin{proof}
		By Theorem \ref{thm:quadratic to fact},  $(A,\circ_{A},\star_{A},\theta_{-\tau(r)},\delta_{-\tau(r)})$ gives rise to a quadratic Rota-Baxter perm symmetric algebra $(A,\circ_{A},\star_{A},\mathcal{B}',\beta')$ of weight $\lambda$, where
		\begin{equation}\label{omega'}
			\mathcal{B}'(x,y)\overset{\eqref{eq:fact to quadratic}}{=}\lambda\langle {\big(r+\tau(r)\big)^{\sharp}}^{-1}(x),y\rangle=-\mathcal{B}(x,y),
		\end{equation}
		and
		\begin{small}
			\begin{eqnarray*}
				\beta'(x)&\overset{\eqref{Br,Brt}}{=}&
				\big(-\tau(r)\big)^{\sharp}\mathcal{B}'^{\natural}(x)
				\overset{\eqref{omega'}}{=}-\big(-\tau(r)\big)^{\sharp}\mathcal{B}^{\natural}(x)
				=-\lambda{\tau(r)}^{\sharp}{\big(r+\tau(r)\big)^{\sharp}}^{-1}(x)\\
				&=&\lambda \Big(r^{\sharp}-\big(r+\tau(r)\big)^{\sharp}\Big){\big(r+\tau(r)\big)^{\sharp}}^{-1}(x)=-\lambda x+\lambda r^{\sharp}{\big(r+\tau(r)\big)^{\sharp}}^{-1}(x)\\
				&=&-\lambda x-r^{\sharp}\mathcal{B}^{\natural}(x)\overset{\eqref{Br,Brt}}{=}-(\lambda\mathrm{id}+\beta)x,
			\end{eqnarray*}
		\end{small}for all $x,y\in A$. Hence
		\begin{equation*}
			(A,\circ_{A},\star_{A},\mathcal{B}',\beta')=\big(A,\circ_{A},\star_{A},-\mathcal{B},-(\lambda\mathrm{id}+\beta)\big).
		\end{equation*} Similarly, we obtain the converse side that the quadratic Rota-Baxter perm symmetric algebra $\big(A,\circ_{A},\star_{A},-\mathcal{B},-(\lambda\mathrm{id}+\beta)\big)$
		of weight $\lambda$ gives rise to the factorizable perm symmetric bialgebra $(A,\circ_{A},\star_{A},\theta_{-\tau(r)},\delta_{-\tau(r)})$ via Theorem \ref{thm:quadratic to fact}.
\end{proof}}

\noindent{\bf Acknowledgments.} This work is supported by NSFC
(12401031, W2412041),
the Postdoctoral Fellowship Program of
CPSF (GZC20240755, 2024T005TJ, 2024M761507) and Nankai Zhide Foundation.

\noindent{\bf Declaration of interests.} 
 The authors have no conflicts of interest to disclose.
 
 \noindent{\bf Data availability.} 
 Data sharing is not applicable to this article as no new data were created or analyzed.

\end{document}